\documentclass[preprint,12pt]{elsarticle}
\usepackage{amssymb}
\usepackage{amsmath}
\usepackage{amsthm}
\usepackage{appendix,bm,graphicx,hyperref,mathrsfs,cases,subcaption}
\usepackage{color, xcolor}
\usepackage{float}
\usepackage[T1]{fontenc}

\def\3bar{{|\hspace{-.01in}\|}}

\usepackage[T1]{fontenc}

\journal{B}	

\numberwithin{equation}{section}
\newtheorem{theorem}{Theorem}[section]

\newtheorem{lemma}[theorem]{Lemma}

\newtheorem{remark}{Remark}[section]
\newtheorem{assumption}{Assumption}[section]

\makeatletter
\newcommand{\Rmnum}[1]{\expandafter\@slowromancap\romannumeral #1@}
\makeatother

\begin{document}
	\sloppy
	\begin{frontmatter}
		
		
		
		\title{Mass conservation, positivity and  energy identical-relation preserving scheme for the Navier-Stokes equations with variable density\tnoteref{t1}}
		
		\tnotetext[t1]{This work is partially supported by the Natural Science Foundation of Chongqing
			(No. CSTB2024NSCQ-MSX0221) and  the Natural Science Foundation of China (Nos. 12271082, 62231016).}
		
		\author[label1]{Fan Yang}
		\author[label1]{Haiyun Dong}
		\address[label1]{College of Mathematics and Statistics, Chongqing University, Chongqing 401331, P.R. China}
		\author[label2]{Maojun Li}
		\address[label2]{School of Mathematical Sciences, University of Electronic Science and Technology of China, Sichuan, 611731, PR China}
		\author[label1]{Kun Wang\corref{cor1}}
	\cortext[cor1]{Corresponding author.}
	\begin{abstract}
		In this paper, we consider a mass conservation, positivity and  energy identical-relation preserving scheme for the Navier-Stokes equations with variable density. Utilizing  the square transformation, we first ensure the positivity of the numerical fluid density, which is form-invariant and regardless of the discrete scheme. Then, by proposing a new recovery technique to eliminate the numerical dissipation of the energy and to balance the loss of the mass when approximating the reformation form, we preserve the original energy identical-relation and mass conservation of the proposed scheme. To the best of our knowledge, this is the first work that can preserve the   original energy identical-relation for the Navier-Stokes equations with variable density. Moreover, the error estimates of the considered scheme are derived. Finally, we show some numerical examples to verify the correctness and efficiency.
	\end{abstract}
	
	\begin{keyword}
		Navier-Stokes equations with variable denstity \sep positivity preserving \sep mass conservation \sep energy identical-relation preserving \sep  error estimate
		
		\MSC[2020] 65N30\sep 76D05\sep 76M05   
		
	\end{keyword}
	\journal{XXX}
\end{frontmatter}

	
	
	\section{Introduction}
	\label{sec1}
In this paper, we   focus on the  incompressible Navier-Stokes equations with variable density
\begin{align}
	\rho_t+\nabla\cdot(\rho u)&=0, &\mathrm{in}~\Omega\times(0,T],\label{1.1}\\
	\rho u_t-\mu\Delta  u+\rho(u\cdot\nabla)u+\nabla p&=f, &\mathrm{in}~ \Omega\times(0,T],\label{1.2}\\
	\nabla\cdot u&=0, &\mathrm{in}~\Omega\times(0,T],\label{1.3}
\end{align}
where $\Omega\subset\mathcal{R}^2$ is a convex polygonal domain with a sufficiently smooth boundary $\partial\Omega$, $\rho=\rho(\mathbf{x},t)=\rho(x,y,t)$ represents the density of the fluid, $u=u(\mathbf{x},t)=(u_1(\mathbf{x},t),u_2(\mathbf{x},t))^\top$ represents the velocity of the fluid, $\mu$ denotes the viscosity coefficient,  $f = (f_1(\mathbf{x}, t),  f_2(\mathbf{x},  t))^\top$ is a given body force. Moreover, we give the following initial conditions and boundary conditions:
\begin{flalign}
	\left\{\begin{array}{l}
		\rho(\mathbf{x}, 0)=\rho_0(\mathbf{x}) ,  \\
		u(\mathbf{x}, 0)=u_0(\mathbf{x}), \\
	\end{array}\right.
	\left\{\begin{array}{l}
		\rho(\mathbf{x}, t)|_{\Gamma_{in}}=a(\mathbf{x}, t),   \\
		u(\mathbf{x}, t)|_{\partial\Omega}=g(\mathbf{x}, t), \\
	\end{array}\right.\nonumber
\end{flalign}
$\rho_0(\mathbf{x})$, $a(\mathbf{x}, t)$, $u_0(\mathbf{x})=(u_{10}(\mathbf{x}),u_{20}(\mathbf{x}))^\top$ and $g(\mathbf{x}, t)=(g_1(\mathbf{x}, t),g_1(\mathbf{x}, t)^\top$ are given functions, $\Gamma_{in}=\{\mathbf{x}\in\partial \Omega:g\cdot\vec{\nu}<0\}$ is the inflow boundary with $\vec{\nu}$ being the outward normal vector, and the initial density $\rho_0(\mathbf{x})$ satisfy the following conditions \cite{bib17}
\begin{equation}
	0 <\rho_0^{min}\leq \rho(t, \mathbf{x}) \leq \rho_0^{max}  ~~~\mathrm{in}~ \Omega.\label{1.4}
\end{equation}
For simplicity, we consider that $g(\mathbf{x}, t)=0$ and assume that the boundary $\partial\Omega$ is impervious, which means $g\cdot\vec{\nu}=0$ on $\partial\Omega$ and $\Gamma_{in}=\emptyset$ in this paper. Navier-Stokes equations with variable density \eqref{1.1}-\eqref{1.3} are a hyperbolic-parabolic coupled nonlinear system, which  plays an important role in  fluid mechanics.

For the existence and uniqueness of the solutions of  Navier-Stokes equations with variable density \eqref{1.1}-\eqref{1.3}, the reader is referred to, e.g., \cite{bib01,plusbib37,bib02,bib03}. On the other hand, there have been lots of attentions in developing efficient numerical methods for \eqref{1.1}-\eqref{1.3}, especially in the schemes preserving physical properties.
In 1992, Bell et al. \cite{bib05} first introduced the projection method for variable density issues, they employed the Crank-Nicolson method for temporal discretization, and utilized a standard difference method for spatial discretization. Subsequently, Almgren et al. \cite{bib04} and Puckett et al. \cite{bib06} investigated the conservative adaptive projection method and the higher-order projection method for tracking fluid interfaces, respectively. Unlike other traditional algorithms, this method reduces computational costs by solving the discrete pressure variable through the incorporation of a Poisson equation. In \cite{bib10}, a novel time-stepping method was introduced which had been verified by some numerical examples. Additionally, Li et al. in \cite{bib11} proposed a second-order mixed stabilized finite element   method for solving Navier-Stokes equations with variable density. Furthermore, Liu and Walkington \cite{bib12} conducted an investigation into the discontinuous Galerkin (DG) method for solving Navier-Stokes equations with variable density. They proved the convergence of the scheme but did not provide any convergence rates. In contrast, Pyo and Shen \cite{bib15} studied two Gauge-Uzawa schemes and demonstrated that the first-order temporally discretized Gauge-Uzawa schemes possess unconditional stability. Moreover, Li et al. \cite{plusbib38} presented a filtered time-stepping technique \cite{plusbib39}, which could improve the time accuracy to second-order. Afterwards, Reuter et al. \cite{plusbib41} introduced a novel algorithm of explicit temporal discretization for low-Mach Navier-Stokes equations with variable density, which achieved second-order accuracy in time. By constructing an implicit temporal scheme with the Taylor series and using a finite element with  standard high-order Lagrange basis functions, Lundgren et al. \cite{plusbib42} considered a fourth-order method for \eqref{1.1}-\eqref{1.3}.

When designing numerical schemes, one of interesting and challenging topics is to preserve the physical properties of the continuous model in the discrete scheme, which has attracted lots of attentions in the past decade. For the Navier-Stokes equations with constant density, by transforming into an equivalent form known as the energy, momentum and angular momentum conserving (EMAC) formulation in \cite{bib18}, a mixed finite element method are proposed, which imposed the incompressible condition weakly and preserved physical properties such as momentum, energy, and enstrophy. This research was further extended to address long-term approximations in \cite{bib19} and three-dimensional problems in \cite{bib20}. Concurrently, a mimetic spectral element method was introduced in \cite{bib21}, that is capable of preserving mass, energy, enstrophy, and vorticity. Additionally, this concept was adapted to problems involving moving domains in \cite{bib22}.  Lately, by deriving the viscosity coefficients through a residual-based shock-capturing approach, Lundgren et al. \cite{plusbib40} presented a novel symmetric and  tensor-based viscosity method, which can ensure the  conservation of angular momentum and the dissipation of kinetic energy. For the variable density incompressible flows, an entropy-stable scheme was explored in \cite{bib23} by combining the discontinuous Galerkin method with an artificial compressible approximation. Recognizing the significance of density bounds in numerical simulations, a bound-preserving discontinuous Galerkin method was introduced in \cite{bib24}. Furthermore, Desmons et al. \cite{bib25} introduced a generalized high-order momentum preserving scheme, which was claimed to be easy for implementation with the finite volume method. To ensure the positivity preserving of the density, a square transformation $\rho=\sigma^2$ was introduced in \cite{bib26,bib15,WLZ2024}. By introducing power-type and exponential-type scalar auxiliary variables to define the system's energy and to balance the incompressible condition's influence respectively, Zhang et al. \cite{bib28} reformulated the Navier-Stokes equations with variable density into an equivalent form and subsequently developed a linear, decoupled, and fully discrete finite element scheme. This scheme preserves the mass, momentum, and modified energy conservation relations. Recently, by introducing a formulation with consistent nonlinear terms, the schemes with the numerical density invariant to global shifts was studied in \cite{LN2024}. And the authors in \cite{LSTZ2024} investigate  schemes which could preserve the lower bound of the numerical density  and energy inequality under the gravitational force.

But, due to the complex nonlinearities and coupling terms, it is challenging to derive error analysis for numerical methods solving the Navier-Stokes equations with variable density. Under the assumptions that the numerical density is bound and can achieves first order convergence, the author  in \cite{GS2011} presented a first-order splitting scheme and deduced its error estimates. Recently, giving up the assumption on the numerical density, Cai et al. \cite{bib16} derived the error estimate of the backward Euler method applied to the 2D Navier-Stokes equations with variable density, leveraging an error splitting technique and discrete maximal $L^p$-regularity. Drawing upon this research, Li and An in \cite{bib17} presented a novel BDF2 finite element scheme, by utilizing the Mini element space to approximate both the velocity and the pressure, and employing the quadratic conforming finite element space to approximate the density. Leveraging a post-processed technique, the authors in \cite{bib13}  demonstrated the  convergence order of $O(\tau^2 + h^2)$ in $L^2$-norm for the numerical density $\rho_h^n$ and numerical velocity $u_h^n$. Lately, by rewriting  the original system, Pan and Cai in \cite{PC2024} proposed a general BDF2 finite element method preserving the energy inequality and deduced its error analysis.  But, there is no literature on error estimates for the fully discrete first-order scheme for solving Navier-Stokes equations with variable density, which can preserve the mass conservation, the positivity of the numerical density and the original energy identical-relation of the system.

In this paper, we will consider a  mass conservation, positivity and  energy identical-relation preserving scheme for the Navier-Stokes equations with variable density \eqref{1.1}-\eqref{1.3}.  To ensure the positivity of the numerical density, we utilize the square transformation considered in \cite{bib26,WLZ2024} to transform the density sub-equation. Compared to other positivity preserving methods, the method considered here has two mainly advantages: form-invariant and irrelevance of the discrete scheme. Therefore, it is possible to directly adopt other schemes in the references for solving the density sub-equation. But, the mass conservation is lost when approximating this reformation form. To overcome this problem, then we use the recovery technique in \cite{HS2021,bib27} to preserve the discrete system's mass. In addition, through constructing a new recovery method, we eliminate successfully the numerical energy dissipation usually existent in the numerical scheme. Moreover, we prove that the scheme considered in this paper not only can inherit the mass conservation, positivity, original energy identical-relation from the continuous equations, but also achieve the following convergence order in the $L^2$-norm
\begin{align*}
	\|\rho(\mathbf{x}, t_n)-\rho_h^n\|_{L^2}^2+\|u(\mathbf{x}, t_n)-u_h^n\|_{L^2}^2\leq C(\tau^2+h^4),
\end{align*}
where $C$ is a general positive constant, $h$ and $\tau$ are the spatial mesh size and the temporal step, respectively.

The rest of this paper is organized as follows. In Section \ref{sec.2}, we introduce some preliminaries, such as functional spaces, some inequalities commonly used, and an equivalent model with some essential properties. Then, based on this equivalent form, we propose a fully discrete first order recovery finite element scheme in Section \ref{sec.3}, that keeps density positivity, mass conservation, and energy identical-relation preserving. Subsequently, in Section \ref{sec.4}, we derive the error estimates of the proposed scheme. Furthermore, in Section \ref{sec.5}, we present some examples to confirm the convergence orders and efficiency of the recovery finite element scheme. Finally, a conclusion remark is made in  Section \ref{sec.6}.

\section{Preliminaries}\label{sec.2}
In this section, after introducing some functional spaces in the first subsection, we will recall some frequently used inequalities and present some essential properties for the Navier-Stokes equations with variable density in Subsections 2.2 and 2.3, respectively.

\subsection{Functional spaces}
For $k\in N^+$ and $1\leq p\leq +\infty$, we denote $L^p(\Omega)$ and $W^{k, p}(\Omega)$ as the classical Lebesgue space and Sobolev space, respectively. The norms of these spaces are denoted by
\begin{align*}
	||u||_{L^p(\Omega)}&=\left(\int_{\Omega}|u(\mathbf{x})|^p\mathrm{d} \mathbf{x}\right)^\frac{1}{p},\\
	||u||_{W^{k,p}(\Omega)}&=\left(\sum\limits_{|j|\leq k}||D^ju||_{L^p(\Omega)}^p\right)^\frac{1}{p}.
\end{align*}
Within this context, $W^{k, 2}(\Omega)$ is also known as the Hilbert space and can be expressed as $H^k(\Omega)$.  $||\cdot||_{L^\infty}$ represents the norm of the space  $L^\infty(\Omega)$ which is defined as
\begin{equation*}
	||u||_{L^\infty(\Omega)}=ess\sup\limits_{\mathbf{x}\in \Omega}|u(\mathbf{x})|,
\end{equation*}
and $(\cdot, \cdot)$ denotes the inner product in $L^2(\Omega)$. Furthermore, we define the following frequently utilized mathematical frameworks:
$$W=H^1(\Omega), \quad V=(H_0^1(\Omega))^2,    \quad V_0=\{v\in V,  \nabla\cdot v=0\},$$
$$M=L_0^2(\Omega)=\{q\in L^2(\Omega), \int_{\Omega}q d\mathbf{x}=0\}.$$

On the other hand, let $\mathcal{T}_h = \{K\}$ be a uniformly regular triangulation partition of $\Omega$ with a mesh size $h (0 <h< 1)$. We also define  the finite element spaces
$$V_h=\{u_h\in C(\bar{\Omega})^2\cap V, ~v_h|_K\in P_2(K)^2, ~\forall K\in\mathcal{T}_h\}\subset V,$$
$$M_h=\{p_h\in C(\bar{\Omega})\cap H^1(\Omega), ~q_h|_K\in P_1(K), ~\forall K\in \mathcal{T}_h, ~\int_{\Omega}q_hd\mathbf{x}=0\}\subset M,$$
$$W_h=\{\rho_h\in  C(\bar{\Omega})\cap W, ~r_h|_K\in P_2(K), ~\forall K\in \mathcal{T}_h\}\subset W,$$
where $P_m(K)$ denotes the polynomial space with degree up to $m$ on every triangle $K\in \mathcal{T}_h$. {Obviously, There exists a positive constant $\beta_h> 0$ such that the so-called inf-sup inequality holds (see, e.g, \cite{bib30}): for each $q_h\in M_h$, there exists $v_h\in V_h, v_h\neq0$, such that
	\begin{align}\label{insu}
		\beta_h||q_h||_{L^2}\leq\sup\limits_{v_h\in V_hv_h\neq0}\frac{(\nabla\cdot v_h,q_h)}{||\nabla v_h||_{L^2}}.
\end{align}}
\subsection{Some inequalities}
We recall some useful inequalities in two dimension in this subsection. For any $v_h$ belongs to the finite element spaces defined above, there hold

1. Inverse inequality \cite{bib31}:
\begin{align}
	||v_h||_{L^p}&\leq Ch^{\frac{2}{p}-\frac{2}{q}}||v_h||_{L^q}, \label{2.1}\\
	||v_h||_{L^\infty}&\leq Ch^{-1}||v_h||_{L^2},\qquad ||\nabla v_h||_{L^\infty}&\leq Ch^{-1}||\nabla v_h||_{L^2} \label{2.2}\\
	||v_h||_{H^1}&\leq Ch^{-1}||v_h||_{L^2};\label{2.3}
\end{align}

2. Agmon's inequality \cite{bib32}:
\begin{equation}
	||v_h||_{L^\infty}\leq C||v_h||_{L^2}^{\frac{1}{2}}||\Delta v_h||_{L^2}^{\frac{1}{2}}.\label{2.4}
\end{equation}

The famous Gronwall lemma which is frequently used for the time dependent problem is as follows:

\begin{lemma}\label{lem.7}(Gronwall inequality \cite{bib17})
	Let  $B>0$ and $a_k, b_k, c_k$ be non-negative numbers such that
	\begin{equation}
		a_n+\tau\sum\limits_{k=0}^{n}b_k\leq\tau\sum\limits_{k=0}^{n}c_k a_k+B,   ~n\geq 0.\label{2.5}
	\end{equation}
	If $\tau c_k< 1$ and  $d_k=(1-\tau c_k)^{-1}$, then there holds
	\begin{equation}
		a_n+\tau\sum\limits_{k=0}^{n}b_k\leq \exp\left(\tau\sum\limits_{k=0}^{n}c_kd_k\right)B,  ~n\geq 0.\label{2.6}
	\end{equation}
\end{lemma}

Moreover, recalling the $L^2$ projection operator $\Pi_h$ \cite{bib17}: $W\rightarrow W_h$
\begin{equation}
	(\Pi_h\sigma-\sigma, r_h)=0, ~\forall \sigma\in W, ~r_h\in W_h, \label{3.13}
\end{equation}
and the Stokes projection $(R_h,Q_h): V\times M\rightarrow V_h\times M_h$
\begin{align}
	(\nabla(R_h u-u), \nabla v_h)-(\nabla\cdot v_h, Q_hp-p)&=0,  \label{3.14}\\
	(\nabla\cdot(R_hu-u), q_h)&=0, \label{3.15}
\end{align}
for $\forall u\in V, p\in M$ and $\forall v_h\in V_h, q_h\in M_h$, we have \cite{bib17,bib30}
\begin{align}
	||u-R_hu||_{L^2}&+h||\nabla(u-R_hu)||_{L^2}+h||p^n-Q_hp||_{L^2}\nonumber\\
	&\leq Ch^3(||u||_{H^3}+||p||_{H^2}), \label{3.16}\\
	||\sigma\!-\!\Pi_h\sigma||\!_{L^2}\!&+\!||\rho\!-\!\Pi_h\rho||\!_{L^2}\!+\!h(||\sigma\!-\!\Pi_h\sigma||\!_{H^1}\!+\!||\rho\!-\!\Pi_h\rho||\!_{H^1})\!\nonumber\\
	&\leq\! Ch^3(||\sigma||_{H^3}+||\rho||_{H^2}).\label{3.17}
\end{align}

\subsection{Some essential properties}
For the Navier-Stokes equations with variable density  \eqref{1.1}-\eqref{1.3}, there hold the following essential properties (see, i.e.,  \cite{bib26,bib17,bib15,bib28}):

1. Positivity:
\begin{equation}
	\rho(\mathbf{x},t)>0.\notag
\end{equation}

2. Mass conservation:
\begin{equation}
	\int_\Omega\rho(\mathbf{x},t)\mathrm{d}\mathbf{x}=\int_\Omega\rho(\mathbf{x},0)\mathrm{d}\mathbf{x}.\nonumber
\end{equation}

3. Energy identical-relation:
\begin{equation}
	\frac{dE(\rho, u)}{dt}=-\mu\int_{\Omega}|\nabla u|^2\mathrm{d}\mathbf{x}+\int_{\Omega}fu\mathrm{d}\mathbf{x},\nonumber
\end{equation}
where the energy $E$ is defined by
\begin{equation}
	E=\frac{1}{2}\int_{\Omega}\rho |u|^2\mathrm{d}\mathbf{x}.\nonumber
\end{equation}

When designing numerical schemes for solving the Navier-Stokes equations with variable density \eqref{1.1}-\eqref{1.3}, it is important to ensure them to preserve the above  properties, which will improve the computational accuracy.

To preserve the positivity, we adopt the square transformation \cite{bib26,bib15,WLZ2024}
\begin{align}
	\rho(\mathbf{x}, t)=(\sigma(\mathbf{x}, t))^2,\label{de}
\end{align}
which  guarantees that the density is non-negative  regardless of the discrete scheme. Moreover, to derive the energy relation of the considered scheme, we adopt an equivalent formulation of the  momentum equation \eqref{1.2}  (see, i.e., \cite{bib15,WLZ2024}), which combining with \eqref{de} and \eqref{1.3} yields
\begin{align}
	\sigma_t+\nabla\cdot(\sigma u)&=0, &\mathrm{in}~ \Omega\times(0,T],\label{2.10}\\
	\sigma(\sigma u)_t-\mu\Delta  u+\rho(u\cdot\nabla)u+\frac{u}{2}\nabla\cdot(\rho u)+\nabla p&=f, &\mathrm{in}~ \Omega\times(0,T],\label{2.11}\\
	\nabla\cdot u&=0,&\mathrm{in}~ \Omega\times(0,T].\label{2.12}
\end{align}
We can see that the equation \eqref{1.1} is form-invariant for this transformation, and the initial data satisfies
\begin{equation}
	\sigma_0(\mathbf{x})=\sqrt{\rho_0(\mathbf{x})}>0\quad \mathrm{and} \quad 0<\sqrt{\rho_0^{min}}\leq \sigma(t, \mathbf{x})\leq \sqrt{\rho_0^{max}},   \quad~\mathrm{in}~\Omega,\label{tsig}
\end{equation}
by  cooperating with (\ref{1.4}) and the positivity of the density.

Furthermore, to derive the error estimate in the subsequent sections, we make the following assumptions on the solutions of the continuous model.
\begin{assumption}\label{as.2.2}
	The solutions of \eqref{2.10}-\eqref{2.12} satisfy the following regularities \cite{bib26,bib17}:
	\begin{align*}
		\sigma&\in C([0,T];H^3(\Omega)),~~~~\sigma_t\in L^{\infty}([0,T];H^1(\Omega))\cap L^{2}([0,T];H^2(\Omega)),\\
		\rho&\in C([0,T];H^3(\Omega))\cap C^1([0,T];H^2(\Omega)),\\
		u&\in C([0,T];H^3(\Omega)^2)\cap C^1([0,T];H^2(\Omega)^2),~~~~p\in C([0,T];H^2(\Omega)).
	\end{align*}
\end{assumption}
\section{Property-preserving scheme}\label{sec.3}
In this section,  we will propose a property-preserving fully discrete first order finite element method for solving the incompressible Navier-Stokes equations \eqref{2.10}-\eqref{2.12} with variable density. Although the positivity of the density is preserved by using  the square transformation \eqref{de}, the mass conservation will be lost when approximating this reformation form. Adopting the recovery technique in \cite{HS2021,bib27}, we recover  the discrete system's mass. In addition, via constructing a new recovery method, we also eliminate the numerical energy dissipation which is usually existent in the classical scheme, which ensures the energy identical-relation of the proposed scheme.

Let $N\in\mathbb{N}^+$ and $\tau=T/N(0<\tau<1)$, thus $0=t_0<t_1<\dots <t_k<t_{k+1}\dots <t_N=T$. Define $D_{\tau}g^{n+1}:=\frac{g^{n+1}-g^n}{\tau}$, then  the first order scheme for the equations \eqref{2.10}-\eqref{2.12} considered in this paper is as follows: Given $(\sigma_h^0, \rho_h^0, u_h^0,\tilde{u}_h^{0})=(\Pi_h\sigma_0, \Pi_h\rho_0, R_hu_0, R_hu_0)$, find $(\sigma_h^{n+1}, \tilde{u}_h^{n+1}, p_h^{n+1}, u_h^{n+1}, \rho_h^{n+1})\\
\in(W_h, V_h, M_h, V_h, W_h)$ for $0\leq n\leq N-1$ through the following steps:\\

\textbf{Step 1.} Find $\sigma_h^{n+1}\in W_h$ such that
\begin{equation}
	(D_{\tau}\sigma_h^{n+1}, r_h)+(\nabla \sigma_h^{n+1}\cdot u_h^n, r_h)+\frac{1}{2}(\sigma_h^{n+1}\nabla\cdot u_h^n,r_h)=0, \quad\forall r_h\in W_h; \label{3.1}
\end{equation}

\textbf{Step 2.} Find $(\tilde{u}_h^{n+1}, p_h^{n+1})\in (V_h, M_h)$ such that
\begin{equation}
	\begin{aligned}
		&(\sigma_h^{n+1}D_{\tau}(\sigma_h^{n+1}\tilde{u}_h^{n+1}), v_h)+\mu(\nabla \tilde{u}_h^{n+1}, \nabla v_h)+(\rho_h^{n}(u_h^n\cdot\nabla)\tilde{u}_h^{n+1}, v_h)\\
		&+\frac{1}{2}(\tilde{u}_h^{n+1}\nabla\cdot(\rho_h^{n}u_h^n), v_h)-(p_h^{n+1}, \nabla\cdot v_h)+(\nabla\cdot \tilde{u}_h^{n+1}, q_h)\\
		&=(f^{n+1}, v_h),\quad \forall (v_h, q_h)\in (V_h, M_h);\label{3.2}
	\end{aligned}
\end{equation}

\textbf{Step 3.} Find  $u_h^{n+1}\in V_h$ by
\begin{equation}
	u_h^{n+1}=\sqrt{\gamma_h^{n+1}}\tilde{u}_h^{n+1},\label{3.4}
\end{equation}

where
\begin{equation}
	\begin{small}
		\gamma_h^{n+1}\!=\!\left\{\begin{aligned}&1\!+\!\frac{||\sigma_h^{n+1}\tilde{u}_h^{n+1}\!-\!\sigma_h^n\tilde{u}_h^n||_{L^2}^2\!-\!||\sigma_h^n\tilde{u}_h^n||_{L^2}^2\!+\!||\sigma_h^nu_h^n||_{L^2}^2}{||\sigma_h^{n+1}\tilde{u}_h^{n+1}||_{L^2}^2}, &||\sigma_h^{n+1}\tilde{u}_h^{n+1}||_{L^2\!}\!\neq\!0\\
			&1,&||\sigma_h^{n+1}\tilde{u}_h^{n+1}||_{L^2}\!=\!0\end{aligned};\right. \label{3.3}
	\end{small}
\end{equation}

\textbf{Step 4.} Find $\rho_h^{n+1}\in W_h$ by
\begin{equation}
	(\rho_h^{n+1},r_h)=(\lambda_h^{n+1}\bar{\rho}_h^{n+1},r_h),\quad \forall r_h\in W_h, \label{3.7}\\
\end{equation}

where
\begin{align}
	\bar{\rho}_h^{n+1}&=(\sigma_h^{n+1})^2,\label{3.5}\\
	\lambda_h^{n+1}&=\frac{\int_{\Omega}\rho_h^n \mathrm{d}\mathbf{x}}{\int_{\Omega}\bar{\rho}_h^{n+1}\mathrm{d}\mathbf{x}}.\label{3.6}
\end{align}

In Steps 1-2, we get the approximation solutions $\sigma_h^{n+1},~\tilde{u}_h^{n+1}$  and $ p_h^{n+1}$ by solving two linear system. But, the mass conservation  and the original energy identical-relation is lost in Steps 1 and 2, respectively. To make the scheme to satisfy the properties of the continuous equations, we recover them in Steps 3-4, which are made up of several assignment operations and can be implemented efficiently.  For the scheme \eqref{3.1}-\eqref{3.6}, there holds the following Theorem.

\begin{theorem}\label{theorem3.1}
	The scheme \eqref{3.1}-\eqref{3.6} inherits the following physical properties of the continuous equations  \eqref{1.1}-\eqref{1.3} for $0\leq n\leq N-1$:\\
	
	1. Positivity: $\rho_h^{n+1}(\mathbf{x}){\geq}0$.\\
	
	2. Mass conservation: $\int_{\Omega}\rho_h^{n+1}\mathrm{d}\mathbf{x}=\int_{\Omega}\rho_h^0\mathrm{d}\mathbf{x}$.\\
	
	3. Energy identical-relation:
	\begin{equation}
		\begin{aligned}
			\left\{\begin{aligned}&D_{\tau}E_h^{n+1}=-\mu\int_{\Omega}|\nabla \tilde{u}_h^{n+1}|^2\mathrm{d}\mathbf{x}+\int_{\Omega}f^{n+1}\tilde{u}_h^{n+1}\mathrm{d}\mathbf{x},\quad ||\sigma_h^{n+1}\tilde{u}_h^{n+1}||_{L^2}\neq0\nonumber\\
				&{\tilde{D}_{\tau}\tilde{E}_h^{n+1}=-\mu\int_{\Omega}|\nabla \tilde{u}_h^{n+1}|^2\mathrm{d}\mathbf{x}+\int_{\Omega}f^{n+1}\tilde{u}_h^{n+1}\mathrm{d}\mathbf{x}, \quad ||\sigma_h^{n+1}\tilde{u}_h^{n+1}||_{L^2}=0} \end{aligned};\right.
		\end{aligned}
	\end{equation}	
	where the energy $D_{\tau}E_h^{n+1}=\frac{||\sigma_h^{n+1}u_h^{n+1}||_{L^2}^2-||\sigma_h^{n}u_h^{n}||_{L^2}^2}{2\tau}$ with $\gamma_h^{n+1}>0${, and $\tilde{D}_{\tau}\tilde{E}_h^{n+1}=\frac{||\sigma_h^{n+1}\tilde{u}_h^{n+1}||_{L^2}^2-||\sigma_h^{n}\tilde{u}_h^{n}||_{L^2}^2+||\sigma_h^{n+1}\tilde{u}_h^{n+1}-\sigma_h^{n}\tilde{u}_h^{n}||_{L^2}^2}{2\tau}$.}
	\end{theorem}
	\begin{proof}
{The proof consists of three parts.}

{\textit{Part I: Proof of the positivity. We only need to prove that $\rho_h^{n+1}\geq0$ if $\rho_h^{n}\geq0$.} If $\sigma_h^{n+1}(\mathbf{x})\equiv0$ almost for any $\mathbf{x}\in \Omega$,} noting that
\begin{align*}
	&(\nabla \sigma_h^{n+1}\cdot u_h^n, r_h)+\frac{1}{2}(\sigma_h^{n+1}\nabla\cdot u_h^n,r_h)\\
	&=(\nabla \cdot(\sigma_h^{n+1} u_h^n), r_h)-\frac{1}{2}(\sigma_h^{n+1}\nabla\cdot u_h^n,r_h)\\
	&=-(\sigma_h^{n+1}u_h^n, \nabla r_h)-\frac{1}{2}(\sigma_h^{n+1}\nabla\cdot u_h^n,r_h){\equiv 0},
\end{align*}
substituting this equation into \eqref{3.1}, we can derive that $\sigma_h^{n}{\equiv}0$,  which follows by $\sigma_h^{0}=\Pi_h\sigma^0{\equiv}0$. It is contradictory with \eqref{tsig}. Therefore, {there exists a subdomain $S\subset \Omega$ such that $meas(S)\neq\emptyset$ and ${\sigma}_h^{n+1}(\mathbf{x})\not\equiv0$ for all $\mathbf{x}\in S$, which follows $\int_\Omega\bar{\rho}_h^{n+1}\mathrm{d}\mathbf{x}=\int_\Omega({\sigma}_h^{n+1})^2\mathrm{d}\mathbf{x}>0$   by using \eqref{3.5}. Thus, $\lambda_h^{n+1}=\frac{\int_{\Omega}\rho_h^n \mathrm{d}\mathbf{x}}{\int_{\Omega}\bar{\rho}_h^{n+1}\mathrm{d}\mathbf{x}}\geq0$ and $\rho_h^{n+1}=\lambda_h^{n+1}\bar{\rho}_h^{n+1}=\lambda_h^{n+1}({\sigma}_h^{n+1})^2\geq0$} can be easily derived by combining the induction method with \eqref{3.5}-\eqref{3.6}.

%

{\textit{Part II: Proof of the mass conservation.}} Using \eqref{3.5} and \eqref{3.6}, we can deduce that mass conservation
\begin{equation} \int_{\Omega}\rho_h^{n+1}\mathrm{d}\mathbf{x}=\int_{\Omega}\lambda^{n+1}\bar{\rho}_h^{n+1}\mathrm{d}\mathbf{x}=\int_{\Omega}\rho_h^n\mathrm{d}\mathbf{x}.\nonumber
\end{equation}

{\textit{Part III: Proof of the energy identical-relation.}} Taking $(v_h, q_h)=(\tilde{u}_h^{n+1}, p_h^{n+1})$ in \eqref{3.2}, and applying
\begin{align*}
	&(\rho_h^{n}(u_h^n\cdot\nabla)\tilde{u}_h^{n+1}, \tilde{u}_h^{n+1})+\frac{1}{2}(\tilde{u}_h^{n+1}\nabla\cdot(\rho_h^{n}u_h^n), \tilde{u}_h^{n+1})\\
	&=\int_{\Omega}\rho_h^{n}(u_h^n\cdot\nabla)\tilde{u}_h^{n+1}\cdot \tilde{u}_h^{n+1}d\mathbf{x}+\frac{1}{2}\int_{\Omega}\tilde{u}_h^{n+1}\nabla\cdot(\rho_h^{n}u_h^n)\cdot\tilde{u}_h^{n+1}d\mathbf{x}\\
	&=\frac{1}{2}\int_{\Omega}\rho_h^{n}u_h^n\cdot\nabla|\tilde{u}_h^{n+1}|^2d\mathbf{x}+\frac{1}{2}\int_{\Omega}|\tilde{u}_h^{n+1}|^2\nabla\cdot(\rho_h^{n}u_h^n)d\mathbf{x}\\
	&=\int_{\partial\Omega}|\tilde{u}_h^{n+1}|^2\rho_h^{n}u_h^n\cdot\vec{\nu}d\mathbf{x}-\frac{1}{2}\int_{\Omega}|\tilde{u}_h^{n+1}|^2\nabla\cdot(\rho_h^{n}u_h^n)d\mathbf{x}\\
	&+\frac{1}{2}\int_{\Omega}|\tilde{u}_h^{n+1}|^2\nabla\cdot(\rho_h^{n}u_h^n)d\mathbf{x}\\
	&=0,
\end{align*}
we can get
\begin{equation*}
	(D_{\tau}(\sigma_h^{n+1}\tilde{u}_h^{n+1}), \sigma_h^{n+1}\tilde{u}_h^{n+1})+\mu\int_{\Omega}|\nabla\tilde{u}_h^{n+1}|^2\mathrm{d}\mathbf{x}=\int_{\Omega}f^{n+1}\tilde{u}_h^{n+1}\mathrm{d}\mathbf{x}.
\end{equation*}

If $||\sigma_h^{n+1}\tilde{u}_h^{n+1}||_{L^2}=0$, {the energy identical-relation is obvious by noting the equation $(D_{\tau}(\sigma_h^{n+1}\tilde{u}_h^{n+1}),\sigma_h^{n+1}\tilde{u}_h^{n+1})\!\!=\!\frac{||\sigma_h^{n+1}\tilde{u}_h^{n+1}||_{L^2}^2\!\!-\!||\sigma_h^{n}\tilde{u}_h^{n}||_{L^2}^2\!\!+\!||\sigma_h^{n+1}\tilde{u}_h^{n+1}\!\!-\!\sigma_h^{n}\tilde{u}_h^{n}||_{L^2}^2}{2\tau}$.}

If $||\sigma_h^{n+1}\tilde{u}_h^{n+1}||_{L^2}\neq0$, due to \eqref{3.4} and \eqref{3.3}, $(D_{\tau}(\sigma_h^{n+1}\tilde{u}_h^{n+1}), \sigma_h^{n+1}\tilde{u}_h^{n+1})$ can be expressed as follows:
\begin{equation}
	\begin{aligned}
		&(D_{\tau}(\sigma_h^{n+1}\tilde{u}_h^{n+1}), \sigma_h^{n+1}\tilde{u}_h^{n+1})\\	&=\frac{||\sigma_h^{n+1}\tilde{u}_h^{n+1}||_{L^2}^2-||\sigma_h^n\tilde{u}_h^n||_{L^2}^2+||\sigma_h^{n+1}\tilde{u}_h^{n+1}-\sigma_h^n\tilde{u}_h^n||_{L^2}^2}{2\tau}\\
		&=\frac{\gamma_h^{n+1}||\sigma_h^{n+1}\tilde{u}_h^{n+1}||_{L^2}^2-||\sigma_h^{n}u_h^{n}||_{L^2}^2}{2\tau}.\label{E}
	\end{aligned}
\end{equation}


Next, we will prove $\gamma_{h}^{n+1}>0$. When $||\sigma_h^{n+1}\tilde{u}_h^{n+1}||_{L^2}=0$, $\sigma_h^{n+1}=1$. Therefore, we only need consider the case when $||\sigma_h^{n+1}\tilde{u}_h^{n+1}||_{L^2}\neq0$ by using the induction method in the following.

(I) When $n=0$,  thanks to $\tilde{u}_h^0=u_h^0$, it yields $\gamma_h^1=1+\frac{||\sigma_h^{1}\tilde{u}_h^{1}-\sigma_h^0\tilde{u}_h^0||_{L^2}^2}{||\sigma_h^{1}\tilde{u}_h^{1}||_{L^2}^2}>0$.

(II) Assume $\gamma_h^m>0$ for all $1\leq m\leq N-1$.  Summing over $n$ from $0$ to $m$ in \eqref{E} and utilizing \eqref{3.4}, we can get
\begin{equation}
	\begin{aligned}			&||\sigma_h^{m+1}\tilde{u}_h^{m+1}||_{L^2}^2-||\sigma_h^0\tilde{u}_h^0||_{L^2}^2+\sum\limits_{i=0}^{m}||\sigma_h^{i+1}\tilde{u}_h^{i+1}-\sigma_h^{i}\tilde{u}_h^i||_{L^2}^2\\ 
		&=\gamma_h^{m+1}||\sigma_h^{m+1}\tilde{u}_h^{m+1}||_{L^2}^2-||\sigma_h^0u_h^0||_{L^2}^2,\nonumber
	\end{aligned}
\end{equation}
which implies, by noting $\tilde{u}_h^0=u_h^0$ again, that
\begin{equation} \gamma_h^{m+1}=1+\frac{\sum\limits_{i=0}^{m}||\sigma_h^{i+1}\tilde{u}_h^{i+1}-\sigma_h^{i}\tilde{u}_h^i||_{L^2}^2}{||\sigma_h^{m+1}\tilde{u}_h^{m+1}||_{L^2}^2}>0.\nonumber
\end{equation}

Therefore, it always holds $\gamma_{h}^{n+1}>0$ for all $0\leq n\leq N-1$. It follows by combining with \eqref{E} that
\begin{equation}
	\begin{aligned}
		(D_{\tau}(\sigma_h^{n+1}\tilde{u}_h^{n+1}), \sigma_h^{n+1}\tilde{u}_h^{n+1})&=\frac{||\sigma_h^{n+1}\sqrt{\gamma_h^{n+1}}\tilde{u}_h^{n+1}||_{L^2}^2-||\sigma_h^nu_h^n||_{L^2}^2}{2\tau}\\
		&=\frac{||\sigma_h^{n+1}u_h^{n+1}||_{L^2}^2-||\sigma_h^nu_h^n||_{L^2}^2}{2\tau}\\
		&=D_{\tau}E_h^{n+1},\nonumber
	\end{aligned}
\end{equation}
which indicates the original energy identical-relation. The proof is completed.
	\end{proof}
	\begin{remark}
Although the energy identical-relation was considered in \cite{bib28}, their energy  is a modified one based on the scalar auxiliary variable method, and their scheme doesn't preserve the positivity of the density. Moreover, if the density $\rho$ is a constant, the equations (\ref{1.1})-(\ref{1.3}) reduce to the classical Navier-Stokes equations, and the energy identical-relation derived in Theorem \ref{theorem3.1} holds in this case, too.  Different from the energy dissipation law which has been widely investigated for the discrete scheme of the Navier-Stokes equations with constant and variable densities by assuming that the body force $f=0$ (see, i.e., \cite{LSL2022,bib26,Q2024}), the energy law proved here for the scheme \eqref{3.1}-\eqref{3.6} is an equality, which is a discrete analogue of the continuous property presented in Section 2.3. If $-\mu\int_{\Omega}|\nabla \tilde{u}_h^{n+1}|^2\mathrm{d}\mathbf{x}+\int_{\Omega}f^{n+1}\tilde{u}_h^{n+1}\mathrm{d}\mathbf{x}\leq 0$ (the body force $f=0$ can be seen as a special case under this condition), the energy of the scheme \eqref{3.1}-\eqref{3.6} will obey the dissipation law.  Otherwise, the energy of the scheme \eqref{3.1}-\eqref{3.6} will increase, which means that the energy from the external body force $f$ is greater than the dissipation part of the system. This is consistent with the continuous property. The numerical example shown in Section 5 will confirm this fact.
\end{remark}

\begin{remark}
	If $||\sigma_h^{n+1}\tilde{u}_h^{n+1}||_{L^2}=0$,  the kinetic energy of the numerical scheme is zero. Although only the energy-identical-relation with a numerical dissipation term $\frac{||\sigma_h^{n+1}\tilde{u}_h^{n+1}-\sigma_h^{n}\tilde{u}_h^{n}||_{L^2}^2}{2\tau}$ is deduced in this case, $||\sigma_h^{n+1}\tilde{u}_h^{n+1}||_{L^2}$ can't be exactly equal to zero due to the existence of the round-off error in the practical simulation.
	\end{remark}
	
	\section{Error estimate}\label{sec.4}
	In this section, we will deduce the error estimate of the scheme \eqref{3.1}-\eqref{3.6}. Firstly, from the definitions of the initial data  and properties of the projections presented in Section \ref{sec.2}, we have the following results for the initial data in the scheme
	\begin{equation}
||\sigma(t_0)-\sigma_h^0||_{L^2}^2+||\rho(t_0)-\rho_h^0||_{L^2}^2+||u(t_0)-u_h^0||_{L^2}^2\leq C(\tau^2+h^4).\label{2.13}
\end{equation}

Then, for simplicity, we write $\sigma^n=\sigma(t_n,\mathbf{x}), u^n=u(t_n,\mathbf{x}), \rho^n=\rho(t_n,\mathbf{x}), p^n=p(t_n,\mathbf{x})$ as exact solution.  According to the $L^2$ projection and Stokes projection recalled in Section 2,  we can split  the errors as
\begin{align*}
e_{\sigma h}^n&=\sigma^n-\sigma_h^n=(\sigma^n-\Pi_h\sigma^n)+(\Pi_h\sigma^n-\sigma_h^n):=\eta_{\sigma h}^n+\theta_{\sigma h}^n, \\
\bar{e}_{\rho h}^n&=\rho^n-\bar{\rho}_h^n=(\rho^n-\Pi_h\rho^n)+(\Pi_h\rho^n-\bar{\rho}_h^n):=\eta_{\rho h}^n+\bar{\theta}_{\rho h}^n, \\
e_{\rho h}^n&=\rho^n-\rho_h^n=(\rho^n-\Pi_h\rho^n)+(\Pi_h\rho^n-\rho_h^n):=\eta_{\rho h}^n+\theta_{\rho h}^n, \\
\tilde{e}_{uh}^n&=u^n-\tilde{u}_h^n=(u^n-R_hu^n)+(R_hu^n-\tilde{u}_h^n):=\eta_{uh}^n+\tilde{\theta}_{uh}^n, \\
e_{uh}^n&=u^n-u_h^n=(u^n-R_hu^n)+(R_hu^n-u_h^n):=\eta_{uh}^n+\theta_{uh}^n, \\
e_{ph}^n&=p^n-p_h^n=(p^n-Q_hp^n)+(Q_hp^n-p_h^n):=\eta_{ph}^n+\theta_{ph}^n.
\end{align*}

On the other hand, from \eqref{2.10}-\eqref{2.11}, we can derive
\begin{equation}
(D_{\tau}\sigma^{n+1}, r)+(\nabla\sigma^{n+1}\cdot u^n, r)+\frac{1}{2}(\sigma^{n+1}\nabla\cdot u^n,r)=(R_{\sigma}^{n+1}, r), ~~\forall r\in W,\label{3.18}
\end{equation}
and
\begin{equation}
\begin{aligned}
	&(\sigma^{n+1}D_{\tau}(\sigma^{n+1}u^{n+1}), v)+\mu(\nabla u^{n+1}, \nabla v)+(\rho^{n}(u^n\cdot\nabla)u^{n+1}, v)\\
	&+\frac{1}{2}(u^{n+1}\nabla\cdot(\rho^{n}u^n), v)-(\nabla\cdot v, p^{n+1})+(\nabla\cdot u^{n+1}, q)\\
	&=(f^{n+1}, v)+(R_{u}^{n+1}, v), ~~\forall (v,q)\in V\times M,\label{3.19}
\end{aligned}
\end{equation}
where
\begin{align*}
R_{\sigma }^{n+1}&=D_{\tau}\sigma^{n+1}-\sigma_t^{n+1}+\nabla\sigma^{n+1}(u^n-u^{n+1}),\\
R_{u}^{n+1}&=\sigma^{n+1}D_{\tau}(\sigma^{n+1}u^{n+1})-\sigma^{n+1}(\sigma u)_t(t_{n+1})\\
&+(\rho^n-\rho^{n+1})(u^n\cdot\nabla)u^{n+1}+\rho^{n+1}((u^n-u^{n+1})\cdot\nabla)u^{n+1}\\
&+\frac{u^{n+1}}{2}\nabla\cdot((\rho^{n}-\rho^{n+1})u^n)+\frac{u^{n+1}}{2}\nabla\cdot(\rho^{n+1}(u^n-u^{n+1})).
\end{align*}
For the above two truncation errors, there holds the following convergence order.
\begin{lemma}\label{lem.3.2}
Under Assumption \ref{as.2.2}, it is valid that
\begin{equation}
	||R_{\sigma }^{n+1}||_{L^2}^2+||R_{u}^{n+1}||_{L^2}^2\leq C\tau^2.\label{3.20}
\end{equation}
\end{lemma}
\begin{proof}
By the Taylor's expansion,  we can easily get
\begin{equation}
	D_{\tau}g^{n+1}-g_t(t_{n+1})=O(\tau),\label{after3.19}
\end{equation}
for any smooth enough function $g$. Based on the expressions for $R_{\sigma }^{n+1}$ and $R_{u}^{n+1}$, along with \eqref{after3.19} and Assumption \ref{as.2.2}, we can deduce
\begin{equation}
	||R_{\sigma }^{n+1}||_{L^2}^2\leq C\tau^2+C||u^n-u^{n+1}||_{L^2}^2\leq C\tau^2,\nonumber
\end{equation}
and
\begin{equation}
	\begin{aligned}
		||R_{u}^{n+1}||_{L^2}^2\leq C\tau^2+C||\rho^n-\rho^{n+1}||_{L^2}^2+C||u^n-u^{n+1}||_{L^2}^2\leq C\tau^2.\nonumber
	\end{aligned}
\end{equation}
The proof is completed.
\end{proof}

Moreover, setting $r=r_h\in W_h\subset W$ and $(v, q)=(v_h,q_h)\in (V_h,M_h)\subset(V, M)$ in \eqref{3.18} and \eqref{3.19},  subtracting \eqref{3.1} and \eqref{3.2}  from \eqref{3.18} and \eqref{3.19}, respectively, we have the error equations
\begin{equation}
\begin{aligned}
	&(D_{\tau}(e_{\sigma h}^{n+1}), r_h)+(\nabla\sigma^{n+1}\cdot e_{uh}^n, r_h)+(u_h^n\cdot\nabla e_{\sigma h}^{n+1}, r_h)\\
	&+\frac{1}{2}(\sigma^{n+1}\nabla\cdot e_{uh}^n,r_h)+\frac{1}{2}(\nabla\cdot u_h^ne_{\sigma h}^{n+1},r_h)=(R_{\sigma }^{n+1}, r_h),\label{3.21}
\end{aligned}
\end{equation}
and
\begin{align*}
&(e_{\sigma h}^{n+1}D_{\tau}(\sigma^{n+1}u^{n+1}),v_h)\!\!+\!(\sigma_h^{n+1}D_{\tau}(e_{\sigma h}^{n+1}u^{n+1}),v_h)\!\!+\!(\sigma_h^{n+1}D_{\tau}(\sigma_h^{n+1}\tilde{e}_{uh}^{n+1}),v_h)\\
&+\mu(\nabla\tilde{e}_{uh}^{n+1},\nabla v_h)+(e_{\rho h}^n(u^n\cdot\nabla)u^{n+1},v_h)+(\rho_h^n(e_{uh}^n\cdot\nabla)u^{n+1},v_h)\nonumber\\
&+(\rho_h^n(u_h^n\cdot\nabla)\tilde{e}_{uh}^{n+1},v_h)+\frac{1}{2}(u^{n+1}\nabla\cdot(e_{\rho h}^nu^n),v_h)+\frac{1}{2}(u^{n+1}\nabla\cdot(\rho_h^ne_{uh}^n),v_h)\nonumber\\
&+\frac{1}{2}(\tilde{e}_{uh}^{n+1}\nabla\cdot(\rho_h^nu_h^n),v_h)-(\nabla\cdot v_h,e_{ph}^{n+1})+(\nabla\cdot\tilde{e}_{uh}^{n+1},q_h)=(R_{u}^{n+1},v_h).
\end{align*}
Thanks to \eqref{3.13}-\eqref{3.15}, the above error equation can be written as
\begin{equation}
\begin{aligned}
	&(\sigma_h^{n+1}D_{\tau}(\sigma_h^{n+1}\tilde{\theta}_{uh}^{n+1}), v_h)+\mu(\nabla\tilde{\theta}_{uh}^{n+1}, \nabla v_h)-(\nabla\cdot v_h, \theta_{ph}^{n+1})\\
	&+(\nabla\cdot \tilde{\theta}_{uh}^{n+1}, q_h)=(R_{u}^{n+1}, v_h)-\sum\limits_{i=1}^{9}(Y_{i}^{n+1}, v_h),\label{3.22}
\end{aligned}
\end{equation}
where	\begin{align*}
&Y_{1}^{n+1}=e_{\sigma h}^{n+1}D_{\tau}(\sigma^{n+1}u^{n+1}), \\
&Y_{2}^{n+1}=\sigma_h^{n+1}D_{\tau}(e_{\sigma h}^{n+1}u^{n+1}), \\
&Y_{3}^{n+1}=\sigma_h^{n+1}D_{\tau}(\sigma_h^{n+1}\eta_{uh}^{n+1}), \\
&Y_{4}^{n+1}=e_{\rho h}^{n}(u^n\cdot\nabla)u^{n+1}, \\
&Y_{5}^{n+1}=\rho_h^{n}(e_{uh}^n\cdot\nabla)u^{n+1}, \\
&Y_{6}^{n+1}=\rho_h^{n}(u_h^n\cdot\nabla)\tilde{e}_{uh}^{n+1},\\
&Y_{7}^{n+1}=\frac{1}{2}u^{n+1}\nabla\cdot(e_{\rho h}^{n}u^n), \\
&Y_{8}^{n+1}=\frac{1}{2}u^{n+1}\nabla\cdot(\rho_h^{n}e_{uh}^n),\\
&Y_{9}^{n+1}=\frac{1}{2}\tilde{e}_{uh}^{n+1}\nabla\cdot(\rho_h^{n}u_h^n).
\end{align*}

Next, we will analyze the error equations \eqref{3.21} and \eqref{3.22} in detail. For the error equation \eqref{3.21}, there holds the following lemma.
\begin{lemma}\label{lem.3.3}
Under Assumptions \ref{as.2.2}, there exists a constant $\tau_1>0$, if  $\tau<\tau_1$, then it is valid, for all $0\leq n\leq N-1$, that
\begin{equation}
	\begin{aligned}
		&||\theta_{\sigma h}^{n+1}||_{L^2}^2+\sum\limits_{i=0}^n||\theta_{\sigma h}^{i+1}-\theta_{\sigma h}^i||_{L^2}^2\\
		&\leq C(\tau^2+h^4)+C\tau\sum\limits_{i=0}^n(h^4||\nabla u_{h}^i||_{L^2}^2+h^4||u_{h}^i||_{L^\infty}^2+||\nabla \theta_{uh}^i||_{L^2}^2).\label{3.23}
	\end{aligned}
\end{equation}
\end{lemma}
\begin{proof}
Firstly,  taking $r_h=2\tau\theta_{\sigma h}^{n+1}\in W_h$ in \eqref{3.21} and employing \eqref{3.13} yield
\begin{equation}
	\begin{aligned}
		&||\theta_{\sigma h}^{n+1}||_{L^2}^2-||\theta_{\sigma h}^n||_{L^2}^2+||\theta_{\sigma h}^{n+1}-\theta_{\sigma h}^n||_{L^2}^2\\
		&\leq -2\tau(\nabla\sigma^{n+1}e_{uh}^n, \theta_{\sigma h}^{n+1})-2\tau(u_h^n\cdot\nabla e_{\sigma h}^{n+1}, \theta_{\sigma h}^{n+1})-\tau(\sigma^{n+1}\nabla\cdot e_{uh}^n,\theta_{\sigma h}^{n+1})\\
		&-\tau(\nabla\cdot u_h^ne_{\sigma h}^{n+1},\theta_{\sigma h}^{n+1})+(R_{\sigma }^{n+1}, 2\tau\theta_{\sigma h}^{n+1}).\label{3.24}
	\end{aligned}
\end{equation}
Then, using \eqref{3.16}, the Poincare inequality and the Young inequality, we can obtain
\begin{equation*}
	\begin{aligned}
		&|2\tau(\nabla\sigma^{n+1}e_{uh}^n, \theta_{\sigma h}^{n+1})+\tau(\sigma^{n+1}\nabla\cdot e_{uh}^n,\theta_{\sigma h}^{n+1})|\\
		\leq& C\tau||\nabla\sigma^{n+1}||_{L^{\infty}}||\nabla e_{uh}^n||_{L^2}||\theta_{\sigma h}^{n+1}||_{L^2}+C\tau||\sigma^{n+1}||_{L^{\infty}}||\nabla\cdot e_{uh}^n||_{L^2}||\theta_{\sigma h}^{n+1}||_{L^2}\\
		\leq& C\tau(||\nabla\eta_{uh}^n||_{L^2}\!\!+\!\!||\nabla\theta_{uh}^n||_{L^2})||\theta_{\sigma h}^{n+1}||_{L^2}\!\!+\!C\tau(||\nabla\eta_{uh}^n||_{L^2}\!\!+\!||\nabla \theta_{uh}^n||_{L^2})||\theta_{\sigma h}^{n+1}||_{L^2}\\
		\leq& C\tau h^4+C\tau||\theta_{\sigma h}^{n+1}||_{L^2}^2+C\tau ||\nabla \theta_{uh}^n||_{L^2}^2. \label{3.25}
	\end{aligned}
\end{equation*}

Then, using \eqref{2.1} and \eqref{3.16} 
we arrive at
\begin{equation*}
\begin{aligned}
	&|2\tau(u_h^n\cdot\nabla e_{\sigma h}^{n+1}, \theta_{\sigma h}^{n+1})+\tau(\nabla\cdot u_h^ne_{\sigma h}^{n+1},\theta_{\sigma h}^{n+1})|\\
	&=|2\tau(u_h^n\cdot\nabla\eta_{\sigma h}^{n+1},\theta_{\sigma h}^{n+1})+2\tau(u_h^n\cdot\nabla\theta_{\sigma h}^{n+1},\theta_{\sigma h}^{n+1})\\
	&+\tau(\nabla\cdot u_h^n\theta_{\sigma h}^{n+1},\theta_{\sigma h}^{n+1})+\tau(\nabla\cdot u_h^n\eta_{\sigma h}^{n+1},\theta_{\sigma h}^{n+1})|\\
	&\leq C\tau||u_h^n||_{L^\infty}||\nabla\eta_{\sigma h}^{n+1}||_{L^2}||\theta_{\sigma h}^{n+1}||_{L^2}+\tau(u_h^n,\nabla|\theta_{\sigma h}^{n+1}|^2)+\tau(\nabla\cdot u_h^n,(\theta_{\sigma h}^{n+1})^2)\\
	&+C\tau||\nabla u_h^n||_{L^\infty}||\eta_{\sigma h}^{n+1}||_{L^2}||\theta_{\sigma h}^{n+1}||_{L^2}\\
	&\leq C\tau h^2||u_{h}^n||_{L^\infty}||\theta_{\sigma h}^{n+1}||_{L^2}\!+\!C\tau h^2||\nabla u_{h}^n||_{L^2}||\theta_{\sigma h}^{n+1}||_{L^2}\\
	&\leq C\tau||\theta_{\sigma h}^{n+1}||_{L^2}^2+C\tau h^4||u_{h}^n||_{L^\infty}^2+C\tau h^4||\nabla u_{h}^n||_{L^2}^2.\label{3.28}
\end{aligned}
\end{equation*}

Finally, combining \eqref{3.20} with the Young inequality, we can deduce
\begin{equation*}
|(R_{\sigma }^{n+1}, 2\tau\theta_{\sigma h}^{n+1})|\leq C\tau||R_{\sigma }^{n+1}||_{L^2}^2+C\tau||\theta_{\sigma h}^{n+1}||_{L^2}^2\leq C\tau^3+C\tau||\theta_{\sigma h}^{n+1}||_{L^2}^2.\label{3.29}
\end{equation*}
Putting these inequalities into  \eqref{3.24} and taking a summation, we have
\begin{equation*}
\begin{aligned}
	||\theta_{\sigma h}^{n+1}||_{L^2}^2+\sum\limits_{i=0}^n||\theta_{\sigma h}^{i+1}-\theta_{\sigma h}^i||_{L^2}^2&\leq C\tau\sum\limits_{i=0}^n(\tau^2+h^4)+C\tau\sum\limits_{i=0}^n||\theta_{\sigma h}^{i+1}||_{L^2}^2\\
	&+C\tau\sum\limits_{i=0}^n(h^4||\nabla u_{h}^i||_{L^2}^2+h^4||u_{h}^i||_{L^\infty}^2+||\nabla \theta_{uh}^i||_{L^2}^2),\label{3.30}
\end{aligned}
\end{equation*}
which implies \eqref{3.23}  by applying the Gronwall inequality \eqref{2.6} and the assumption on the time step $\tau$. The proof is completed.
\end{proof}

To estimate  the error equation \eqref{3.22}, we first analyze the term $Y_2^{n+1}$, which is more complicated compared with other terms.
\begin{lemma}\label{lemma4.3} Under Assumption \ref{as.2.2},  it is valid for the term $Y_2^{n+1}$ in \eqref{3.22}, for $0\leq n\leq N-1$, that\\
\begin{equation}
\begin{aligned}
	&2\tau|(Y_{2}^{n+1}, \tilde{\theta}_{uh}^{n+1})|\\
	&\leq \frac{\mu\tau}{9}||\nabla \tilde{\theta}_{uh}^{n+1}||_{L^2}^2+C\tau^3||u_h^n||_{L^{\infty}}^2\\
	&+C\tau h^4||u_h^n||_{L^{\infty}}^2||\sigma_h^{n+1}||_{L^\infty}^2\!+\!C\tau ||\sigma_h^{n+1}||_{L^\infty}^2||D_{\tau}\theta_{\sigma h}^{n+1}||_{L^3}^2(h^4\!\!+\!||\nabla\theta_{uh}^n||_{L^2}^2)\\
	&+C\tau h^2||D_{\tau}\theta_{\sigma h}^{n+1}||_{L^2}^2||\sigma_h^{n+1}||_{L^{\infty}}^2(||\nabla u_h^n||_{L^3}^2+||u_h^n||_{L^{\infty}}^2)\\
	&+C\tau h^2||\nabla e_{\sigma h}^{n+1}||_{L^2}^2||u_h^n||_{L^{\infty}}^2||\sigma_h^{n+1}||_{L^{\infty}}^2(||\nabla u_h^n||_{L^3}^2+||u_h^n||_{L^{\infty}}^2)\\
	&+C\tau||u_h^n||_{L^{\infty}}^2||\sigma_h^{n+1}||_{L^\infty}^2||\nabla\theta_{uh}^n||_{L^2}^2+C\tau||\nabla u_h^n||_{L^{3}}^2||e_{\sigma h}^{n+1}||_{L^2}^2||u_h^n||_{L^{\infty}}^2\\
	&+C\tau||e_{\sigma h}^{n+1}||_{L^2}^2(||u_h^n||_{W^{1,3}}^2||\sigma_h^{n+1}||_{L^{\infty}}^2||u_h^n||_{L^{\infty}}^2)\\
	&+C\tau||e_{\sigma h}^{n+1}||_{L^2}^2(||u_h^n||_{L^{\infty}}^4||\sigma_h^{n+1}||_{W^{1,3}}^2+||u_h^n||_{L^{\infty}}^4||\sigma_h^{n+1}||_{L^{\infty}}^2)\\
	&+C\tau||\sigma_h^{n+1}||_{L^{\infty}}^2||e_{\sigma h}^{n+1}||_{L^2}^2+C\tau h^4||\sigma_h^{n+1}||_{L^{\infty}}^2.
	\label{3.47}
\end{aligned}
\end{equation}
\end{lemma}
\begin{proof}
Obviously, $2\tau|(Y_{2}^{n+1}, \tilde{\theta}_{uh}^{n+1})|$ can be disassembled into three terms
\begin{equation}
\begin{aligned}
	&2\tau|(Y_{2}^{n+1}, \tilde{\theta}_{uh}^{n+1})|\\
	&=2\tau|(\sigma_h^{n+1}D_{\tau}(e_{\sigma h}^{n+1}u^{n+1}),\tilde{\theta}_{uh}^{n+1})|\\
	&\leq 2\tau|(\sigma_h^{n+1}e_{\sigma h}^{n+1}D_{\tau}u^{n+1}, \tilde{\theta}_{uh}^{n+1})|\!+\!2\tau|(\sigma_h^{n+1}u^nD_{\tau}e_{\sigma h}^{n+1},\tilde{\theta}_{uh}^{n+1})|\\
	&\leq 2\tau|(\sigma_h^{n+1}e_{\sigma h}^{n+1}D_{\tau}u^{n+1}, \tilde{\theta}_{uh}^{n+1})|\!+\!2\tau|(\sigma_h^{n+1}u^nD_{\tau}\eta_{\sigma h}^{n+1}, \tilde{\theta}_{uh}^{n+1})|\\
	&+2\tau|(\sigma_h^{n+1}u^nD_{\tau}\theta_{\sigma h}^{n+1},\tilde{\theta}_{uh}^{n+1})|. \label{after3.34}
\end{aligned}
\end{equation}

For the first term in \eqref{after3.34}, we have
\begin{equation}
\begin{aligned}
	&2\tau|(\sigma_h^{n+1}e_{\sigma h}^{n+1}D_{\tau}u^{n+1}, \tilde{\theta}_{uh}^{n+1})|\\
	&\leq C\tau||\sigma_h^{n+1}||_{L^{\infty}}||e_{\sigma h}^{n+1}||_{L^2}||D_{\tau}u^{n+1}||_{L^3}||\tilde{\theta}_{uh}^{n+1}||_{L^6}\\
	&\leq \frac{\mu\tau}{81}||\nabla \tilde{\theta}_{uh}^{n+1}||_{L^2}^2+C\tau||\sigma_h^{n+1}||_{L^{\infty}}^2||e_{\sigma h}^{n+1}||_{L^2}^2,
	\label{4.22}
\end{aligned}
\end{equation}
where we have used
\begin{equation}
||D_{\tau}u^{n+1}||_{L^3}\leq ||u_t+O(\tau)||_{L^3}\leq C.\nonumber
\end{equation}

Additionally, thanks to the Poincare inequality, the second term in \eqref{after3.34} can be estimated as follows:
\begin{equation}
\begin{aligned}
	&2\tau|(\sigma_h^{n+1}u^nD_{\tau}\eta_{\sigma h}^{n+1}, \tilde{\theta}_{uh}^{n+1})|\\
	&\leq C\tau ||\sigma_h^{n+1}||_{L^{\infty}}||u^n||_{L^{\infty}}||D_{\tau}\eta_{\sigma h}^{n+1}||_{L^2}||\nabla \tilde{\theta}_{uh}^{n+1}||_{L^2}\\
	&\leq \frac{\mu\tau}{81}||\nabla \tilde{\theta}_{uh}^{n+1}||_{L^2}^2+C\tau h^4||\sigma_h^{n+1}||_{L^{\infty}}^2, \label{4.23}
\end{aligned}
\end{equation}
where the following inequality \cite{bib17} is used in the last step
\begin{equation}
||D_{\tau}\eta_{\sigma h}^{n+1}||_{L^2}\leq Ch^2||D_{\tau}\sigma^{n+1}||_{H^2}\leq Ch^2||\sigma_t+O(\tau)||_{H^2}\leq Ch^2.\nonumber
\end{equation}

Finally, by employing \eqref{3.16}, the last term in \eqref{after3.34} follows by
\begin{equation}
\begin{aligned}
	&2\tau|(\sigma_h^{n+1}u^nD_{\tau}\theta_{\sigma h}^{n+1},\tilde{\theta}_{uh}^{n+1})|\\
	&\leq 2\tau|(\sigma_h^{n+1}e_{uh}^nD_{\tau}\theta_{\sigma h}^{n+1},\tilde{\theta}_{uh}^{n+1})|+2\tau|(\sigma_h^{n+1}u_h^nD_{\tau}\theta_{\sigma h}^{n+1},\tilde{\theta}_{uh}^{n+1})|\\
	&\leq C\tau||D_{\tau}\theta_{\sigma h}^{n+1}||_{L^3}(||\eta_{uh}^n||_{L^6}+||\theta_{uh}^n||_{L^6})||\sigma_h^{n+1}||_{L^\infty}||\tilde{\theta}_{uh}^{n+1}||_{L^2}\\
	&+2\tau|(\sigma_h^{n+1}u_h^nD_{\tau}\theta_{\sigma h}^{n+1},\tilde{\theta}_{uh}^{n+1})|\\
	&\leq \frac{\mu\tau}{81}||\nabla\tilde{\theta}_{uh}^{n+1}||_{L^2}^2+C\tau h^4||\sigma_h^{n+1}||_{L^\infty}||D_{\tau}\theta_{\sigma h}^{n+1}||_{L^3}^2\\
	&+C\tau||\sigma_h^{n+1}||_{L^\infty}||D_{\tau}\theta_{\sigma h}^{n+1}||_{L^3}^2||\nabla \theta_{uh}^n||_{L^2}^2+2\tau|(\sigma_h^{n+1}u_h^nD_{\tau}\theta_{\sigma h}^{n+1},\tilde{\theta}_{uh}^{n+1})|.\label{4.24}
\end{aligned}
\end{equation}
To estimate the term $2\tau|(\sigma_h^{n+1}u_h^nD_{\tau}\theta_{\sigma h}^{n+1},\tilde{\theta}_{uh}^{n+1})|$ in \eqref{4.24}, we introduce the piecewise constant finite element space \cite{bib17}
\begin{equation}
W_h^0=\{q_h\in L^2(\Omega)|q_h\in P_0(K), \forall K\in \mathcal{T}_h\}.\nonumber
\end{equation}
Let $S_h$  denote the $L^2$ projection operator from $L^2(\Omega)$ onto $W_h^0$ \cite{bib17}, then
\begin{equation}
||q-S_hq||_{L^2}\leq Ch||q||_{H^1}~~ \mathrm{and} ~~||S_hq||_{L^2}\leq ||q||_{L^2},\label{Sh}
\end{equation}
which follows that
\begin{equation}
\begin{aligned}
	&||(u_h^n\cdot\tilde{\theta}_{uh}^{n+1})-S_h(u_h^n\cdot\tilde{\theta}_{uh}^{n+1})||_{L^2}\\
	&\leq Ch||u_h^n\cdot\tilde{\theta}_{uh}^{n+1}||_{H^1}\\
	&\leq Ch(||\nabla u_h^n||_{L^3}||\nabla\tilde{\theta}_{uh}^{n+1}||_{L^2}+||u_h^n||_{L^{\infty}}||\nabla\tilde{\theta}_{uh}^{n+1}||_{L^2}).\label{4.26}
\end{aligned}
\end{equation}
Thus, using \eqref{4.26} and Young inequality, we have
\begin{equation}
\begin{aligned}
	&2\tau|(\sigma_h^{n+1}u_h^nD_{\tau}\theta_{\sigma h}^{n+1},\tilde{\theta}_{uh}^{n+1})|\\
	&=2\tau|(D_{\tau}\theta_{\sigma h}^{n+1},\sigma_h^{n+1}((u_h^n\cdot\tilde{\theta}_{uh}^{n+1})-S_h(u_h^n\cdot\tilde{\theta}_{uh}^{n+1})))|\\
	&+2\tau|(D_{\tau}\theta_{\sigma h}^{n+1},\sigma_h^{n+1}S_h(u_h^n\cdot\tilde{\theta}_{uh}^{n+1}))|\\
	&\leq C\tau h||D_{\tau}\theta_{\sigma h}^{n+1}||_{L^2}||\sigma_h^{n+1}||_{L^{\infty}}(||\nabla u_h^n||_{L^3}||\nabla\tilde{\theta}_{uh}^{n+1}||_{L^2}\!\!+\!||u_h^n||_{L^{\infty}}||\nabla\tilde{\theta}_{uh}^{n+1}||_{L^2})\\
	&+2\tau|(D_{\tau}\theta_{\sigma h}^{n+1},\sigma_h^{n+1}S_h(u_h^n\cdot\tilde{\theta}_{uh}^{n+1}))|\\
	&\leq \frac{\mu\tau}{81}||\nabla \tilde{\theta}_{uh}^{n+1}||_{L^2}^2+C\tau h^2||D_{\tau}\theta_{\sigma h}^{n+1}||_{L^2}^2||\sigma_h^{n+1}||_{L^{\infty}}^2(||\nabla u_h^n||_{L^3}^2+||u_h^n||_{L^{\infty}}^2)\\
	&+2\tau|(D_{\tau}\theta_{\sigma h}^{n+1},\sigma_h^{n+1}S_h(u_h^n\cdot\tilde{\theta}_{uh}^{n+1}))|.\label{4.27}
\end{aligned}
\end{equation}
Subsequently, 
taking $r_h=2\tau\sigma_h^{n+1}S_h(u_h^n\cdot\tilde{\theta}_{uh}^{n+1})\in W_h$ in \eqref{3.21} and applying \eqref{3.13}, we arrive at
\begin{equation}
\begin{aligned}
	2\tau|(D_{\tau}\theta_{\sigma h}^{n+1}, \sigma_h^{n+1}S_h(u_h^n\cdot\tilde{\theta}_{uh}^{n+1}))|&\leq2\tau\sum\limits_{i=1}^4|(Z_{i}^{n+1}, \sigma_h^{n+1}S_h(u_h^n\cdot\tilde{\theta}_{uh}^{n+1}))|\\
	&+2\tau|(R_{\sigma }^{n+1}, \sigma_h^{n+1}S_h(u_h^n\cdot\tilde{\theta}_{uh}^{n+1}))|,
\end{aligned}
\label{3.41}	
\end{equation}
where
\begin{align*}
Z_{1}^{n+1}&=\nabla\sigma^{n+1}e_{uh}^n, \\
Z_{2}^{n+1}&=\nabla e_{\sigma h}^{n+1}u_h^n, \\
Z_{3}^{n+1}&=\frac{1}{2}\sigma^{n+1}\nabla\cdot e_{uh}^n,\\
Z_{4}^{n+1}&=\frac{1}{2}\nabla\cdot u_h^ne_{\sigma h}^{n+1}.
\end{align*}
Utilizing  \eqref{3.16} and \eqref{Sh}, we can derive
\begin{equation}
\begin{aligned}
	&2\tau|(Z_{1}^{n+1},\sigma_h^{n+1}S_h(u_h^n\cdot\tilde{\theta}_{uh}^{n+1}))|\\
	&\leq C\tau||\nabla\sigma^{n+1}||_{L^{\infty}}||e_{uh}^n||_{L^2}||u_h^{n}||_{L^{\infty}}||\sigma_h^{n+1}||_{L^\infty}||\tilde{\theta}_{uh}^{n+1}||_{L^2}\\
	&\leq C\tau (||\eta_{uh}^n||_{L^2}+||\theta_{uh}^n||_{L^2})||u_h^n||_{L^{\infty}}||\sigma_h^{n+1}||_{L^\infty}||\nabla\tilde{\theta}_{uh}^{n+1}||_{L^2}\\
	&\leq \frac{\mu\tau}{81}||\nabla \tilde{\theta}_{uh}^{n+1}||_{L^2}^2+C\tau h^6||u_h^n||_{L^{\infty}}^2||\sigma_h^{n+1}||_{L^\infty}^2\\
	&+C\tau||\nabla\theta_{uh}^n||_{L^2}^2||u_h^n||_{L^{\infty}}^2||\sigma_h^{n+1}||_{L^\infty}^2.\label{3.42}
\end{aligned}
\end{equation}
Thanks to \eqref{4.26} and the integration by parts, we get
\begin{equation}
\begin{aligned}
	&2\tau|(Z_{2}^{n+1},\sigma_h^{n+1}S_h(u_h^n\cdot\tilde{\theta}_{uh}^{n+1}))|\\
	&\leq 2\tau|(\nabla e_{\sigma h}^{n+1}u_h^n,\sigma_h^{n+1}(S_h(u_h^n\cdot\tilde{\theta}_{uh}^{n+1})-(u_h^n\cdot\tilde{\theta}_{uh}^{n+1})))|\\
	&+2\tau|(\nabla e_{\sigma h}^{n+1}u_h^n,\sigma_h^{n+1}u_h^n\cdot\tilde{\theta}_{uh}^{n+1})|\\
	&\leq C\tau h||\nabla e_{\sigma h}^{n+1}||_{L^2}||u_h^n||_{L^{\infty}}||\sigma_h^{n+1}||_{L^{\infty}}||\nabla\tilde{\theta}_{uh}^{n+1}||_{L^2}(||\nabla u_h^n||_{L^3}\!\!+\!||u_h^n||_{L^{\infty}})\\
	&+C\tau||e_{\sigma h}^{n+1}||_{L^2}||\nabla u_h^n||_{L^3}||\sigma_h^{n+1}||_{L^{\infty}}||u_h^n||_{L^{\infty}}||\tilde{\theta}_{uh}^{n+1}||_{L^6}\\
	&+C\tau||e_{\sigma h}^{n+1}||_{L^2}||u_h^n||_{L^\infty}||\nabla\sigma_h^{n+1}||_{L^3}||u_h^n||_{L^{\infty}}||\tilde{\theta}_{uh}^{n+1}||_{L^6}\\
	&+C\tau||e_{\sigma h}^{n+1}||_{L^2}||u_h^n||_{L^\infty}||\sigma_h^{n+1}||_{L^\infty}||\nabla u_h^n||_{L^3}||\tilde{\theta}_{uh}^{n+1}||_{L^6}\\
	&+C\tau||e_{\sigma h}^{n+1}||_{L^2}||u_h^n||_{L^\infty}||\sigma_h^{n+1}||_{L^\infty}||u_h^n||_{L^{\infty}}||\nabla\tilde{\theta}_{uh}^{n+1}||_{L^2}\\
	&\leq \frac{\mu\tau}{81}||\nabla\tilde{\theta}_{uh}^{n+1}||_{L^2}^2+C\tau||e_{\sigma h}^{n+1}||_{L^2}^2||u_h^n||_{W^{1,3}}^2||\sigma_h^{n+1}||_{L^{\infty}}^2||u_h^n||_{L^{\infty}}^2\\
	&+C\tau||e_{\sigma h}^{n+1}||_{L^2}^2(||u_h^n||_{L^{\infty}}^4||\sigma_h^{n+1}||_{W^{1,3}}^2+||u_h^n||_{L^{\infty}}^4||\sigma_h^{n+1}||_{L^{\infty}}^2)\\
	&+C\tau h^2||\nabla e_{\sigma h}^{n+1}||_{L^2}^2||u_h^n||_{L^{\infty}}^2||\sigma_h^{n+1}||_{L^{\infty}}^2(||\nabla u_h^n||_{L^3}^2\!+\!||u_h^n||_{L^{\infty}}^2).\label{3.43}
\end{aligned}
\end{equation}
Employing \eqref{3.16}, \eqref{Sh} and Young inequality, we have
\begin{equation}
\begin{aligned}
	&2\tau|(Z_{3}^{n+1},\sigma_h^{n+1}S_h(u_h^n\cdot\tilde{\theta}_{uh}^{n+1}))|\\
	&\leq
	C\tau||\sigma^{n+1}||_{L^{\infty}}(||\nabla\eta_{uh}^n||_{L^2}+||\nabla \theta_{uh}^n||_{L^2})||u_h^n||_{L^{\infty}}||\sigma_h^{n+1}||_{L^\infty}||\tilde{\theta}_{uh}^{n+1}||_{L^2}\\
	&\leq \frac{\mu\tau}{81}||\nabla\tilde{\theta}_{uh}^{n+1}||_{L^2}^2+C\tau h^4||u_h^n||_{L^{\infty}}^2||\sigma_h^{n+1}||_{L^\infty}^2\\
	&+C\tau||\nabla\theta_{uh}^n||_{L^2}^2||u_h^n||_{L^{\infty}}^2||\sigma_h^{n+1}||_{L^\infty}^2,\label{3.44}
\end{aligned}
\end{equation}
and
\begin{equation}
\begin{aligned}
	&2\tau|(Z_{4}^{n+1},\sigma_h^{n+1}S_h(u_h^n\cdot\tilde{\theta}_{uh}^{n+1}))|\\
	&\leq C\tau||\nabla u_h^n||_{L^{\infty}}||e_{\sigma h}^{n+1}||_{L^2}||u_h^n||_{L^{3}}||\sigma_h^{n+1}||_{L^\infty}||\tilde{\theta}_{uh}^{n+1}||_{L^6}\\
	&\leq \frac{\mu\tau}{81}||\nabla\tilde{\theta}_{uh}^{n+1}||_{L^2}^2+C\tau ||\nabla u_h^n||_{L^{3}}^2||e_{\sigma h}^{n+1}||_{L^2}^2||u_h^n||_{L^{\infty}}^2||\sigma_h^{n+1}||_{L^\infty}^2.\label{3.45}
\end{aligned}
\end{equation}
Furthermore,  utilizing \eqref{3.20} and \eqref{Sh}, we can obtain
\begin{equation}
\begin{aligned}
	&2\tau|(R_{\sigma }^{n+1}, \sigma_h^{n+1}S_h(u_h^n\cdot\tilde{\theta}_{uh}^{n+1}))|\\
	&\leq C\tau||R_{\sigma }^{n+1}||_{L^2}||u_h^n||_{L^{\infty}}||\sigma_h^{n+1}||_{L^\infty}||\tilde{\theta}_{uh}^{n+1}||_{L^2}\\
	&\leq \frac{\mu\tau}{81}||\nabla\tilde{\theta}_{uh}^{n+1}||_{L^2}^2+C\tau||R_{\sigma }^{n+1}||_{L^2}^2||u_h^n||_{L^{\infty}}^2||\sigma_h^{n+1}||_{L^\infty}^2\\
	&\leq \frac{\mu\tau}{81}||\nabla\tilde{\theta}_{uh}^{n+1}||_{L^2}^2+C\tau^3||u_h^n||_{L^{\infty}}^2||\sigma_h^{n+1}||_{L^\infty}^2. \label{3.46}
\end{aligned}
\end{equation}

Putting \eqref{4.27}-\eqref{3.46} into \eqref{4.24}, and combining with \eqref{4.22} and \eqref{4.23},  we arrive at \eqref{3.47}. The proof is completed.
\end{proof}
\begin{lemma}\label{lem.3.4} Under Assumptions \ref{as.2.2}, it is valid for the error equations \eqref{3.22}, for all $0\leq n\leq N-1$, that
\begin{equation}\label{3.37}
\begin{aligned}				&||\sigma_h^{n+1}\tilde{\theta}_{uh}^{n+1}||_{L^2}^2-||\sigma_h^{n}\tilde{\theta}_{uh}^{n}||_{L^2}^2+||\sigma_h^{n+1}\tilde{\theta}_{uh}^{n+1}-\sigma_h^{n}\tilde{\theta}_{uh}^{n}||_{L^2}^2+\mu\tau||\nabla \tilde{\theta}_{uh}^{n+1}||_{L^2}^2\\
	&\leq C\tau||e_{\sigma h}^{n+1}||_{L^2}^2+C\tau^3||u_h^n||_{L^{\infty}}^2\\
	&+C\tau h^4||u_h^n||_{L^{\infty}}^2||\sigma_h^{n+1}||_{L^\infty}^2+C\tau ||\sigma_h^{n+1}||_{L^\infty}^2||D_{\tau}\theta_{\sigma h}^{n+1}||_{L^3}^2(h^4+||\nabla\theta_{uh}^n||_{L^2}^2)\\
	&+C\tau h^2||D_{\tau}\theta_{\sigma h}^{n+1}||_{L^2}^2||\sigma_h^{n+1}||_{L^{\infty}}^2(||\nabla u_h^n||_{L^3}^2+||u_h^n||_{L^{\infty}}^2)\\
	&+C\tau h^2||\nabla e_{\sigma h}^{n+1}||_{L^2}^2||u_h^n||_{L^{\infty}}^2||\sigma_h^{n+1}||_{L^{\infty}}^2(||\nabla u_h^n||_{L^3}^2+||u_h^n||_{L^{\infty}}^2)\\
	&+C\tau||u_h^n||_{L^{\infty}}^2||\sigma_h^{n+1}||_{L^\infty}^2||\nabla\theta_{uh}^n||_{L^2}^2+C\tau||\nabla u_h^n||_{L^{3}}^2||e_{\sigma h}^{n+1}||_{L^2}^2||u_h^n||_{L^{\infty}}^2\\
	&+C\tau||e_{\sigma h}^{n+1}||_{L^2}^2(||u_h^n||_{W^{1,3}}^2||\sigma_h^{n+1}||_{L^{\infty}}^2||u_h^n||_{L^{\infty}}^2)\\
	&+C\tau||e_{\sigma h}^{n+1}||_{L^2}^2(||u_h^n||_{L^{\infty}}^4||\sigma_h^{n+1}||_{W^{1,3}}^2+||u_h^n||_{L^{\infty}}^4||\sigma_h^{n+1}||_{L^{\infty}}^2)\\
	&+C\tau||\sigma_h^{n+1}||_{L^{\infty}}^2||e_{\sigma h}^{n+1}||_{L^2}^2+C\tau h^4||\sigma_h^{n+1}||_{L^{\infty}}^2\\
	&+C\tau h^4||\bar{\rho}_h^{n+1}||_{L^{\infty}}^2+C\tau h^6||\sigma_h^{n+1}||_{L^{\infty}}^2||D_{\tau}\sigma_h^{n+1}||_{L^3}^2\\
	&+C\tau||e_{\rho h}^n||_{L^2}^2+C\tau h^6||\lambda_h^n\sigma_h^{n}||_{L^{\infty}}^2+C\tau||\lambda_h^n\sigma_h^{n}||_{L^{\infty}}^2||\sigma_h^n\theta_{uh}^n||_{L^2}^2\\
	&+C\tau h^4+C\tau||\lambda_h^n\sigma_h^{n}||_{L^{\infty}}^2||u_h^n||_{L^{\infty}}^2(||\sigma_h^{n+1}\tilde{\theta}_{uh}^{n+1}||_{L^2}^2\!\!+\!||\sigma_h^{n+1}\!\!-\!\sigma_h^n||_{L^2}^2||\tilde{\theta}_{uh}^{n+1}||_{L^\infty}^2)\\
	&+C\tau||e_{\rho h}^n||_{L^2}^2+C\tau h^6||\lambda_h^n\sigma_h^{n}||_{L^{\infty}}^2+C\tau||\lambda_h^n\sigma_h^{n}||_{L^{\infty}}^2||\sigma_h^{n}\theta_{uh}^n||_{L^2}^2\\
	&+C\tau h^4+C\tau||\lambda_h^n\sigma_h^{n}||_{L^\infty}^2||u_h^n||_{L^{\infty}}^2(h^6||\sigma_{h}^{n}||_{L^\infty}^2\\
	&+||\sigma_{h}^{n+1}\tilde{\theta}_{uh}^{n+1}||_{L^2}^2+||\sigma_{h}^{n+1}-\sigma_{h}^{n}||_{L^2}^2||\tilde{\theta}_{uh}^{n+1}||_{L^\infty}^2).
\end{aligned}
\end{equation}
\end{lemma}
\begin{proof}
Setting $(v_h,q_h)=2\tau(\tilde{\theta}_{uh}^{n+1}, \theta_{ph}^{n+1})$ into \eqref{3.22}, we obtain
\begin{equation}
\begin{aligned}
	&||\sigma_h^{n+1}\tilde{\theta}_{uh}^{n+1}||_{L^2}^2-||\sigma_h^{n}\tilde{\theta}_{uh}^{n}||_{L^2}^2+||\sigma_h^{n+1}\tilde{\theta}_{uh}^{n+1}-\sigma_h^{n}\tilde{\theta}_{uh}^{n}||_{L^2}^2\\
	&+2\mu\tau||\nabla\tilde{\theta}_{uh}^{n+1}||_{L^2}^2=2\tau(R_{u}^{n+1}, \tilde{\theta}_{uh}^{n+1})-2\tau\sum\limits_{i=1}^{9}(Y_{i}^{n+1},\tilde{\theta}_{uh}^{n+1}).\label{after3.32}
\end{aligned}
\end{equation}

Next, we analyze $2\tau\sum\limits_{i=1}^{9}(Y_{i}^{n+1},\tilde{\theta}_{uh}^{n+1}),~i=1,2,\dots,9$ one by one. Firstly, by applying the Young inequality and Poincare inequality, we can get
\begin{equation}
\begin{aligned}
	&2\tau|(Y_{1}^{n+1}, \tilde{\theta}_{uh}^{n+1})|\\
	&=2\tau|(e_{\sigma h}^{n+1}D_{\tau}(\sigma^{n+1}u^{n+1}), \tilde{\theta}_{uh}^{n+1})|\\
	&\leq C\tau||e_{\sigma h}^{n+1}||_{L^2}||D_{\tau}(\sigma^{n+1}u^{n+1})||_{L^{\infty}}||\tilde{\theta}_{uh}^{n+1}||_{L^2}\\
	&\leq C\tau||e_{\sigma h}^{n+1}||_{L^2}||\nabla \tilde{\theta}_{uh}^{n+1}||_{L^2}\\
	&\leq \frac{\mu\tau}{18}||\nabla \tilde{\theta}_{uh}^{n+1}||_{L^2}^2+C\tau||e_{\sigma h}^{n+1}||_{L^2}^2. \label{3.38}
\end{aligned}
\end{equation}

The second term $2\tau(Y_2^{n+1},\tilde{\theta}_{uh}^{n+1})$ is estimated in Lemma \ref{lemma4.3}.

For the third term, by using \eqref{3.16}, Poincare inequality and Young inequality, there holds
\begin{equation}
\begin{aligned}
	&2\tau|(Y_{3}^{n+1}, \tilde{\theta}_{uh}^{n+1})|\\
	&=2\tau|(\sigma_h^{n+1}D_{\tau}\eta_{uh}^{n+1}\sigma_h^{n+1},\tilde{\theta}_{uh}^{n+1})|+2\tau|(\sigma_h^{n+1}D_{\tau}\sigma_h^{n+1}\eta_{uh}^n,\tilde{\theta}_{uh}^{n+1})|\\
	&\leq C\tau ||\bar{\rho}_h^{n+1}||_{L^{\infty}}||D_{\tau}\eta_{uh}^{n+1}||_{L^2}||\tilde{\theta}_{uh}^{n+1}||_{L^2}\\
	&+C\tau||\sigma_h^{n+1}||_{L^\infty}||D_{\tau}\sigma_h^{n+1}||_{L^3}||\eta_{uh}^n||_{L^2}||\tilde{\theta}_{uh}^{n+1}||_{L^6}\\
	&\leq \frac{\mu\tau}{18}||\nabla \tilde{\theta}_{uh}^{n+1}||_{L^2}^2+C\tau h^4||\bar{\rho}_h^{n+1}||_{L^{\infty}}^2+C\tau h^6||\sigma_h^{n+1}||_{L^{\infty}}^2||D_{\tau}\sigma_h^{n+1}||_{L^3}^2,\label{3.48}
\end{aligned}
\end{equation}
where we have used
\begin{equation}
||D_{\tau}\eta_{uh}^{n+1}||_{L^2}\leq Ch^2||D_{\tau}u^{n+1}||_{H^2}\leq Ch^2||u_t+O(\tau)||_{H^2}\leq Ch^2.\nonumber
\end{equation}

Similarly, we can derive that
\begin{equation}
\begin{aligned}
	2\tau|(Y_{4}^{n+1}, \tilde{\theta}_{uh}^{n+1})|&=2\tau|(e_{\rho h}^{n}(u^n\cdot\nabla)u^{n+1}, \tilde{\theta}_{uh}^{n+1})|\\
	&\leq C\tau ||u^n||_{L^{\infty}}||\nabla u^{n+1}||_{L^3}||e_{\rho h}^{n}||_{L^2}||\nabla \tilde{\theta}_{uh}^{n+1}||_{L^2}\\
	&\leq \frac{\mu\tau}{18}||\nabla \tilde{\theta}_{uh}^{n+1}||_{L^2}^2+C\tau||e_{\rho h}^n||_{L^2}^2, \label{3.49}
\end{aligned}
\end{equation}
and by employing \eqref{3.16}, \eqref{3.5} and the Young inequality, we arrive at
\begin{equation}
\begin{aligned}
	&2\tau|(Y_{5}^{n+1}, \tilde{\theta}_{uh}^{n+1})|\\
	&=2\tau|(\rho_h^{n}(e_{uh}^n\cdot\nabla)u^{n+1}, \tilde{\theta}_{uh}^{n+1})|\\
	&\leq C\tau||\lambda_h^n\sigma_h^{n}||_{L^{\infty}}||\nabla u^{n+1}||_{L^3}(||\sigma_h^n\eta_{uh}^n||_{L^2}+||\sigma_h^n\theta_{uh}^n||_{L^2})|| \tilde{\theta}_{uh}^{n+1}||_{L^6}\\
	&\leq \frac{\mu\tau}{18}||\nabla \tilde{\theta}_{uh}^{n+1}||_{L^2}^2+C\tau h^6||\lambda_h^n\sigma_h^{n}||_{L^{\infty}}^2+C\tau||\lambda_h^n\sigma_h^{n}||_{L^{\infty}}^2||\sigma_h^n\theta_{uh}^n||_{L^2}^2. \label{3.50}
\end{aligned}
\end{equation}

By using the error splitting, \eqref{3.16}, the equality $\rho_h^n=\lambda_h^n(\sigma_h^n)^2$ and the integration by parts, we can deduce
\begin{equation}
\begin{aligned}
	&2\tau|(Y_{6}^{n+1}, \tilde{\theta}_{uh}^{n+1})|=2\tau|(\rho_h^{n}(u_h^n\cdot\nabla)\tilde{e}_{uh}^{n+1}, \tilde{\theta}_{uh}^{n+1})|\\
	&\leq C\tau||\lambda_h^n\sigma_h^{n}||_{L^{\infty}}||u_h^n||_{L^{\infty}}(||\nabla\eta_{uh}^{n+1}||_{L^2}+||\nabla\tilde{\theta}_{uh}^{n+1}||_{L^2})||\sigma_h^{n}\tilde{\theta}_{uh}^{n+1}||_{L^2}\\
	&\leq \frac{\mu\tau}{18}||\nabla \tilde{\theta}_{uh}^{n+ 1}|| _{L^2}^2 + C\tau h^4+C\tau||\lambda_h^n\sigma_h^{n}||_{L^{\infty}}^2||u_h^n||_{L^{\infty}}^2||\sigma_h^{n}\tilde{\theta}_{uh}^{n+1}||_{L^2}^2\\
	&\leq \frac{\mu\tau}{18}||\nabla \tilde{\theta}_{uh}^{n+ 1}|| _{L^2}^2 + C\tau h^4\\
	&+C\tau||\lambda_h^n\sigma_h^{n}||_{L^{\infty}}^2||u_h^n||_{L^{\infty}}^2(||\sigma_h^{n+1}\tilde{\theta}_{uh}^{n+1}||_{L^2}^2+||\sigma_h^{n+1}-\sigma_h^n||_{L^2}^2||\tilde{\theta}_{uh}^{n+1}||_{L^\infty}^2). \label{3.51}
\end{aligned}
\end{equation}

Similarly, noting the equality
\begin{align}\label{cue}
\nabla\cdot(u_h\otimes v_h)=(\nabla \cdot u_h)v_h+(u_h\cdot \nabla)v_h,
\end{align}
\eqref{3.5} and the Poincare inequality, there hold
\begin{equation}
\begin{aligned}
	&2\tau|(Y_{7}^{n+1}, \tilde{\theta}_{uh}^{n+1})|\\
	&=\tau|(u^{n+1}\nabla\cdot(e_{\rho h}^{n}u^n), \tilde{\theta}_{uh}^{n+1})|\\
	&=\tau|(\nabla\cdot(u^{n+1}\otimes(e_{\rho h}^{n}u^n))-(e_{\rho h}^{n}u^n\cdot\nabla)u^{n+1}, \tilde{\theta}_{uh}^{n+1})|\\
	&=\tau|-(u^{n+1}\otimes(e_{\rho h}^{n}u^n),\nabla \tilde{\theta}_{uh}^{n+1})-((e_{\rho h}^{n}u^n\cdot\nabla)u^{n+1}, \tilde{\theta}_{uh}^{n+1})|\\
	&\leq C\tau|| u^{n+1}||_{L^\infty}||e_{\rho h}^{n}||_{L^2}||u^n||_{L^\infty}||\nabla\tilde{\theta}_{uh}^{n+1}||_{L^2}\\
	&+C\tau||e_{\rho h}^{n}||_{L^2}||u^{n}||_{L^\infty}||\nabla u^n||_{L^\infty}||\tilde{\theta}_{uh}^{n+1}||_{L^2}\\
	&\leq \frac{\mu\tau}{18}||\nabla \tilde{\theta}_{uh}^{n+1}||_{L^2}^2+C\tau||e_{\rho h}^n||_{L^2}^2, \label{3.52}
\end{aligned}
\end{equation}
and
\begin{equation}
\begin{aligned}
	&2\tau|(Y_{8}^{n+1}, \tilde{\theta}_{uh}^{n+1})|\\
	&=\tau|(u^{n+1}\nabla\cdot(\rho_h^{n}e_{uh}^n), \tilde{\theta}_{uh}^{n+1})|\\
	&=\tau|-(u^{n+1}\otimes(\rho_h^{n}e_{uh}^n), \nabla\tilde{\theta}_{uh}^{n+1})-((\rho_h^{n}e_{uh}^n\cdot\nabla)u^{n+1},\tilde{\theta}_{uh}^{n+1})|\\
	&\leq C\tau|| u^{n+1}||_{L^\infty}||\lambda_h^n\sigma_h^{n}||_{L^\infty}||\sigma_h^{n}e_{uh}^n||_{L^2}||\nabla\tilde{\theta}_{uh}^{n+1}||_{L^2}\\
	&+C\tau||\lambda_h^n\sigma_h^{n}||_{L^\infty}||\sigma_h^{n}e_{uh}^n||_{L^2}||\nabla u^{n+1}||_{L^\infty}||\tilde{\theta}_{uh}^{n+1}||_{L^2}\\
	&\leq \frac{\mu\tau}{18}||\nabla\tilde{\theta}_{uh}^{n+1}||_{L^2}^2+C\tau h^6||\lambda_h^n\sigma_h^{n}||_{L^{\infty}}^2+C\tau||\lambda_h^n\sigma_h^{n}||_{L^{\infty}}^2||\sigma_h^{n}\theta_{uh}^n||_{L^2}^2. \label{3.53}
\end{aligned}
\end{equation}

Furthermore, by utilizing  \eqref{3.16}, we arrive at
\begin{equation}
\begin{aligned}
	&2\tau|(Y_{9}^{n+1}, \tilde{\theta}_{uh}^{n+1})|\\
	&=\tau|(\tilde{e}_{uh}^{n+1}\nabla\cdot(\rho_h^{n}u_h^n), \tilde{\theta}_{uh}^{n+1})|\\
	&=\tau|(\nabla\cdot(\tilde{e}_{uh}^{n+1}\otimes(\rho_h^{n}u_h^n))-((\rho_h^{n}u_h^n)\cdot\nabla)\tilde{e}_{uh}^{n+1}, \tilde{\theta}_{uh}^{n+1})|\\
	&=\tau|-(\tilde{e}_{uh}^{n+1}\otimes(\rho_h^{n}u_h^n),\nabla \tilde{\theta}_{uh}^{n+1})-((\rho_h^{n}u_h^n)\cdot\nabla)\tilde{e}_{uh}^{n+1}, \tilde{\theta}_{uh}^{n+1})|\\
	&\leq C\tau||\sigma_{h}^{n}\tilde{e}_{uh}^{n+1}||_{L^2}||\lambda_h^n\sigma_h^{n}||_{L^\infty}||u_h^n||_{L^{\infty}}||\nabla \tilde{\theta}_{uh}^{n+1}||_{L^2}\\
	&+C\tau||\lambda_h^n\sigma_h^{n}||_{L^{\infty}}||u_h^n||_{L^\infty}||\nabla\tilde{e}_{uh}^{n+1}||_{L^2}||\sigma_{h}^{n}\tilde{\theta}_{uh}^{n+1}||_{L^2}\\
	&\leq C\tau||\sigma_{h}^{n}\tilde{e}_{uh}^{n+1}||_{L^2}||\lambda_h^n\sigma_h^{n}||_{L^\infty}||u_h^n||_{L^{\infty}}||\nabla \tilde{\theta}_{uh}^{n+1}||_{L^2}\\
	&+C\tau||\lambda_h^n\sigma_h^{n}||_{L^{\infty}}||u_h^n||_{L^\infty}(||\nabla\eta_{uh}^{n+1}||_{L^2}+||\nabla\tilde{\theta}_{uh}^{n+1}||_{L^2})||\sigma_{h}^{n}\tilde{\theta}_{uh}^{n+1}||_{L^2}\\
	&\leq \frac{\mu\tau}{18}||\nabla\tilde{\theta}_{uh}^{n+1}||_{L^2}^2+C\tau h^4+C\tau||\lambda_h^n\sigma_h^{n}||_{L^\infty}^2||u_h^n||_{L^{\infty}}^2||\sigma_{h}^{n}\tilde{e}_{uh}^{n+1}||_{L^2}^2\\
	&\leq \frac{\mu\tau}{18}||\nabla\tilde{\theta}_{uh}^{n+1}||_{L^2}^2+C\tau h^4+C\tau||\lambda_h^n\sigma_h^{n}||_{L^\infty}^2||u_h^n||_{L^{\infty}}^2||\sigma_{h}^{n}(\eta_{uh}^{n+1}+\tilde{\theta}_{uh}^{n+1})||_{L^2}^2\\
	&\leq \frac{\mu\tau}{18}||\nabla\tilde{\theta}_{uh}^{n+1}||_{L^2}^2+C\tau h^4+C\tau||\lambda_h^n\sigma_h^{n}||_{L^\infty}^2||u_h^n||_{L^{\infty}}^2(h^6||\sigma_{h}^{n}||_{L^\infty}^2\\
	&+||\sigma_{h}^{n+1}\tilde{\theta}_{uh}^{n+1}||_{L^2}^2+||\sigma_{h}^{n+1}-\sigma_{h}^{n}||_{L^2}^2||\tilde{\theta}_{uh}^{n+1}||_{L^\infty}^2).\label{3.54}
\end{aligned}
\end{equation}

Finally, by employing \eqref{3.20}, we obtain that
\begin{equation}
2\tau(R_{u}^{n+1}, \tilde{\theta}_{uh}^{n+1})=C\tau||R_{u}^{n+1}||_{L^2}^2+C\tau||\tilde{\theta}_{uh}^{n+1}||_{L^2}^2\leq C\tau^3+C\tau||\tilde{\theta}_{uh}^{n+1}||_{L^2}^2.\label{Ru1}
\end{equation}

Thus, substituting \eqref{3.47} and \eqref{3.38}-\eqref{Ru1} into \eqref{after3.32}, we can have \eqref{3.37}. The proof is completed.
\end{proof}
\begin{lemma}\label{lem.3.5} Under Assumption \ref{as.2.2}, it is valid, for any $0\leq n\leq N-1$, that
\begin{align}
|1-\lambda_h^{n+1}|&\leq C(||e_{{\rho} h}^n||_{L^2}+|1-\lambda_h^{n+1}|||\bar{e}_{\rho h}^{n+1}||_{L^2}+||\bar{e}_{\rho h}^{n+1}||_{L^2}), \label{3.31}\\
||e_{{\rho} h}^{n+1}||_{L^2}&\leq C(|1-\lambda_h^{n+1}|+|1-\lambda_h^{n+1}|||\bar{e}_{\rho h}^{n+1}||_{L^2}+||\bar{e}_{\rho h}^{n+1}||_{L^2}),\label{3.32}\\
|1-\gamma_h^{n+1}|&\leq C\tau^2+ C||\sigma_h^{n+1}\tilde{e}_{uh}^{n+1}||_{L^2}^2+C||e_{\sigma h}^{n+1}||_{L^2}^2\notag\\
&+C||e_{\sigma h}^n||_{L^2}^2+C||\sigma_h^n\tilde{e}_{uh}^n||_{L^2}^2\notag\\
&+C||\sigma_h^n||_{L^\infty}(||u_h^n||_{L^{\infty}}+||\tilde{u}_h^n||_{L^{\infty}})(||\sigma_h^ne_{uh}^n||_{L^2}+||\sigma_h^n\tilde{e}_{uh}^n||_{L^2}).\label{3.33}
\end{align}
\end{lemma}
\begin{proof}
The proof of \eqref{3.31} and \eqref{3.32} can be seen in \cite{bib27}. Next, we prove \eqref{3.33}. It is clear that when $||\sigma_h^{n+1}\tilde{u}_h^{n+1}||_{L^2}=0$, the result holds trivially. When $||\sigma_h^{n+1}\tilde{u}_h^{n+1}||_{L^2}\neq 0$, there exists $\epsilon_0>0$ such that $||\sigma_h^{n+1}\tilde{u}_h^{n+1}||_{L^2}^2\geq \epsilon_0$, using Taylor's expansion and \eqref{3.3}, we derive:
\begin{equation}
\begin{aligned}
	&|1-\gamma_h^{n+1}|\\
	&\leq \frac{1}{\epsilon_0}(||\sigma_h^{n+1}\tilde{u}_h^{n+1}-\sigma_h^n\tilde{u}_h^n||_{L^2}^2-||\sigma_h^n\tilde{u}_h^n||_{L^2}^2+||\sigma_h^nu_h^n||_{L^2}^2)\\
	&\leq C||(\sigma_h^{n+1}\tilde{u}_h^{n+1}\!-\!\sigma_h^{n+1}u^{n+1})\!+\!(\sigma_h^{n+1}u^{n+1}\!-\!\sigma^{n+1}u^{n+1})\\
	&+(\sigma^{n+1}u^{n+1}\!-\!\sigma^{n}u^{n})+(\sigma^{n}u^{n}-\sigma_h^{n}u^{n})+(\sigma_h^{n}u^{n}-\sigma_h^{n}\tilde{u}_h^{n})||_{L^2}^2\\
	&+C(\sigma_h^nu_h^n,\sigma_h^n(u_h^n-u^n+u^n-\tilde{u}_h^n))+C(\sigma_h^n(u_h^n-u^n+u^n-\tilde{u}_h^n),\sigma_h^n\tilde{u}_h^n)\\
	&\leq C||\sigma_h^{n+1}\tilde{e}_{uh}^{n+1}||_{L^2}^2\!+\!C||e_{\sigma h}^{n+1}||_{L^2}^2\!+\!C\tau^2\!+\!C||e_{\sigma h}^n||_{L^2}^2\!+\!C||\sigma_h^n\tilde{e}_{uh}^n||_{L^2}^2\\
	&+C||\sigma_h^n||_{L^{\infty}}(||u_h^n||_{L^{\infty}}+||\tilde{u}_h^n||_{L^{\infty}})(||\sigma_h^ne_{uh}^n||_{L^2}+||\sigma_h^n\tilde{e}_{uh}^n||_{L^2}).\nonumber
\end{aligned}
\end{equation}
The proof is completed.
\end{proof}
\begin{theorem}\label{the.3.6} Under Assumption \ref{as.2.2} and $\tau=O(h^2)$, there exists $\tau^*>0$, if $\tau\leq\tau^*$, it is valid for the scheme \eqref{3.1}-\eqref{3.6} , for $1\leq n\leq N$, that
\begin{align}
||e_{\sigma h}^n||_{L^2}^2+||\bar{e}_{\rho h}^n||_{L^2}^2{+\sum\limits_{i=1}^{n}||\sigma_{h}^{i}-\sigma_{h}^{i-1}||_{L^2}^2}&\leq C(\tau^2+h^4),\label{3.55}\\
|1-\lambda_h^n|^2 &\leq C(\tau^2+h^4), \label{3.56}\\
||\sigma_h^ne_{\rho h}^n||_{L^2}^2 &\leq C(\tau^2+h^4),\label{3.57}\\
||\tilde{e}_{uh}^{n}||_{L^2}^2 +\tau\sum\limits_{i=1}^{n}||\nabla \tilde{e}_{uh}^{i}||_{L^2}^2&\leq C(\tau^2+h^4), \label{3.58}\\
|1-\gamma_h^{n}|^2&\leq C(\tau^2+h^4),\label{3.59}\\
||\sigma_h^ne_{uh}^{n}||_{L^2}^2+\tau\sum\limits_{i=1}^{n}||\nabla e_{uh}^{i}||_{L^2}^2&\leq C(\tau^2+h^4). \label{3.60}
\end{align}
{Moreover, there holds that
\begin{align}\label{new1}
	||\sigma_h^n||_{L^\infty}&+|\lambda_h^n|+||\bar{\rho}_h^n||_{L^\infty}+||\sigma_h^n||_{W^{1,3}}+||D_{\tau}\sigma_{h}^{n}||_{L^3}\nonumber\\
	&+|| \tilde{u}_h^{n}||_{L^\infty}^2+||\nabla \tilde{u}_h^{n}||_{L^3}^2+||{u}_h^{n}||_{L^\infty}^2+||\nabla {u}_h^{n}||_{L^3}^2\leq C.
	\end{align}}
\end{theorem}
\begin{proof}
We will prove the results by using the induction method.

(I) Case of $n=1$.

(I-1)~Through the initial data defined in the scheme \eqref{3.1}-\eqref{3.6}, we know
\begin{align*}
\theta_{\sigma h}^0=\theta_{\rho h}^0=\tilde{\theta}_{uh}^0=\theta_{uh}^0&=0,\\
||u_h^0||_{L^\infty}+||\nabla u_h^0||_{L^2}+||u_h^0||_{W^{1,3}}&\leq C,\\
||\sigma_h^0||_{L^{\infty}}+||\sigma_h^0||_{W^{1,3}}+||\rho_h^0||_{L^\infty}&\leq C,
\end{align*}
which combining with  Lemma \ref{lem.3.3} yields
\begin{equation}
||\theta_{\sigma h}^1||_{L^2}^2+||\theta_{\sigma h}^1-\theta_{\sigma h}^0||_{L^2}^2\leq C(\tau^2+h^4).\label{3.62}
\end{equation}
Then, using the inverse inequality, we   get
\begin{align}
||e_{\sigma h}^1||_{L^2}^2&\leq ||\eta_{\sigma h}^1||_{L^2}^2+||\theta_{\sigma h}^1||_{L^2}^2\leq C(\tau^2+h^4+h^6)\leq C(\tau^2+h^4),\label{3.63}\\
||\nabla e_{\sigma h}^1||_{L^2}^2&\leq ||\nabla \eta_{\sigma h}^1||_{L^2}^2+||\nabla\theta_{\sigma h}^1||_{L^2}^2\leq Ch^4+Ch^{-2}||\theta_{\sigma h}^1||_{L^2}^2\leq Ch^2,\label{4.53}\\
{||\sigma_{h}^{1}-\sigma_{h}^{0}||_{L^2}^2}&{=||\sigma_{h}^{1}-\sigma^{1}+\sigma^{1}-\sigma_{h}^{0}||_{L^2}^2\leq C(\tau^2+h^4).}
\end{align}
Thus,   $||\sigma_h^1||_{L^2}^2-||\sigma^1||_{L^2}^2\leq ||\sigma^1-\sigma_h^1||_{L^2}^2\leq C(\tau^2+h^4)$ implies that $||\sigma_h^1||_{L^2}^2\leq C+C(\tau^2+h^4)\leq C$ and $||\sigma^1+\sigma_h^1||_{L^2}^2\leq C$, which yield
\begin{equation}
||\bar{e}_{\rho h}^1||_{L^2}^2=||(\sigma^1)^2-(\sigma_h^1)^2||_{L^2}^2\leq C||e_{\sigma h}^1||_{L^2}^2\leq C(\tau^2+h^4). \label{3.64}
\end{equation}

(I-2) There exists $\tau_2>0$ and $h_0>0$, if $\tau\leq \min\{\tau_1,\tau_2\}$ and $h\leq h_0$, then $||\bar{e}_{\rho h}^1||_{L^2}^2\leq\epsilon_1<1$ with $\epsilon_1$ being a positive constant (see \eqref{3.64}),  \eqref{3.31} in Lemma \ref{lem.3.5} and \eqref{2.13} imply that

\begin{equation}
\begin{aligned}
	|1-\lambda_h^{1}|^2		&\leq \frac{C||e_{\rho h}^0||_{L^2}^2+C||\bar{e}_{\rho h}^{1}||_{L^2}^2}{1-||\bar{e}_{\rho h}^{1}||_{L^2}^2}\leq \frac{C||e_{\rho h}^0||_{L^2}^2+C||\bar{e}_{\rho h}^{1}||_{L^2}^2}{1-\epsilon_1}\\
	&\leq C(\tau^2+h^4),\label{3.65}
\end{aligned}
\end{equation}
and
\begin{align}
|\lambda_h^{1}|\leq 1+|1-\lambda_h^{1}|^2\leq C.
\end{align}

(I-3)~Using \eqref{3.32} in Lemma \ref{lem.3.5}, \eqref{3.64} and \eqref{3.65}, we can derive
\begin{equation}
||e_{\rho h}^1||_{L^2}^2\leq C(|1-\lambda_h^{1}|^2+|1-\lambda_h^{1}|^2||\bar{e}_{{\rho} h}^{1}||_{L^2}^2+||\bar{e}_{{\rho} h}^{1}||_{L^2}^2)\leq C(\tau^2+h^4).\label{3.66}
\end{equation}

Through the inverse inequality and \eqref{3.62}, we have
\begin{align}
||\sigma_h^1||_{L^\infty}&\leq ||\Pi_h\sigma^1||_{L^\infty}+Ch^{-1}||\theta_{\sigma h}^1||_{L^2}\leq C+Ch\leq C,\\
||\bar{\rho}_h^1||_{L^\infty}&\leq ||\sigma_h^1||_{L^{\infty}}^2\leq C.\\
||\sigma_h^1||_{W^{1,3}}&\leq ||\Pi_h\sigma^1||_{W^{1,3}}+||\nabla\theta_{\sigma h}^1||_{L^3}\leq C+Ch^{\frac{2}{3}}\leq C,\\
\end{align}

Taking $r_h=D_{\tau}\theta_{\sigma h}^{1}\in W_h$ in \eqref{3.21}, using $\tau\leq Ch^2$, \eqref{2.13}, \eqref{3.63}-\eqref{4.53}, we can estimate $||D_{\tau}\theta_{\sigma h}^{1}||_{L^2}^2$ as follows
\begin{equation}
\begin{aligned}
	||D_{\tau}\theta_{\sigma h}^{1}||_{L^2}^2&\leq ||\nabla \sigma^{1}||_{L^{\infty}}||e_{uh}^0||_{L^2}||D_{\tau}\theta_{\sigma h}^{1}||_{L^2}+||u_h^0||_{L^{\infty}}||\nabla e_{\sigma h}^{1}||_{L^2}||D_{\tau}\theta_{\sigma h}^{1}||_{L^2}\\
	&+\frac{1}{2}||\sigma^{1}||\!_{L^{\infty}}\!||\nabla e_{uh}^0||\!_{L^2}\!||D_{\tau}\theta_{\sigma h}^{1}||\!_{L^2}\!+\!\frac{1}{2}||\nabla u_h^0||\!_{L^{\infty}}\!||e_{\sigma h}^{1}||\!_{L^2}\!||D_{\tau}\theta_{\sigma h}^{1}||\!_{L^2}\\
	&+||R_{\sigma 1}^{1}||_{L^2}||D_{\tau}\theta_{\sigma h}^{1}||_{L^2}\\
	&\leq Ch^2||D_{\tau}\theta_{\sigma h}^{1}||_{L^2}\!\!+\!Ch||D_{\tau}\theta_{\sigma h}^{1}||_{L^2}\!\!+\!C(||\nabla\eta_{uh}^0||_{L^2}\!\!+\!0)||D_{\tau}\theta_{\sigma h}^{1}||_{L^2}\\
	&+C\tau||D_{\tau}\theta_{\sigma h}^{1}||_{L^2}\\
	&\leq Ch^2||D_{\tau}\theta_{\sigma h}^{1}||_{L^2}+Ch||D_{\tau}\theta_{\sigma h}^{1}||_{L^2}\\
	&\leq Ch^4+\frac{1}{2}||D_{\tau}\theta_{\sigma h}^{1}||_{L^2}^2+Ch^2\\
	&\leq Ch^2, \nonumber
\end{aligned}
\end{equation}
which contributes to
\begin{equation}
||D_{\tau}\theta_{\sigma h}^{1}||_{L^3}\leq Ch^{-\frac{1}{3}}||D_{\tau}\theta_{\sigma h}^{1}||_{L^2}\leq Ch^\frac{2}{3}\leq C.
\end{equation}
On the other hand, we can easily obtain
\begin{equation*}
||D_{\tau}\sigma_h^{1}||_{L^3}\leq ||D_{\tau}(\Pi_h\sigma^{1})||_{L^3}+||D_{\tau}\theta_{\sigma h}^{1}||_{L^3}\leq C+Ch^{\frac{2}{3}}\leq C.
\end{equation*}

(I-4) Employing Lemma \ref{lem.3.4} and inequalities mentioned above, we can deduce
\begin{equation}
\begin{aligned}	&||\sigma_h^{1}\tilde{\theta}_{uh}^{1}||_{L^2}^2-||\sigma_h^{0}\tilde{\theta}_{uh}^{0}||_{L^2}^2+||\sigma_h^{1}\tilde{\theta}_{uh}^{1}-\sigma_h^0\tilde{\theta}_{uh}^0||_{L^2}^2+\mu\tau||\nabla \tilde{\theta}_{uh}^{1}||_{L^2}^2\\
	&\leq C\tau(\tau^2+h^4)+C\tau||\sigma_h^{1}\tilde{\theta}_{uh}^1||_{L^2}^2.\label{3.70}
\end{aligned}
\end{equation}
There exists $\tau_3>0$, if $\tau\leq \tau*:=\min\{\tau_1,\tau_2,\tau_3\}$, then $1-C\tau>0$, thus we obtain
\begin{equation}
||\sigma_h^{1}\tilde{\theta}_{uh}^{1}||_{L^2}^2\!\!+\!||\sigma_h^{1}\tilde{\theta}_{uh}^{1}\!\!-\!\sigma_h^0\tilde{\theta}_{uh}^0||_{L^2}^2\!\!+\!\mu\tau||\nabla \tilde{\theta}_{uh}^{1}||_{L^2}^2\!\!\leq\! C\tau^3+C\tau h^4\leq C\tau(\tau^2+h^4),\nonumber
\end{equation}
which implies
\begin{equation}
||\sigma_h^1\tilde{\theta}_{uh}^{1}||_{L^2}^2 +C\tau||\nabla\tilde{\theta}_{uh}^{1}||_{L^2}^2\leq C\tau(\tau^2+h^4),\label{3.72}
\end{equation}
and
\begin{equation}
||\sigma_h^1\tilde{e}_{uh}^{1}||_{L^2}^2 +\tau||\nabla \tilde{e}_{uh}^{1}||_{L^2}^2\leq C\tau(\tau^2+h^4).\label{after3.66}
\end{equation}

(I-5)~By applying \eqref{3.33} and \eqref{2.13}, we can draw the conclusion that:
\begin{equation}
\begin{aligned}
	|1-\gamma_h^1|^2&\leq C\tau^4+C(||\sigma_h^1\tilde{e}_{uh}^1||_{L^2}^4+||e_{\sigma h}^1||_{L^2}^4+||e_{\sigma h}^0||_{L^2}^4+||\sigma_h^0\tilde{e}_{uh}^0||_{L^2}^4)\\
	&+C(||\sigma_h^0e_{uh}^0||_{L^2}^2+||\sigma_h^0\tilde{e}_{uh}^0||_{L^2}^2)\\
	&\leq C\tau(\tau^2+h^4).\label{3.73}
\end{aligned}
\end{equation}

(I-6)~Utilizing \eqref{3.72}, we derive:
\begin{equation}
\begin{aligned}
	||\sigma_h^1\theta_{uh}^{1}||_{L^2}^2 &=||\sigma_h^1(R_hu^{1}-\tilde{u}_h^{1})+\sigma_h^1(\tilde{u}_h^{1}-u_h^{1})||_{L^2}^2\\
	&\leq ||\sigma_h^1\tilde{\theta}_{uh}^{1}||_{L^2}^2+||\sigma_h^1(\tilde{u}_h^{1}-\sqrt{\gamma_{h}^{1}}\tilde{u}_h^{1})||_{L^2}^2\\
	&\leq C\tau(\tau^2+h^4)+|1-\sqrt{\gamma_{h}^{1}}|^2||\sigma_h^1\tilde{u}_h^{1}||_{L^2}^2. \label{3.74}
\end{aligned}
\end{equation}
Since $\tau= O(h^2)$, then when $h$ is sufficiently small, we can get $0\leq 1-Ch^2\leq \gamma_{h}^{1}\leq 1+Ch^2$ from \eqref{3.73}, it follows that $1+\sqrt{\gamma_{h}^{1}}$ is lower boundedness  and
\begin{equation}
|1-\sqrt{\gamma_{h}^{1}}|^2\leq \left|\frac{1-\gamma_{h}^{1}}{1+\sqrt{\gamma_{h}^{1}}}\right|^2\leq C\tau(\tau^2+h^4).\label{3.75}
\end{equation}
Noting \eqref{3.74} and
\begin{equation}
||\tilde{u}_h^{1}||_{L^2}^2\leq ||u^{1}||_{L^2}^2+||\tilde{e}_{uh}^{1}||_{L^2}^2\leq C+C\tau(\tau^2+h^4)\leq C,\nonumber
\end{equation}
we can deduce that
\begin{equation}
||\sigma_h^1\theta_{uh}^{1}||_{L^2}^2\leq C\tau(\tau^2+h^4). \nonumber
\end{equation}
Using \eqref{2.3}, \eqref{3.74}, \eqref{3.72}, \eqref{3.75}, the condition $\tau\leq Ch^2$ and inequalities

\begin{align}		
||\nabla \tilde{u}_h^{1}||_{L^2}^2\leq& 2(||\nabla R_hu^{1}||_{L^2}^2+||\nabla \tilde{\theta}_{uh}^{1}||_{L^2}^2)\leq C+Ch^{-2}||\tilde{\theta}_{uh}^{1}||_{L^2}^2\leq C,\label{62}\\
|| \tilde{u}_h^{1}||_{L^\infty}^2\leq& 2(||R_hu^{1}||_{L^\infty}^2+|| \tilde{\theta}_{uh}^{1}||_{L^\infty}^2)\leq C+Ch^{-2}||\tilde{\theta}_{uh}^{1}||_{L^2}^2\leq C,\label{63}
\end{align}
we obtain that
\begin{equation}
\begin{aligned}
	\tau||\nabla \theta_{uh}^{1}||_{L^2}^2&\leq\tau||\nabla \tilde{\theta}_{uh}^{1}||_{L^2}^2+\tau|1-\sqrt{\gamma_{h}^{1}}|^2||\nabla \tilde{u}_h^{1}||_{L^2}^2\\
	&\leq C\tau(\tau^2+h^4).\label{4.54}
\end{aligned}
\end{equation}

Therefore, it is valid that
\begin{equation}
||\sigma_h^1e_{uh}^{1}||_{L^2}^2+\tau||\nabla e_{uh}^{1}||_{L^2}^2\leq C(\tau^2+h^4).\label{after3.72}
\end{equation}
{ Similarly to \eqref{62}-\eqref{63}, thank to \eqref{2.1} and \eqref{4.54}, there hold that
\begin{align}		
	||\nabla \tilde{u}_h^{1}||_{L^2}^2+|| \tilde{u}_h^{1}||_{L^\infty}^2\leq&  C,\label{62}\\
	|| \nabla\tilde{u}_h^{1}||_{L^3}^2\leq 2(||
	\nabla R_hu^{1}||_{L^3}^2+|| \nabla\tilde{\theta}_{uh}^{1}||_{L^3}^2)\leq& C.\label{63}
	\end{align}}
	
	(II) Assuming that \eqref{3.55} to \eqref{new1} are valid for $m =0, 1, 2, \dots , n-1(1 \leq n \leq N)$, following the similar process in (I), we can prove that they hold for $m = n$, too. The proof is completed.
\end{proof}
{\begin{lemma}\label{la.3.7}
Under Assumptions of Theorem \ref{the.3.6}, it is valid for the scheme \eqref{3.1}-\eqref{3.6} , for $1\leq n\leq N$, that
\begin{align}\label{3.new3}
	||D_\tau e_{\sigma h}^{n}||_{L^2}^2\leq C(\tau+h^2).
\end{align}
\end{lemma}
\begin{proof}
Firstly,  taking $r_h=D_\tau\theta_{\sigma h}^{n+1}\in W_h$ in \eqref{3.21} yields
\begin{equation}
	\begin{aligned}
		&||D_\tau\theta_{\sigma h}^{n+1}||_{L^2}^2\\
		&= -(D_\tau\eta_{\sigma h}^{n+1},D_\tau\theta_{\sigma h}^{n+1})-(\nabla\sigma^{n+1}e_{uh}^n, D_\tau\theta_{\sigma h}^{n+1})-(u_h^n\cdot\nabla e_{\sigma h}^{n+1}, D_\tau\theta_{\sigma h}^{n+1})\\
		&-\frac{1}{2}(\sigma^{n+1}\nabla\cdot e_{uh}^n,D_\tau\theta_{\sigma h}^{n+1})-\frac{1}{2}(\nabla\cdot u_h^ne_{\sigma h}^{n+1},D_\tau\theta_{\sigma h}^{n+1})+(R_{\sigma }^{n+1}, D_\tau\theta_{\sigma h}^{n+1}).\label{3.new2}
	\end{aligned}	
\end{equation}
Thanks to Theorem \ref{the.3.6}, \eqref{3.16}, the inverse inequality and Assumption 2.1, there hold
\begin{align*}
	|-(D_\tau\eta_{\sigma h}^{n+1},D_\tau\theta_{\sigma h}^{n+1})|\leq& ||D_\tau\eta_{\sigma h}^{n+1}||_{L^2}||D_\tau\theta_{\sigma h}^{n+1}||_{L^2}\\
	\leq& \frac{1}{12}||D_\tau\theta_{\sigma h}^{n+1}||_{L^2}^2+C||D_\tau\eta_{\sigma h}^{n+1}||_{L^2}^2\\
	\leq& \frac{1}{12}||D_\tau\theta_{\sigma h}^{n+1}||_{L^2}^2+Ch^4,\\
	|-(\nabla\sigma^{n+1}e_{uh}^n, D_\tau\theta_{\sigma h}^{n+1})|\leq&C||\nabla\sigma^{n+1}||_{L^\infty}||\nabla e_{uh}^n||_{L^2}||D_\tau\theta_{\sigma h}^{n+1}||_{L^2}\\
	\leq& \frac{1}{12}||D_\tau\theta_{\sigma h}^{n+1}||_{L^2}^2+C(\tau+h^2),\\
	|-(u_h^n\cdot\nabla e_{\sigma h}^{n+1}, D_\tau\theta_{\sigma h}^{n+1})|=&|-(\nabla\cdot (e_{\sigma h}^{n+1}u_h^n)-e_{\sigma h}^{n+1}\nabla\cdot u_h^n, D_\tau\theta_{\sigma h}^{n+1})|\\
	=&|(e_{\sigma h}^{n+1}u_h^n, \nabla D_\tau\theta_{\sigma h}^{n+1})-(e_{\sigma h}^{n+1}\nabla\cdot u_h^n, D_\tau\theta_{\sigma h}^{n+1})|\\
	\leq&|| e_{\sigma h}^{n+1}||_{L^2}||u_h^n||_{L^\infty} ||\nabla D_\tau\theta_{\sigma h}^{n+1}||_{L^2}\\
	+&|| e_{\sigma h}^{n+1}||_{L^2}||\nabla u_h^n||_{L^\infty} || D_\tau\theta_{\sigma h}^{n+1}||_{L^2}\\
	\leq&Ch^{-1}(|| e_{\sigma h}^{n+1}||_{L^2}||u_h^n||_{L^\infty} || D_\tau\theta_{\sigma h}^{n+1}||_{L^2}\\
	+&|| e_{\sigma h}^{n+1}||_{L^2}||\nabla u_h^n||_{L^2} || D_\tau\theta_{\sigma h}^{n+1}||_{L^2})\\
	\leq& \frac{1}{12}||D_\tau\theta_{\sigma h}^{n+1}||_{L^2}^2+C(\tau+h^2),\\
	|-\frac{1}{2}(\sigma^{n+1}\nabla\cdot e_{uh}^n,D_\tau\theta_{\sigma h}^{n+1})|
	\leq& C||\sigma^{n+1}||_{L^\infty} ||\nabla e_{uh}^n||_{L^2}|| D_\tau\theta_{\sigma h}^{n+1}||_{L^2}\\
	\leq& C||\sigma^{n+1}||_{L^\infty} ||\nabla e_{uh}^n||_{L^2}|| D_\tau\theta_{\sigma h}^{n+1}||_{L^2}\\
	\leq& \frac{1}{12}||D_\tau\theta_{\sigma h}^{n+1}||_{L^2}^2+(\tau+h^2),\\
	|-\frac{1}{2}(\nabla\cdot u_h^ne_{\sigma h}^{n+1},D_\tau\theta_{\sigma h}^{n+1})|\leq&\frac{1}{2}||\nabla u_h^n||_{L^\infty}||e_{\sigma h}^{n+1}||_{L^2}||D_\tau\theta_{\sigma h}^{n+1}||_{L^2}\\
	\leq&\frac{1}{2}h^{-1}||\nabla  u_h^n||_{L^2}||e_{\sigma h}^{n+1}||_{L^2}||D_\tau\theta_{\sigma h}^{n+1}||_{L^2}\\
	\leq& \frac{1}{12}||D_\tau\theta_{\sigma h}^{n+1}||_{L^2}^2+C(\tau+h^2),
\end{align*}
\begin{align*}
	|(R_{\sigma }^{n+1}, D_\tau\theta_{\sigma h}^{n+1}|\leq&||R_{\sigma }^{n+1}||_{L^2}||D_\tau\theta_{\sigma h}^{n+1})||_{L^2}\\
	\leq& \frac{1}{12}||D_\tau\theta_{\sigma h}^{n+1}||_{L^2}^2+C\tau^2.
\end{align*}
Putting above estimates into \eqref{3.new2}, we get
\begin{align*}
	||D_\tau\theta_{\sigma h}^{n+1}||_{L^2}^2\leq C(\tau+h^2),
\end{align*}
which combining with the fact $e_{\sigma h}^n=\eta_{\sigma h}^n+\theta_{\sigma h}^n$ and the triangle inequality yields \eqref{3.new3}. The proof is completed.
\end{proof}
\begin{lemma}\label{la.3.8}
Under Assumptions of Theorem \ref{the.3.6}, it is valid for the scheme \eqref{3.1}-\eqref{3.6} , for $1\leq n\leq N$, that
\begin{align}\label{3.new4}
	\tau\sum_{i=1}^m||D_\tau (\sigma_h^{i}\tilde{\theta}_{uh}^{i})||_{L^2}^2\leq C(\tau+h^2).
\end{align}
\end{lemma}
\begin{proof}
Taking $(v_h,q_h)=(2\tau D_\tau\tilde{\theta}_{uh}^{n+1}, 0)$ in \eqref{3.22} and noting the equalities
\begin{align*}
	&\mu(\nabla\tilde{\theta}_{uh}^{n+1}, 2\tau \nabla D_\tau\tilde{\theta}_{uh}^{n+1})-(\nabla\cdot 2\tau D_\tau\tilde{\theta}_{uh}^{n+1}, \theta_{ph}^{n+1})+(\nabla\cdot \tilde{\theta}_{uh}^{n+1}, 0)\\
	&= \mu||\nabla\tilde{\theta}_{uh}^{n+1}||_{L^2}^2-\mu||\nabla\tilde{\theta}_{uh}^{n}||_{L^2}^2+\mu||\nabla(\tilde{\theta}_{uh}^{n+1}-\tilde{\theta}_{uh}^{n})||_{L^2}^2,
\end{align*}
\begin{align*}
	&(\sigma_h^{n+1}D_{\tau}(\sigma_h^{n+1}\tilde{\theta}_{uh}^{n+1}), 2\tau D_\tau\tilde{\theta}_{uh}^{n+1})\\
	&=2\tau (D_{\tau}(\sigma_h^{n+1}\tilde{\theta}_{uh}^{n+1}), \sigma_h^{n+1}D_\tau\tilde{\theta}_{uh}^{n+1})\\
	&=2\tau (D_{\tau}(\sigma_h^{n+1}\tilde{\theta}_{uh}^{n+1}), D_\tau(\sigma_h^{n+1}\tilde{\theta}_{uh}^{n+1})-\tilde{\theta}_{uh}^{n}D_\tau\sigma_h^{n+1})\\
	&=2\tau ||D_{\tau}(\sigma_h^{n+1}\tilde{\theta}_{uh}^{n+1})||_{L^2} -2\tau(D_\tau(\sigma_h^{n+1}\tilde{\theta}_{uh}^{n+1}),\tilde{\theta}_{uh}^{n}D_\tau\sigma_h^{n+1}),
\end{align*}
we obtain
\begin{align}
	&2\tau ||D_{\tau}(\sigma_h^{n+1}\tilde{\theta}_{uh}^{n+1})||_{L^2}+\mu(||\nabla\tilde{\theta}_{uh}^{n+1}||_{L^2}^2-||\nabla\tilde{\theta}_{uh}^{n}||_{L^2}^2+||\nabla(\tilde{\theta}_{uh}^{n+1}-\tilde{\theta}_{uh}^{n})||_{L^2}^2)\nonumber\\
	=&2\tau(D_\tau(\sigma_h^{n+1}\tilde{\theta}_{uh}^{n+1}),\tilde{\theta}_{uh}^{n}D_\tau\sigma_h^{n+1})\!\!+\!\!2\tau (R_{u}^{n+1}, D_\tau\tilde{\theta}_{uh}^{n+1})\!\!-\!\!2\tau\sum\limits_{i=1}^{9}(Y_{i}^{n+1},  D_\tau\tilde{\theta}_{uh}^{n+1}),\label{n3.22}
\end{align}
where $Y_{i}^{n+1}$ are defined below \eqref{3.22}. First, due to Theorem \ref{the.3.6}, Lemma \ref{lem.3.2} and the inverse inequality, there hold
\begin{align*}
	&|2\tau(D_\tau(\sigma_h^{n+1}\tilde{\theta}_{uh}^{n+1}),\tilde{\theta}_{uh}^{n}D_\tau\sigma_h^{n+1})|\\
	&\leq 2\tau||D_\tau(\sigma_h^{n+1}\tilde{\theta}_{uh}^{n+1})||_{L^2}||\tilde{\theta}_{uh}^{n}||_{L^\infty}||D_\tau\sigma_h^{n+1}||_{L^2}\\
	&\leq C\tau||D_\tau(\sigma_h^{n+1}\tilde{\theta}_{uh}^{n+1})||_{L^2}h^{-1}||\tilde{\theta}_{uh}^{n}||_{L^2}||D_\tau\sigma_h^{n+1}||_{L^2}\\
	&\leq C\tau||D_\tau(\sigma_h^{n+1}\tilde{\theta}_{uh}^{n+1})||_{L^2}h^{-1}||\nabla \tilde{\theta}_{uh}^{n}||_{L^2}||D_\tau\sigma_h^{n+1}||_{L^2}\\
	&\leq \frac{\tau}{16}||D_\tau(\sigma_h^{n+1}\tilde{\theta}_{uh}^{n+1})||_{L^2}^2+C\tau(\tau+h^2),
\end{align*}

\begin{align*}
	&|2\tau (R_{u}^{n+1}, D_\tau\tilde{\theta}_{uh}^{n+1})|\\
	&=|2 (R_{u}^{n+1}, \tilde{\theta}_{uh}^{n+1}-\tilde{\theta}_{uh}^{n})|\\
	&\leq C||R_{u}^{n+1}||_{L^2} ||\nabla(\tilde{\theta}_{uh}^{n+1}-\tilde{\theta}_{uh}^{n})||_{L^2}\\
	&\leq \frac{\mu}{16}||\nabla(\tilde{\theta}_{uh}^{n+1}-\tilde{\theta}_{uh}^{n})||_{L^2}^2+C\tau^2.
\end{align*}
Second, applying the equality \eqref{cue} and Theorem \ref{the.3.6} again, we have
\begin{align*}
	|-2\tau(Y_{1}^{n+1},  D_\tau\tilde{\theta}_{uh}^{n+1})|=&|-2(Y_{1}^{n+1},  \tilde{\theta}_{uh}^{n+1}-\tilde{\theta}_{uh}^{n})|\\
	=&|-2(e_{\sigma h}^{n+1}D_{\tau}(\sigma^{n+1}u^{n+1}),  \tilde{\theta}_{uh}^{n+1}-\tilde{\theta}_{uh}^{n})|\\
	\leq &2||e_{\sigma h}^{n+1}||_{L^2}||D_{\tau}(\sigma^{n+1}u^{n+1})||_{L^\infty}||\tilde{\theta}_{uh}^{n+1}-\tilde{\theta}_{uh}^{n}||_{L^2}\\
	\leq&C||e_{\sigma h}^{n+1}||_{L^2}||(\sigma^{n+1}u^{n+1})_t+O(\tau)||_{L^\infty}\cdot \\
	&||\nabla(\tilde{\theta}_{uh}^{n+1}-\tilde{\theta}_{uh}^{n})||_{L^2}\\
	\leq& \frac{\mu}{16}||\nabla(\tilde{\theta}_{uh}^{n+1}-\tilde{\theta}_{uh}^{n})||_{L^2}^2+C(\tau^2+h^4),\\
	|-2\tau(Y_{3}^{n+1},  D_\tau\tilde{\theta}_{uh}^{n+1})|=&|-2(\sigma_h^{n+1}D_{\tau}(\sigma_h^{n+1}\eta_{uh}^{n+1}),  \tilde{\theta}_{uh}^{n+1}-\tilde{\theta}_{uh}^{n})|\\
	\leq& 2||\sigma_h^{n+1}||_{L^\infty}||D_{\tau}(\sigma_h^{n+1}\eta_{uh}^{n+1})||_{L^2}||\tilde{\theta}_{uh}^{n+1}-\tilde{\theta}_{uh}^{n}||_{L^2}\\
	\leq& C||\sigma_h^{n+1}||_{L^\infty}||\sigma_h^{n+1}D_{\tau}\eta_{uh}^{n+1}+\eta_{uh}^{n+1}D_{\tau}\sigma_h^{n+1}||_{L^2}\cdot\\
	&~~~||\nabla(\tilde{\theta}_{uh}^{n+1}-\tilde{\theta}_{uh}^{n})||_{L^2}\\
	\leq& 2||\sigma_h^{n+1}||_{L^\infty}(||\sigma_h^{n+1}D_{\tau}\eta_{uh}^{n+1}||_{L^2}+||\eta_{uh}^{n+1}D_{\tau}\sigma_h^{n+1}||_{L^2})\cdot\\
	&||\nabla(\tilde{\theta}_{uh}^{n+1}-\tilde{\theta}_{uh}^{n})||_{L^2}\\
	\leq& 2||\sigma_h^{n+1}||_{L^\infty}(||\sigma_h^{n+1}||_{L^\infty}||D_{\tau}\eta_{uh}^{n+1}||_{L^2}\\
	&+||\eta_{uh}^{n+1}||_{L^2}||D_{\tau}\sigma_h^{n+1}||_{L^\infty})(||\nabla(\tilde{\theta}_{uh}^{n+1}-\tilde{\theta}_{uh}^{n})||_{L^2})\\
	\leq& \frac{\mu}{16}||\nabla(\tilde{\theta}_{uh}^{n+1}-\tilde{\theta}_{uh}^{n})||_{L^2}^2+C(\tau^2+h^4),\\
	|-2\tau(Y_{4}^{n+1},  D_\tau\tilde{\theta}_{uh}^{n+1})|=&|-2(e_{\rho h}^{n}(u^n\cdot\nabla)u^{n+1},  \tilde{\theta}_{uh}^{n+1}-\tilde{\theta}_{uh}^{n})|\\
	\leq&C||e_{\rho h}^{n}||_{L^2}||(u^n\cdot\nabla)u^{n+1}||_{L^\infty}||\nabla(\tilde{\theta}_{uh}^{n+1}-\tilde{\theta}_{uh}^{n})||_{L^2}^2\\
	\leq&\frac{\mu}{16}||\nabla(\tilde{\theta}_{uh}^{n+1}-\tilde{\theta}_{uh}^{n})||_{L^2}^2+C(\tau^2+h^4),\\
	|-2\tau(Y_{5}^{n+1},  D_\tau\tilde{\theta}_{uh}^{n+1})|=&|-2(\rho_h^{n}(e_{uh}^n\cdot\nabla)u^{n+1},  \tilde{\theta}_{uh}^{n+1}-\tilde{\theta}_{uh}^{n})|\\
	\leq&2||\lambda_h^{n}\sigma_h^n\nabla u^{n+1}||_{L^\infty}||\sigma_h^ne_{uh}^n||_{L^2}||\nabla(\tilde{\theta}_{uh}^{n+1}-\tilde{\theta}_{uh}^{n})||_{L^2}^2\\
	\leq& \frac{\mu}{16}||\nabla(\tilde{\theta}_{uh}^{n+1}-\tilde{\theta}_{uh}^{n})||_{L^2}^2+C(\tau^2+h^4),
\end{align*}
\begin{align*}
	|-2\tau(Y_{6}^{n+1},  D_\tau\tilde{\theta}_{uh}^{n+1})|&=|-2(\rho_h^{n}(u_h^n\cdot\nabla)\tilde{e}_{uh}^{n+1},  \tilde{\theta}_{uh}^{n+1}-\tilde{\theta}_{uh}^{n})|\\
	=&|2(\nabla\cdot(\rho_h^{n}u_h^n\otimes\tilde{e}_{uh}^{n+1})\!\!-\!(\nabla\cdot(\rho_h^{n}u_h^n))\tilde{e}_{uh}^{n+1},\tilde{\theta}_{uh}^{n+1}\!\!-\!\tilde{\theta}_{uh}^{n})|\\
	=&2|(\rho_h^{n}u_h^n\otimes\tilde{e}_{uh}^{n+1},\nabla(\tilde{\theta}_{uh}^{n+1}-\tilde{\theta}_{uh}^{n}))\\
	&-(\nabla\cdot(\rho_h^{n}u_h^n)\tilde{e}_{uh}^{n+1},\tilde{\theta}_{uh}^{n+1}-\tilde{\theta}_{uh}^{n})|\\
	\leq& C||\lambda_h^{n}\sigma_h^nu_h^n||_{L^\infty}||\sigma_h^n\tilde{e}_{uh}^{n+1}||_{L^2}||\nabla(\tilde{\theta}_{uh}^{n+1}-\tilde{\theta}_{uh}^{n})||_{L^2}\\
	&+C(||\sigma_h^n\nabla\lambda_h^1\!\cdot\! u_h^n||_{L^3}+||\lambda_h^n\nabla\sigma_h^1\cdot\! u_h^n||_{L^3}\\
	&+||\sigma_h^n\sigma_h^1\nabla\cdot\!u_h^n||_{L^3})\cdot||\sigma_h^n\tilde{e}_{uh}^{n+1}||_{L^2}||\tilde{\theta}_{uh}^{n+1}-\tilde{\theta}_{uh}^{n}||_{L^6}\\
	\leq& C(||\sigma_h^{n+1}\tilde{e}_{uh}^{n+1}||_{L^2}+||\sigma_h^{n+1}-\sigma_h^{n}||_{L^2}||\tilde{e}_{uh}^{n+1}||_{L^2})\cdot\\
	&||\nabla(\tilde{\theta}_{uh}^{n+1}-\tilde{\theta}_{uh}^{n})||_{L^2}[||\lambda_h^{n}\sigma_h^nu_h^n||_{L^\infty}\\
	&+C(||\sigma_h^n\nabla\lambda_h^1\!\cdot\! u_h^n||_{L^3}+||\lambda_h^n\nabla\sigma_h^1\cdot u_h^n||_{L^3}\\
	&+||\sigma_h^n\sigma_h^1\nabla\cdot u_h^n||_{L^3}\!)]\\
	\leq& \frac{\mu}{32}||\nabla(\tilde{\theta}_{uh}^{n+1}-\tilde{\theta}_{uh}^{n})||_{L^2}^2+C(\tau^2+h^4),\\
	|-2\tau(Y_{7}^{n+1},  D_\tau\tilde{\theta}_{uh}^{n+1})|=&|-(u^{n+1}\nabla\cdot(e_{\rho h}^{n}u^n),  \tilde{\theta}_{uh}^{n+1}\!\!-\!\!\tilde{\theta}_{uh}^{n})|\\
	=&|\nabla\!\cdot\!((e_{\rho h}^{n}u^n)\otimes u^{n+1})\!\!-\!\!((e_{\rho h}^{n}u^n)\cdot\nabla)u^{n+1},\tilde{\theta}_{uh}^{n+1}\!\!-\!\!\tilde{\theta}_{uh}^{n})|\\
	=&|((e_{\rho h}^{n}u^n)\otimes u^{n+1}, \nabla( \tilde{\theta}_{uh}^{n+1}-\tilde{\theta}_{uh}^{n}))\\
	&-(((e_{\rho h}^{n}u^n)\cdot\nabla)u^{n+1},  \tilde{\theta}_{uh}^{n+1}-\tilde{\theta}_{uh}^{n})|\\
	\leq&||e_{\rho h}^{n}||_{L^2}||u^n||_{L^\infty}||u^{n+1}||_{L^\infty}||\nabla(\tilde{\theta}_{uh}^{n+1}-\tilde{\theta}_{uh}^{n})||_{L^2}\\
	&+||e_{\rho h}^{n}||_{L^2}||u^n||_{L^\infty}||\nabla u^{n+1}||_{L^\infty} ||\nabla(\tilde{\theta}_{uh}^{n+1}-\tilde{\theta}_{uh}^{n})||_{L^2}\\
	\leq& \frac{\mu}{32}||\nabla(\tilde{\theta}_{uh}^{n+1}-\tilde{\theta}_{uh}^{n})||_{L^2}^2+C(\tau^2+h^4),\\
	|-2\tau(Y_{8}^{n+1},  D_\tau\tilde{\theta}_{uh}^{n+1})|=&|-(u^{n+1}\nabla\cdot(\rho_h^{n}e_{uh}^n),  \tilde{\theta}_{uh}^{n+1}-\tilde{\theta}_{uh}^{n})|\\
	=&|((\rho_h^{n}e_{uh}^n)\otimes u^{n+1},\nabla(\tilde{\theta}_{uh}^{n+1}-\tilde{\theta}_{uh}^{n}))\\
	&-(((\rho_h^{n}e_{uh}^n)\cdot\nabla) u^{n+1},\tilde{\theta}_{uh}^{n+1}-\tilde{\theta}_{uh}^{n})|\\
	\leq& \frac{\mu}{32}||\nabla(\tilde{\theta}_{uh}^{n+1}-\tilde{\theta}_{uh}^{n})||_{L^2}^2+C(\tau^2+h^4),\\
	|-2\tau(Y_{9}^{n+1},  D_\tau\tilde{\theta}_{uh}^{n+1})|=&|-(\tilde{e}_{uh}^{n+1}\nabla\cdot(\rho_h^{n}u_h^n),  \tilde{\theta}_{uh}^{n+1}-\tilde{\theta}_{uh}^{n})|\\
	\leq& \frac{\mu}{32}||\nabla(\tilde{\theta}_{uh}^{n+1}-\tilde{\theta}_{uh}^{n})||_{L^2}^2+C(\tau^2+h^4).
\end{align*}
Finally, using Theorem \ref{the.3.6} and following the similar process in proving Lemma \ref{lemma4.3}, we obtain
\begin{align*}
	|-2\tau(Y_{2}^{n+1},  D_\tau\tilde{\theta}_{uh}^{n+1})|=&|-2(\sigma_h^{n+1}D_{\tau}(e_{\sigma h}^{n+1}u^{n+1}),  \tilde{\theta}_{uh}^{n+1}-\tilde{\theta}_{uh}^{n})|\\
	\leq& \frac{\mu}{32}||\nabla(\tilde{\theta}_{uh}^{n+1}-\tilde{\theta}_{uh}^{n})||_{L^2}^2+C(\tau^2+h^4).\\
\end{align*}
Substituting these estimates into \eqref{n3.22} and taking a summation with respect to $n$, we arrive at \eqref{3.new4}.
\end{proof}
\begin{theorem}\label{the.3.9}
Under Assumptions of Theorem \ref{the.3.6}, it is valid for the scheme \eqref{3.1}-\eqref{3.6} , for $1\leq n\leq N$, that
\begin{align}\label{3.new6}
	\tau\sum_{i=1}^m||e_{ph}^{i}||_{L^2}^2\leq C(\tau+h^2).
\end{align}
\begin{proof}
	Using \eqref{insu} and   \eqref{3.22}, we get
	\begin{align*}
		&\beta_h||\theta_{ph}^{n+1}||_{L^2}\\
		\leq&\!\!\!\sup\limits_{v_h\in V_hv_h\neq0}\frac{(\nabla\cdot v_h,\theta_{ph}^{n+1})}{||\nabla v_h||_{L^2}}\nonumber\\
		\leq&\!\!\!\sup\limits_{v_h\in V_hv_h\neq0}\!\!\frac{(\sigma_h^{n+1}D_{\tau}(\sigma_h^{n+1}\tilde{\theta}_{uh}^{n+1}), v_h)\!\!+\!\!\mu(\nabla\tilde{\theta}_{uh}^{n+1}, \nabla v_h)\!\!-\!\!(R_{u}^{n+1}, v_h)\!\!+\!\!\sum\limits_{i=1}^{9}(Y_{i}^{n+1}, v_h)}{||\nabla v_h||_{L^2}}.
	\end{align*}
	Due to the estimates in the proof of Lemma \ref{la.3.8} and \eqref{3.new4}, there holds
	\begin{align*}
		\tau\sum_{i=1}^m||\theta_{ph}^{n+1}||_{L^2}^2\leq C(\tau+h^2),
	\end{align*}
	which combining with the fact $e_{ph}^{n}=\eta_{ph}^{n}+\theta_{ph}^{n}$ and the triangle inequality yields \eqref{3.new6}. The proof is completed.
\end{proof}
\end{theorem}
\begin{remark}
Although the error estimate for the pressure proved in Theorem \ref{the.3.9} is lower than that for other functions provided in Theorem  \ref{the.3.6},  it is the same  as that in \cite{bib16} which is the best result for the Navier-Stokes equations with variable density in the reference.
\end{remark}}

\section{Numerical Results}\label{sec.5}
In this section, we will show some numerical examples to demonstrate the convergence orders  and the efficiency of the proposed scheme. {All simulations in the following are implemented by using FreeFEM \cite{hecht}.}
\subsection{Convergence order}
Firstly, we verify the convergence order of the proposed scheme. Let the domain $\Omega$ be a unit circle and the analytical solution as \cite{bib26}
\begin{align*}
\rho(x, y, t)&=2+x\cos(\sin(t))+y\sin(\sin(t)), \\
u(x, y, t)&=(-y\cos(t), x\cos(t))^{\top}, \\
p(x, y, t)&=\sin(x)\sin(y)\sin(t).
\end{align*}
With $\mu = 0.1 $ and the time step $\tau=\frac{1}{2^i},~i=3,4,5,6,7$, we collect the numerical results in Tables 1 and 2, from which we can see that the expectant convergence orders are got for all tested cases.

\subsection{Property-preserving test}
In this part, we test the property-preserving of the proposed scheme through two examples, which includes evolutions of the density, energy, mass (before recovery and after recovery), $\lambda_h^{n+1}$, $\gamma_h^{n+1}$ and differences in the energy identical-relation with the body force $f=0$ and $f\neq 0$, respectively.

Define the differences between two sides of the energy identical-relation in Theorem \ref{theorem3.1} as
\begin{equation*}
D_E^n=\left|E_h^{n+1}-E_h^{n}+\mu\tau\int_{\Omega}|\nabla \tilde{u}_h^{n+1}|^2\mathrm{d}\mathbf{x}-\tau\int_{\Omega}f^{n+1}\tilde{u}_h^{n+1}\mathrm{d}\mathbf{x}\right|.
\end{equation*}
Setting the time step {$\tau=0.001$,} the mesh size $h=0.05$, the finial time {$T=10$} and the domain $\Omega=\left(0, 1\right)^2$ with homogenous Dirichlet boundary conditions on $\partial\Omega$, we firstly test the case with the body force  $f=0$ and the initial data $\rho_0(\mathbf{x})=1, u_0(\mathbf{x})=(10x^2(x-1)^2y(y-1)(2y-1), -10x(x-1)(2x-1)y^2(y-1)^2)^{\top}$. It is easy to check that $u_0$ satisfies the homogenous Dirichlet boundary conditions and $\nabla\cdot u_0=0$.
The evolutions of the density, mass, energy,  $\lambda_h^{n+1}$, $\gamma_h^{n+1}$ and $D_E^n$ for different viscosities ($\mu=0.005, 0.002, 0.001$) are shown in Figure 1, from which we can see that  the  density remains positive, the mass after recovery is always conserved, the  energy is dissipative, $\lambda_h^{n+1}$ and $\gamma_h^{n+1}$ are both very close to 1. Moreover, we can see that  the differences $D_E^n$ between two sides of the energy identical-relation are  close to $0$. These suggests that the properties are preserved very well, which is consistent with the theoretical prediction deduced above.

Then, with the same  computational environment as that in the above but replacing the body force with $f=((2+x+y)\cos(t),(2+x+y)\sin(t))^{\top}$, we investigate the evolution of density, mass, energy,  $\lambda_h^{n+1}$, $\gamma_h^{n+1}$  and $D_E^n$ for various viscosities ($\mu=0.01, 0.005, 0.002$) and $T=20$. The simulations are presented in Figure 2. Similar results as the above are obtained for the numerical density and mass, which obey the properties derived in Theorem \ref{theorem3.1}. But the energy is not dissipative in this case, which forms a quasi-periodic evolution due to the periodic body force $f$.  Another observation is that, although the differences  $D_E^n$ between two sides of the energy identical-relation are also close to $0$, they almost captures the varying period of the energy, which indicates that the energy identical-relation holds, too. All of these confirms the predictions derived in Theorem \ref{theorem3.1}.

\subsection{Back-step flow}
In this {subsection}, we apply the proposed scheme to the back-step flow. With the boundary condition set in Figure 3,  taking {$\rho_0(\mathbf{x})=1, \rho|_{inflow}(\mathbf{x},t)=1$}, $u_0(\mathbf{x})=0,$ $\mu=0.01$ and $\tau=0.01$, we show the simulation results in Figures 4-6. From the results we can see that, as the time develops, the vortex appears and becomes more and more larger near the step,  which is good agreement with that in the references \cite{bib36}.

\subsection{Flow around a circular cylinder}
In this {subsection}, we apply the proposed finite element scheme to the flow around a circular cylinder in this {subsection}. The domain is defined as  $\Omega\in(0,6)\times(0,1)$ with  no-slip boundary conditions being imposed to the top and the bottom of the channel as well as the surface of the cylinder, a circle with the radius being $0.15$ centers at $(x,y)=(1,0.5)$, and the initial velocity $u_0(\mathbf{x})=0$. For the simulation parameters, we set $\mu=\frac{1}{300}, \tau=0.01$ and the inflow boundary condition is prescribed  as $u_1(\mathbf{x},t)=6y(1-y), u_2(\mathbf{x},t)=0$. While  we impose the condition $-pI+\frac{\partial u}{\partial n}=0$ on the outlet, where $I$ is the unit matrix of $2\times 2$. {First, we investigate the problem with a constant density, i.e.,}  $\rho_0(\mathbf{x})=1,  \rho|_{inflow}(\mathbf{x},t)=1$. The contour plots for the velocity components $u_1, u_2$ and the pressure $p$ are presented in Figures 7-9. At the beginning, both velocity and pressure are almost symmetric with respect to the line $y=0.5$ (when $t=3$). But as the time develops, the turbulence will appear  and get obviously  (when $t=7$) after the flow past through the circle. But their values keep symmetric with respect to the line $y=0.5$ before the circle. These are similar to that in  \cite{bib35}. {Then, we study the case with a variable density. With the same computational parameters as above but replacing the density with  $\rho_0(\mathbf{x})=1+\sin(y),  \rho|_{inflow}(\mathbf{x},t)=1+\sin(y)$, we present the velocities ($u_{1h}^n$ and $u_{2h}^n$) and the pressure in Figures 10-12. We can find that, due to the variable density, the symmetries of the velocities and the pressure are lost  from the beginning ($t=0.5$) compared with the case with a constant density. And the turbulence is very obvious at $t=2$, which is much earlier than the problem with a constant density ($t=7$). Moreover, we display the development of the density with respect to the time in Figure 13. It can be saw that, although small numerical oscillation appears due to the hyperbolic property of the density equation, the proposed scheme can capture the distribution of the variable density in all tested times.}  All of these confirm the efficiency of the proposed scheme.

\section{Conclusions}\label{sec.6}
A first order fully discrete finite element scheme which maintains mass conservation, positivity  and energy identical-relation preserving for the Navier-Stokes equations with variable density is studied in this paper. The error estimates are also proved, which are verified through some examples. {Due to the hyperbolic property of the density equation, there are  small numerical oscillation in the numerical density, which may be eliminated by using the least square finite element method \cite{bib15,WLZ2024} and the discontinuous Galerkin method \cite{bib24}.}  But there are some technique problems in the error estimate when extending this idea to {these methods and} the higher-order scheme preserving the property. Moreover, the property-preserving schemes and their error estimates for the Navier-Stokes equations with variable density  coupled with other fields, such as the electric-field (see, e.g., \cite{PHP2021,WLZ2024}) and the magnetic-field (see, e.g., \cite{ZFS2024}) are also very interesting. All of these will be considered in future.

{\section*{Acknowledgements}
We would like to thank the editor and referees for their many helpful comments
and suggestions.}

\addcontentsline{toc}{section}{References}
\bibliographystyle{plain}
\bibliography{ref.bib}
\newpage

\section*{Statements and Declarations}
\textbf{Funding}
The corresponding author is partially supported by the Natural Science Foundation of Chongqing (No. CSTB2024NSCQ-MSX0221) and  the third author is partially supported by the Natural Science Foundation of China (Nos. 12571417, 12271082, 62231016).\\
\\
\textbf{Competing Interests}
The authors have no relevant financial or non-financial interests to disclose.\\
\\
\textbf{Author Contributions}
Fan Yang: Implementation, Numerical analysis, Writing,  Editing. Haiyun Dong: Method, Editing.
Maojun Li: Implementation, Editing.  Kun Wang:  Method, Numerical analysis, Writing, Editing.\\
\\
\textbf{Data Availability}
Enquiries about data availability should be directed to
the authors.

\newpage
\begin{table}\label{t1}
	\caption{Convergence orders of the proposed scheme.}
	\centering
	\begin{tabular}{ccccccc}
		\hline
		$\tau=h^2$ &$||u-u_h^N||_{L^2}$ & Order&  $||\rho-\rho_h^N||_{L^2}$ & Order&$||p-p_h^N||_{L^2}$ & Order   \\
		\hline
		$1/8$&2.7203e-2&--  &4.9852e-2&--   &4.9851e-2&--   \\
		$1/16$&1.2849e-2&1.0821&2.8868e-2&0.7882&3.2805e-2&0.6037\\
		$1/32$&6.1064e-3&1.0733&1.3717e-2&1.0735&1.7208e-2&0.9309\\
		$1/64$&2.9529e-3&1.0482&7.1024e-3&0.9496&8.7426e-3&0.9769\\
		$1/128$&1.4414e-3&1.0346&3.5811e-3&0.9879&4.6256e-3&0.9184\\\hline
	\end{tabular}	
\end{table}

\begin{table}\label{t2}
	\caption{Convergence orders of the recovery factors.}
	\centering
	\begin{tabular}[b]{ccccc}
		\hline
		$\tau=h^2$ & $|1-\lambda_h^N|$ & Order&  $|1-\gamma_h^N|$ & Order \\
		\hline
		$1/8$&1.2613e-3&--  &1.3977e-1&-- \\
		$1/16$&6.0375e-4&1.0628&6.5406e-2&1.0956\\
		$1/32$&2.8595e-4&1.0782&3.1613e-2&1.0489\\
		$1/64$&1.5208e-4&0.9109&1.5539e-2&1.0246\\
		$1/128$&8.2399e-5&0.8841&7.0300e-3&1.0124\\\hline	
	\end{tabular}	
\end{table}

\begin{figure}
	\centering
	\begin{subfigure}{0.4\textwidth}
		\includegraphics[width=1\textwidth]{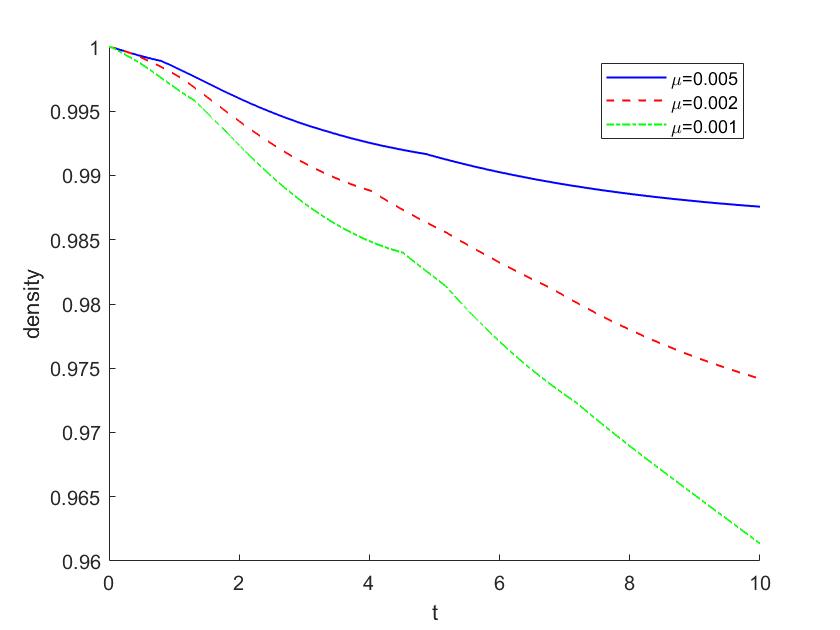}
		\caption{Minimum of $\rho_h^{n+1}$}
	\end{subfigure}
	\begin{subfigure}{0.4\textwidth}
		\includegraphics[width=1\textwidth]{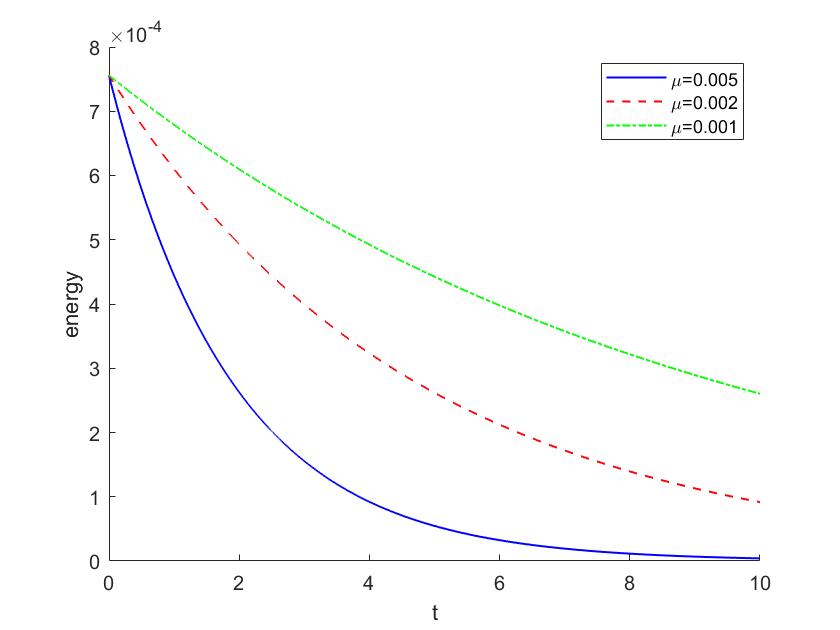}
		\caption{Evolution of  energy $E_h^{n+1}$}
	\end{subfigure}
	\begin{subfigure}{0.4\textwidth}
		\includegraphics[width=1\textwidth]{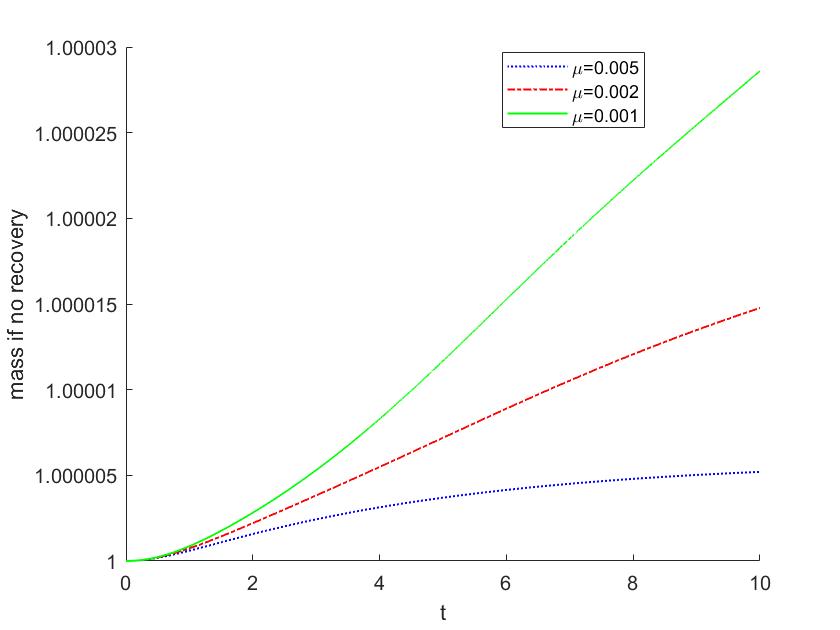}
		\caption{{Evolution of mass without recovery}}
	\end{subfigure}
	\begin{subfigure}{0.4\textwidth}
		\includegraphics[width=1\textwidth]{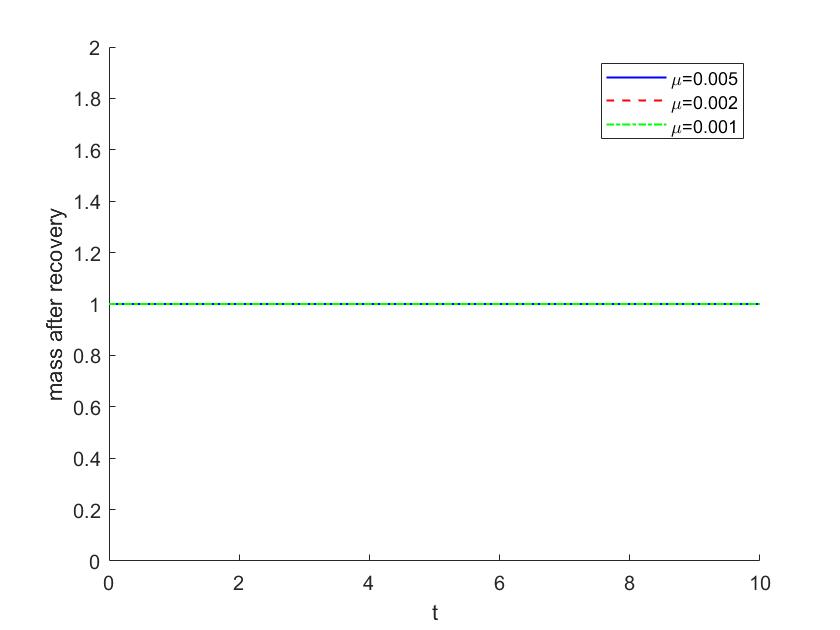}
		\caption{Evolution of mass with recovery}
	\end{subfigure}
	\begin{subfigure}{0.4\textwidth}
		\includegraphics[width=1\textwidth]{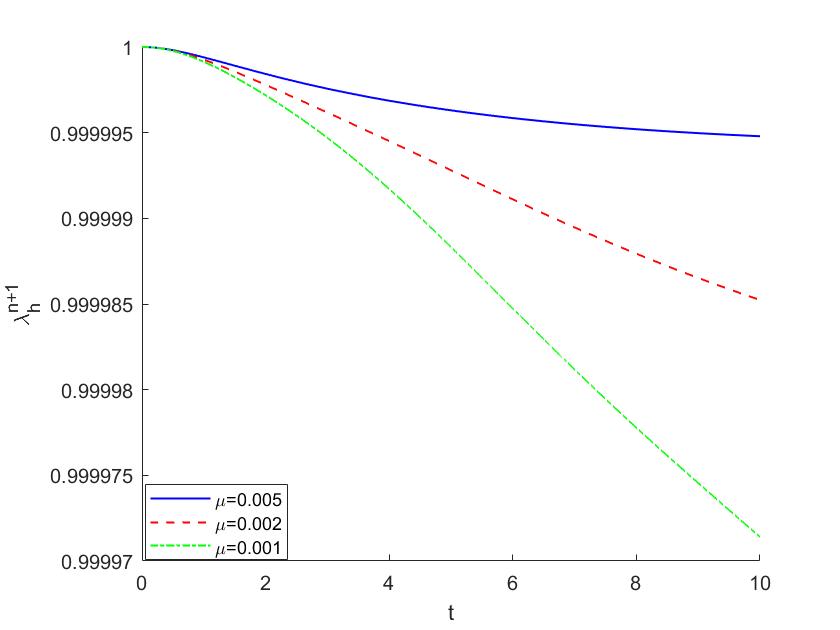}
		\caption{{Evolution of $\lambda_h^{n+1}$}}
	\end{subfigure}
	\begin{subfigure}{0.4\textwidth}
		\includegraphics[width=1\textwidth]{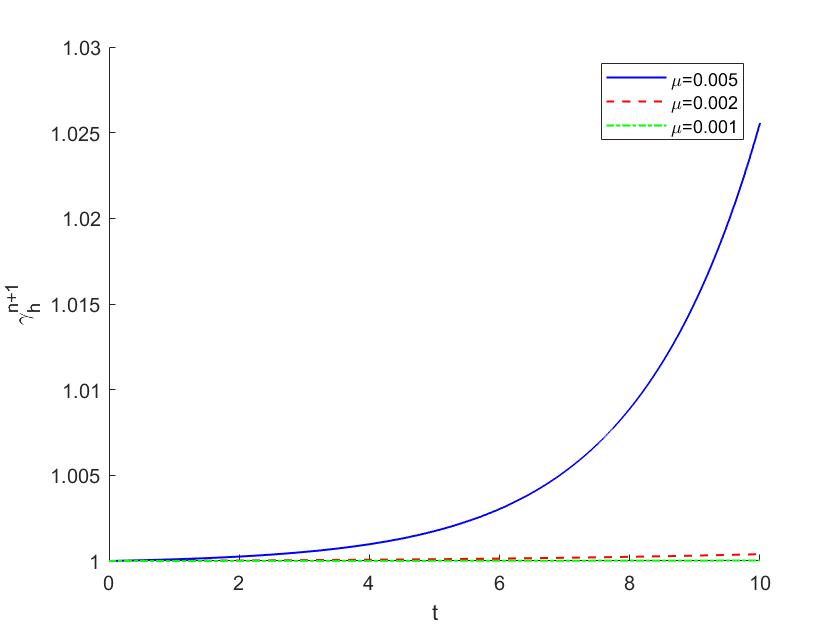}
		\caption{{Evolution of $\gamma_h^{n+1}$}}
	\end{subfigure}
	\begin{subfigure}{0.4\textwidth}
		\includegraphics[width=1\textwidth]{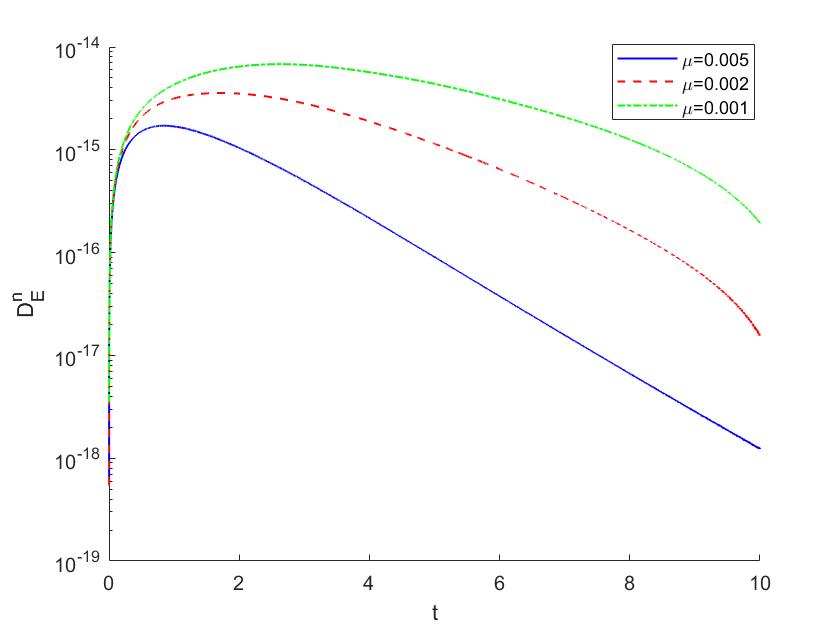}
		\caption{Evolution of $D_E^n$}
	\end{subfigure}
	\caption{Evolutions of the density, energy, {mass without recovery, mass with recovery, $\lambda_h^{n+1}$, $\gamma_h^{n+1}$} and $D_E^n$ with $f=0$.}
\end{figure}

\begin{figure}
	\centering
	\begin{subfigure}{0.4\textwidth}
		\includegraphics[width=1\textwidth]{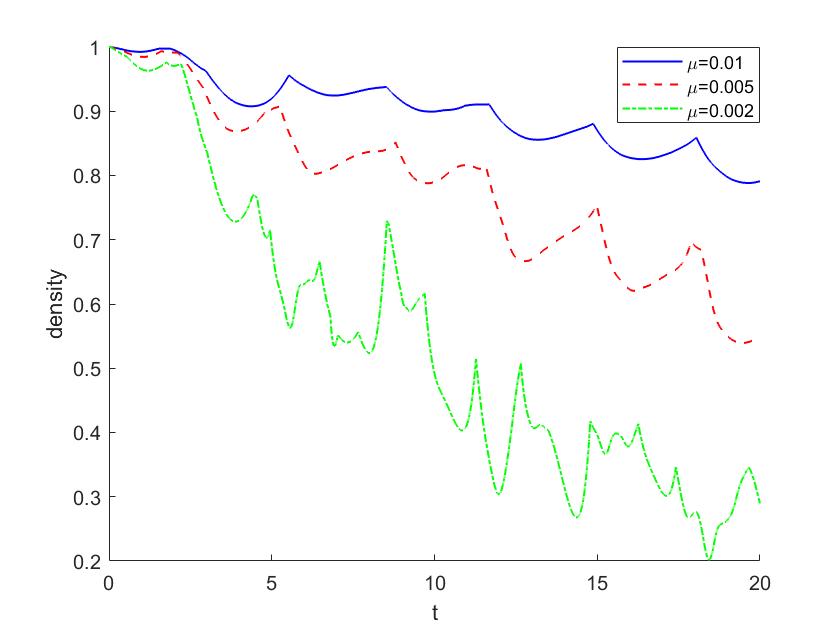}\quad
		\caption{Minimum of $\rho_h^{n+1}$}
	\end{subfigure}
	\begin{subfigure}{0.4\textwidth}
		\includegraphics[width=1\textwidth]{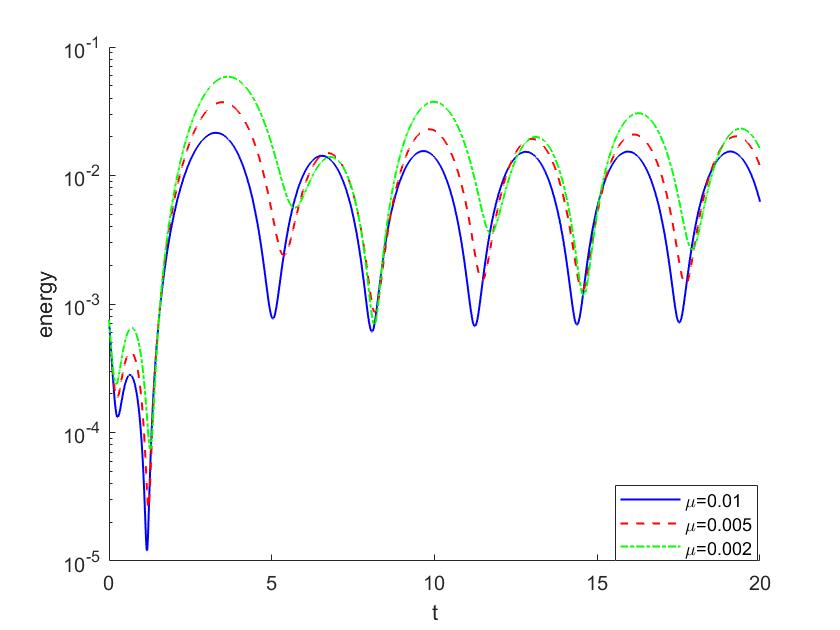}
		\caption{Evolution of energy $E_h^{n+1}$}
	\end{subfigure}
	\begin{subfigure}{0.4\textwidth}
		\includegraphics[width=1\textwidth]{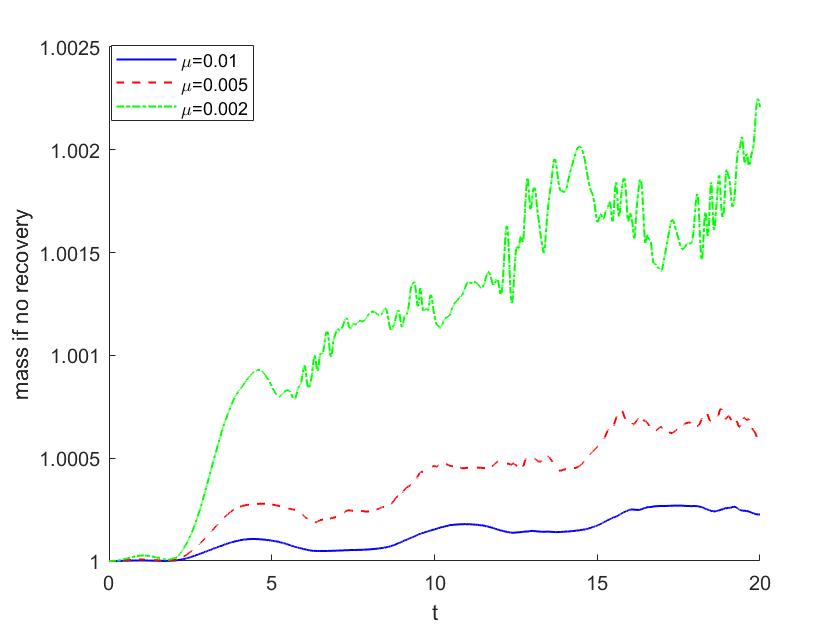}
		\caption{{Evolution of mass without recovery}}
	\end{subfigure}
	\begin{subfigure}{0.4\textwidth}
		\includegraphics[width=1\textwidth]{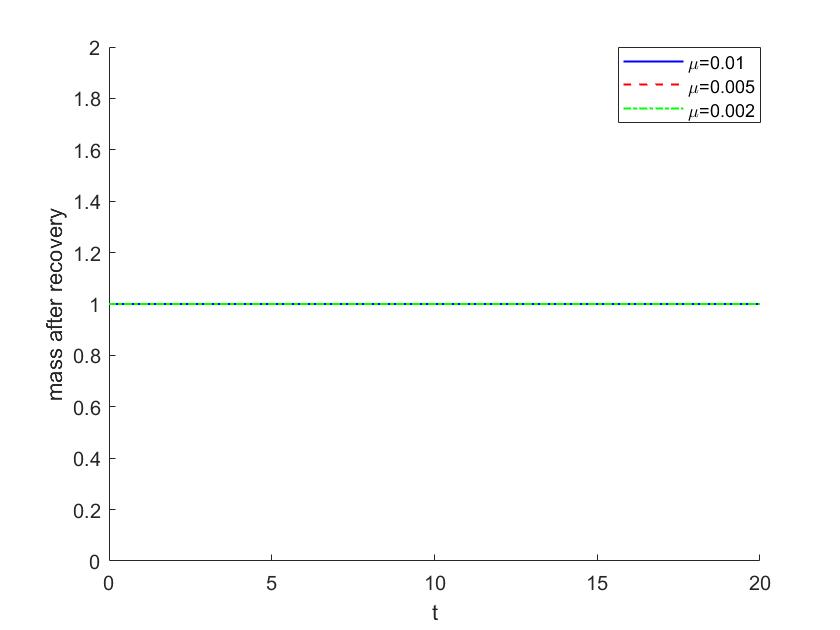}
		\caption{Evolution of mass with recovery}
	\end{subfigure}
	\begin{subfigure}{0.4\textwidth}
		\includegraphics[width=1\textwidth]{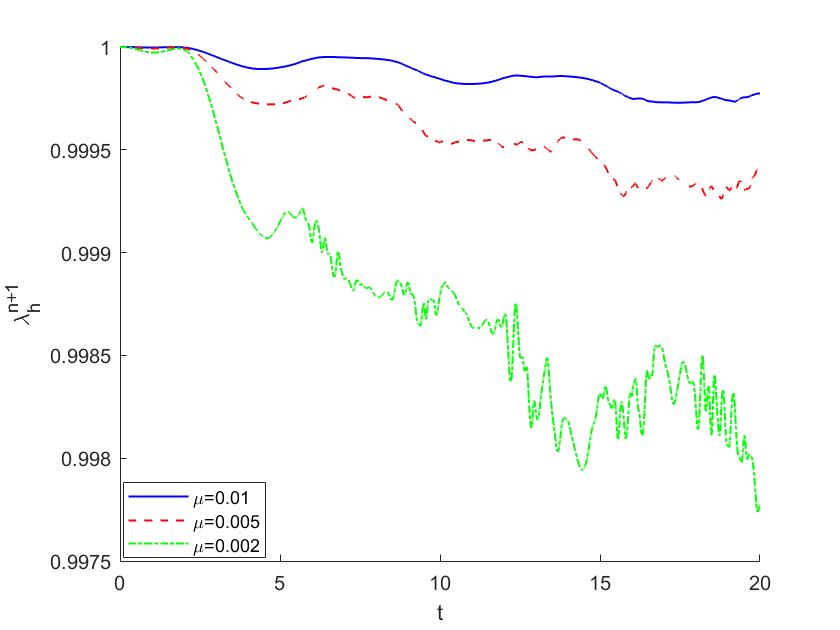}\quad
		\caption{{Evolution of $\lambda_h^{n+1}$}}
	\end{subfigure}
	\begin{subfigure}{0.4\textwidth}
		\includegraphics[width=1\textwidth]{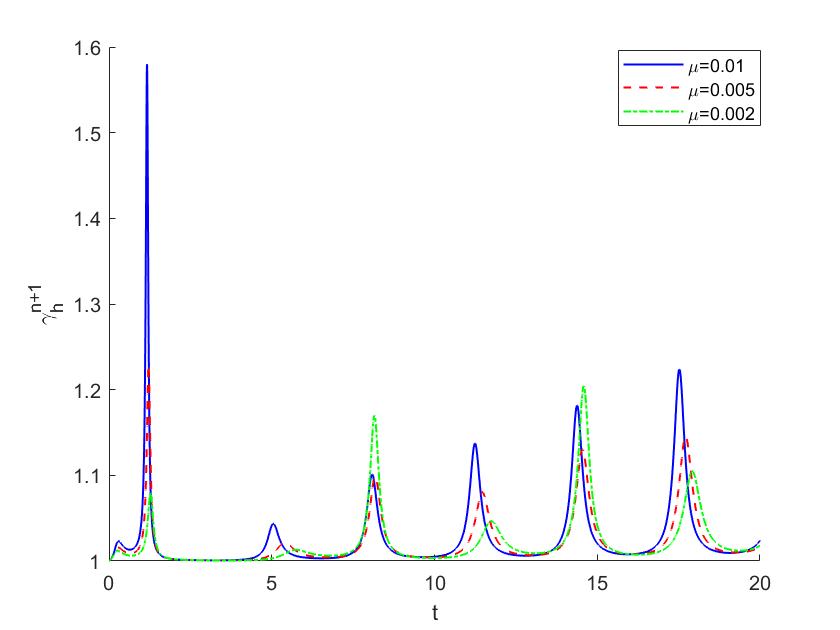}
		\caption{{Evolution of $\gamma_h^{n+1}$}}
	\end{subfigure}
	\begin{subfigure}{0.4\textwidth}
		\includegraphics[width=1\textwidth]{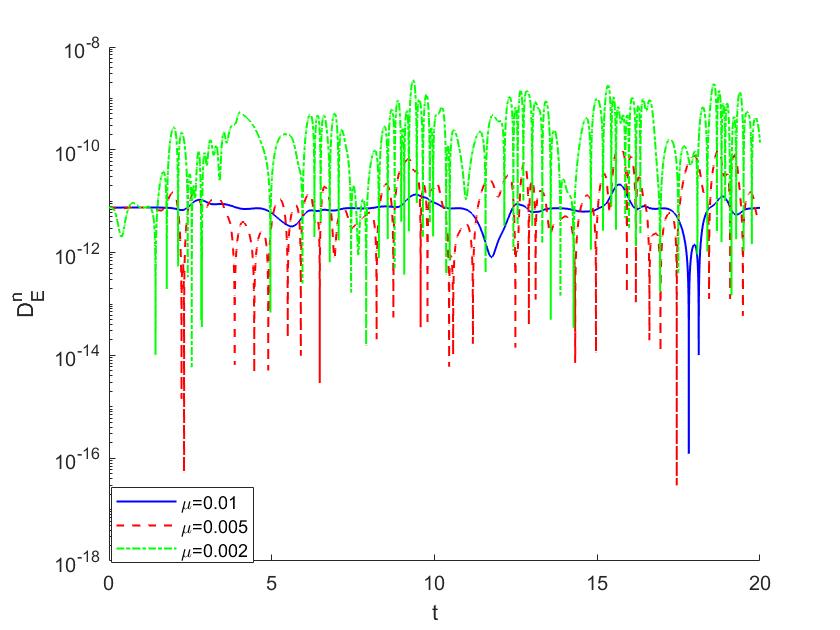}
		\caption{Evolution of $D_E^n$}
	\end{subfigure}
	\caption{Evolutions of the density, energy, {mass without recovery, mass  with recovery, $\lambda_h^{n+1}$, $\gamma_h^{n+1}$} and $D_E^n$ with $f\neq 0$.}
\end{figure}

\begin{figure}
	\centering\includegraphics[width=12cm,height=3cm]{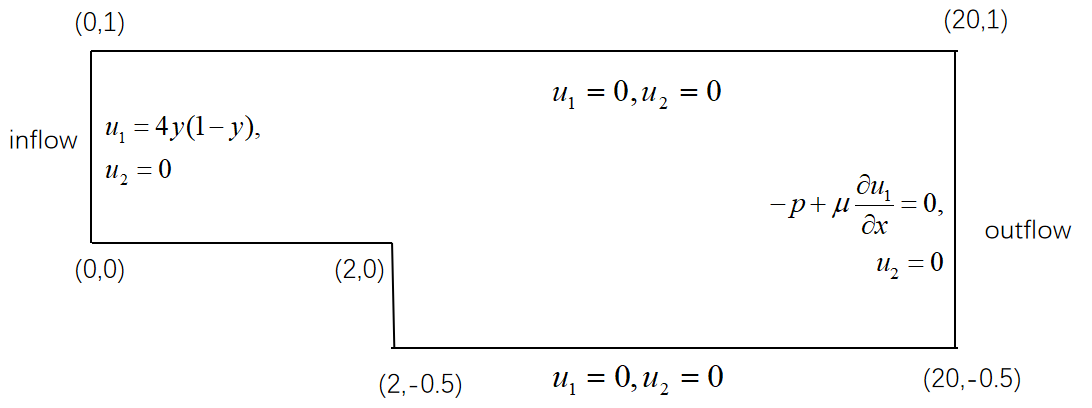}
	\caption{Analytical regions and boundary conditions.}
	\label{Analytical regions and boundary conditions}
	
	\begin{minipage}[c]{10cm}
		\includegraphics[width=10cm,height=1.95cm]{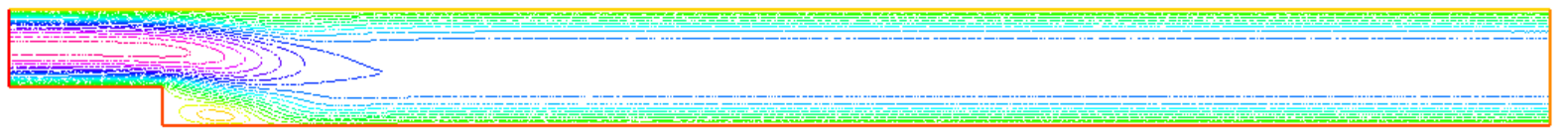}
	\end{minipage}
	\begin{minipage}[c]{10cm}
		\includegraphics[width=10cm,height=1.95cm]{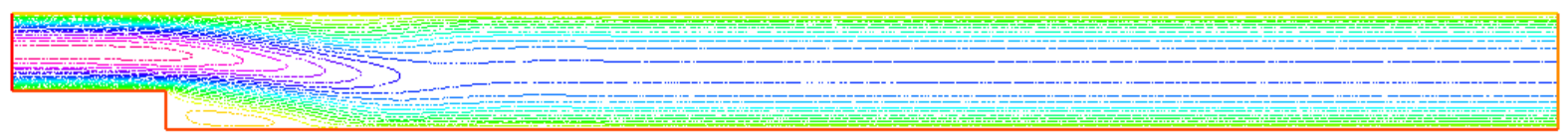}
	\end{minipage}
	\begin{minipage}[c]{10cm}
		\includegraphics[width=10cm,height=1.95cm]{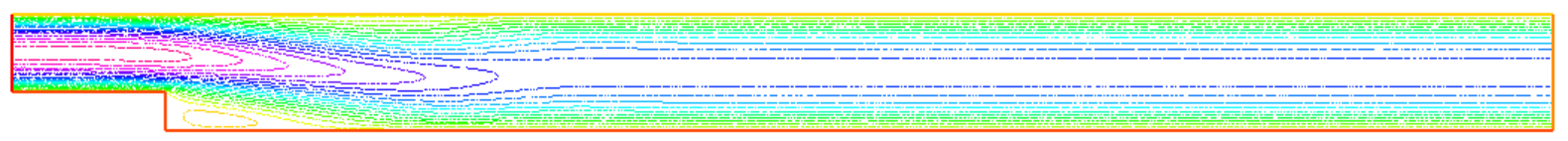}
	\end{minipage}
	\caption{Velocity $u_{1h}^{n}$ of the back-step flow  at $t=3$ (top),  $t=5$ (middle), $t=7$ (bottom).}
	
	\begin{minipage}[c]{10cm}
		\includegraphics[width=10cm,height=1.95cm]{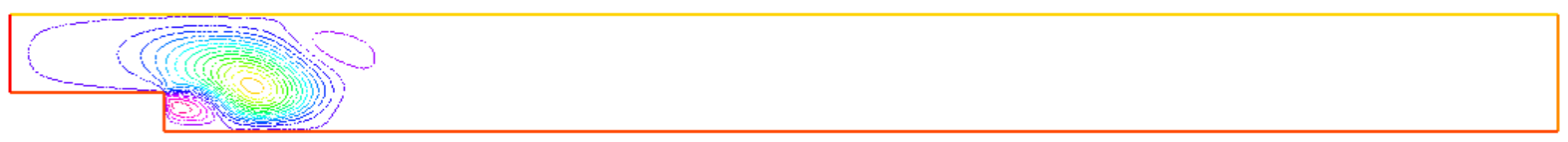}
	\end{minipage}
	\begin{minipage}[c]{10cm}
		\includegraphics[width=10cm,height=1.95cm]{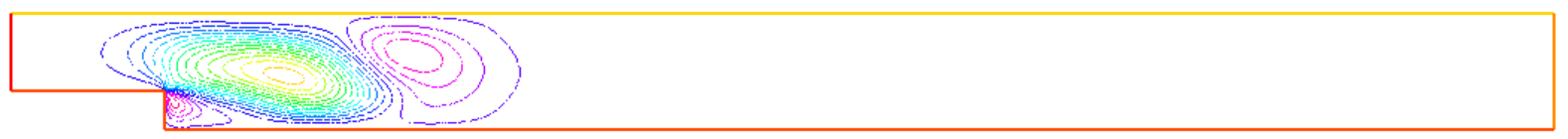}
	\end{minipage}
	\begin{minipage}[c]{10cm}
		\includegraphics[width=10cm,height=1.95cm]{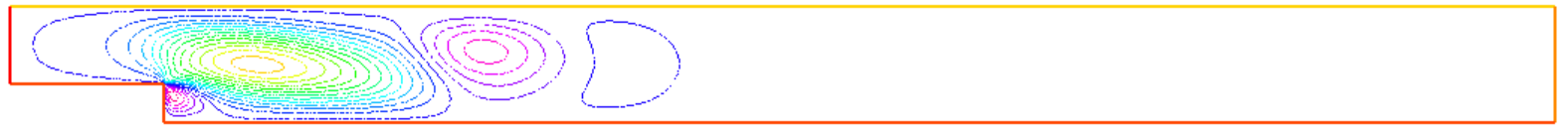}
	\end{minipage}
	\caption{Velocity $u_{2h}^n$ of the back-step flow  at $t=3$ (top),  $t=5$ (middle),  $t=7$ (bottom).}
	
	\begin{minipage}[c]{10cm}
		\includegraphics[width=10cm,height=1.95cm]{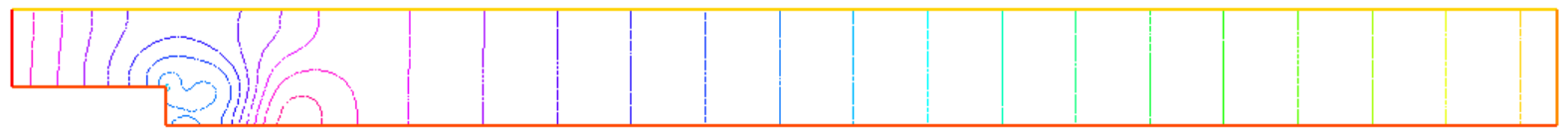}
	\end{minipage}
	\begin{minipage}[c]{10cm}
		\includegraphics[width=10cm,height=1.95cm]{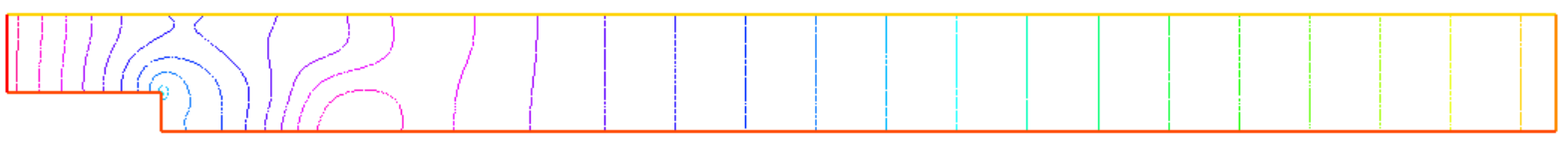}
	\end{minipage}
	\begin{minipage}[c]{10cm}
		\includegraphics[width=10cm,height=1.95cm]{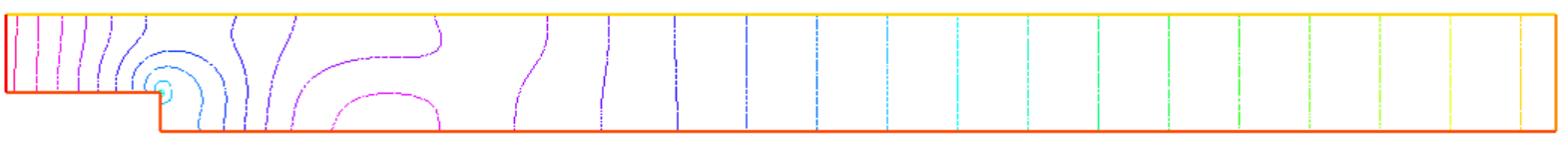}
	\end{minipage}
	\caption{Pressure $p_h^n$ of the back-step flow  at $t=3$ (top),  $t=5$ (middle),  $t=7$ (bottom).}
\end{figure}

\begin{figure}
	\centering
	\begin{minipage}[c]{10cm}
		\includegraphics[width=10cm,height=1.95cm]{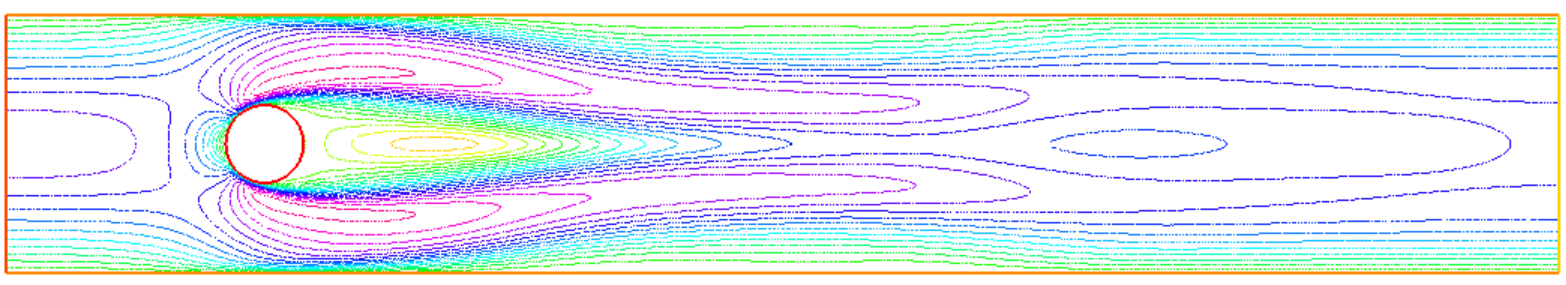}
	\end{minipage}
	\begin{minipage}[c]{10cm}
		\includegraphics[width=10cm,height=1.95cm]{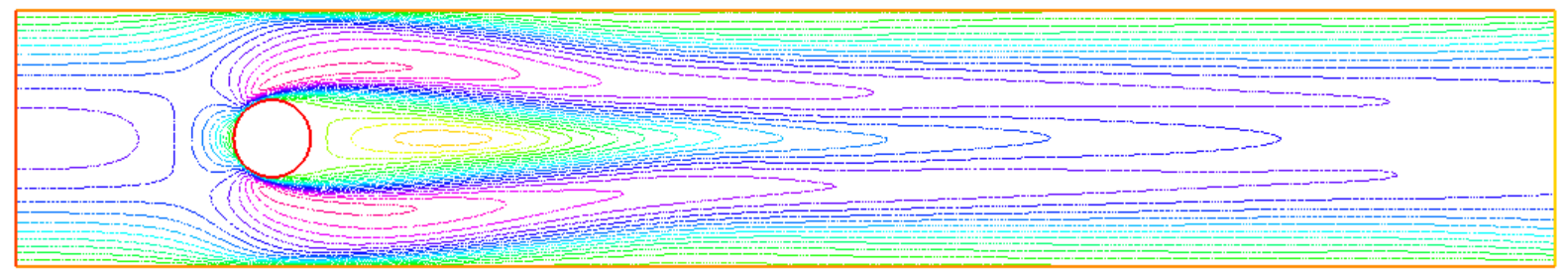}
	\end{minipage}
	\begin{minipage}[c]{10cm}
		\includegraphics[width=10cm,height=1.95cm]{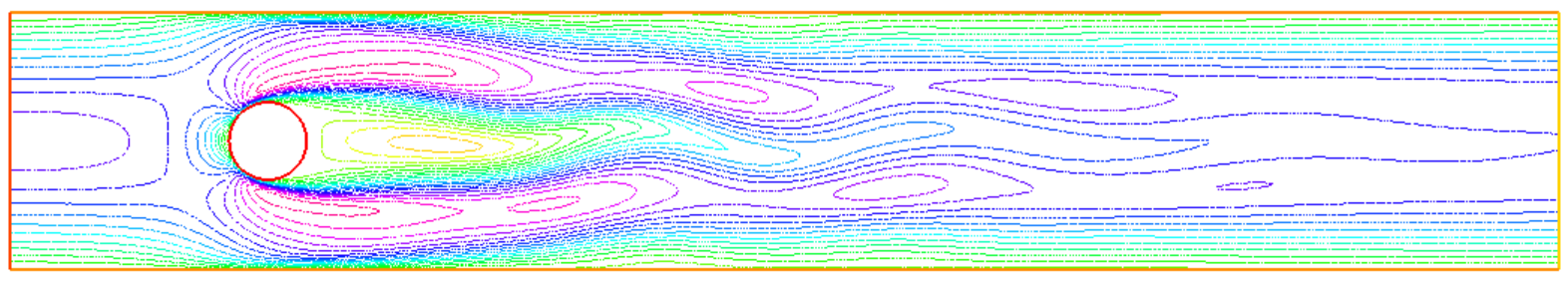}
	\end{minipage}
	\caption{Velocity $u_{1h}^n$ of the cylinder flow {with a constant density} at $t=3$ (top),  $t=5$ (middle),  $t=7$ (bottom).}
	
	\begin{minipage}[c]{10cm}
		\includegraphics[width=10cm,height=1.95cm]{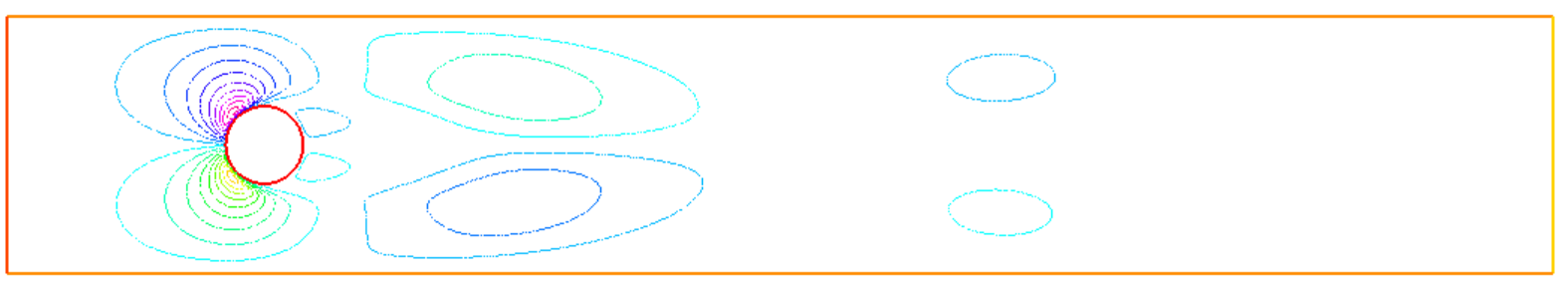}
	\end{minipage}
	\begin{minipage}[c]{10cm}
		\includegraphics[width=10cm,height=1.95cm]{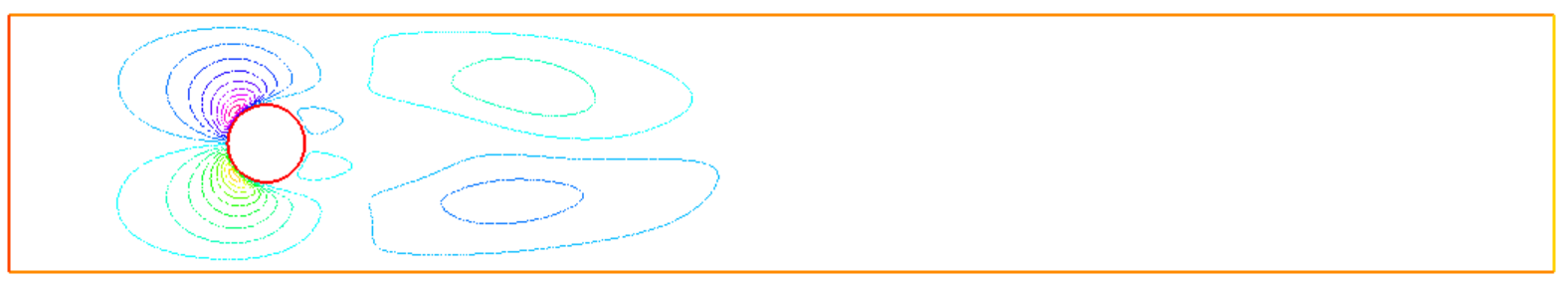}
	\end{minipage}
	\begin{minipage}[c]{10cm}
		\includegraphics[width=10cm,height=1.95cm]{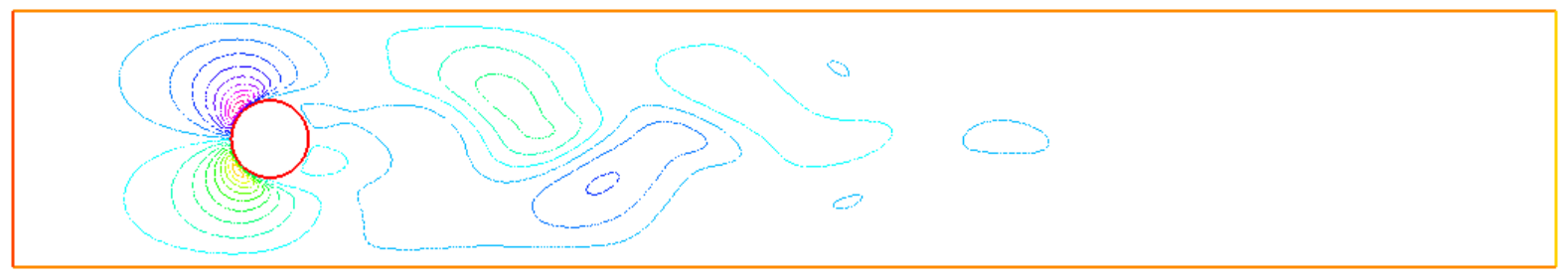}
	\end{minipage}
	\caption{Velocity $u_{2h}^n$ of the cylinder flow  {with a constant density} at $t=3$ (top),  $t=5$ (middle),  $t=7$ (bottom).}
	
	\begin{minipage}[c]{10cm}
		\includegraphics[width=10cm,height=1.95cm]{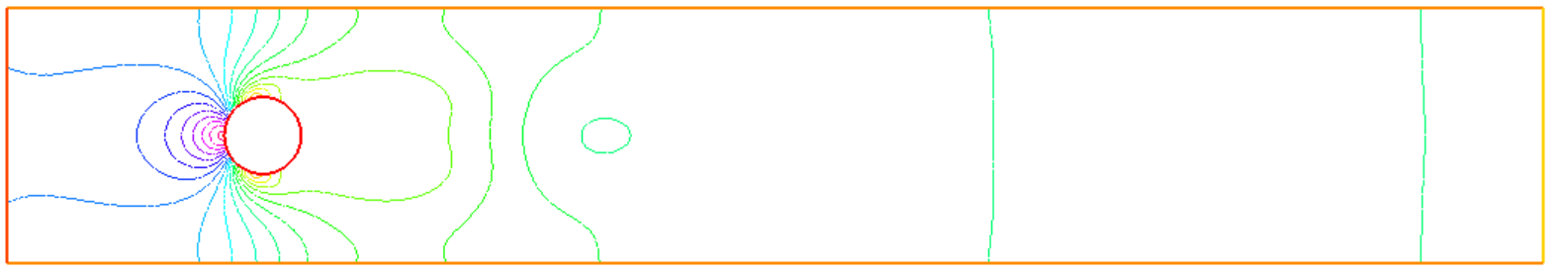}
	\end{minipage}
	\begin{minipage}[c]{10cm}
		\includegraphics[width=10cm,height=1.95cm]{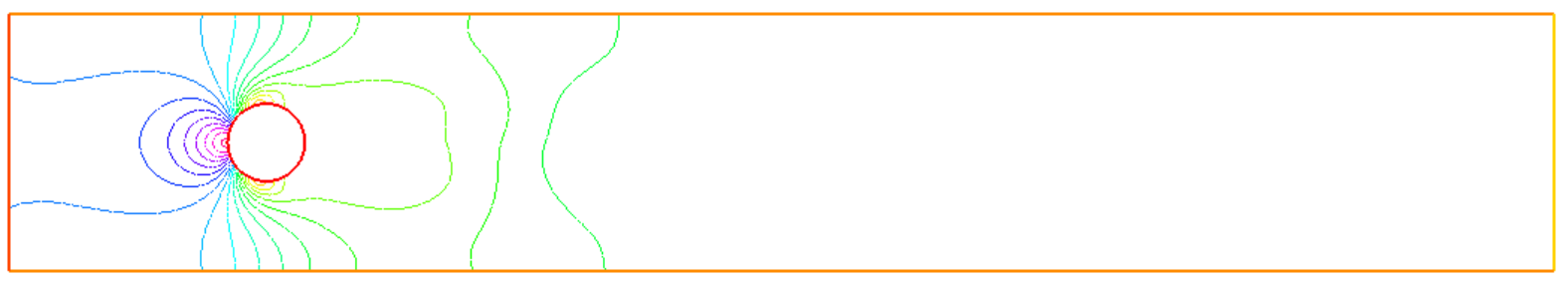}
	\end{minipage}
	\begin{minipage}[c]{10cm}
		\includegraphics[width=10cm,height=1.95cm]{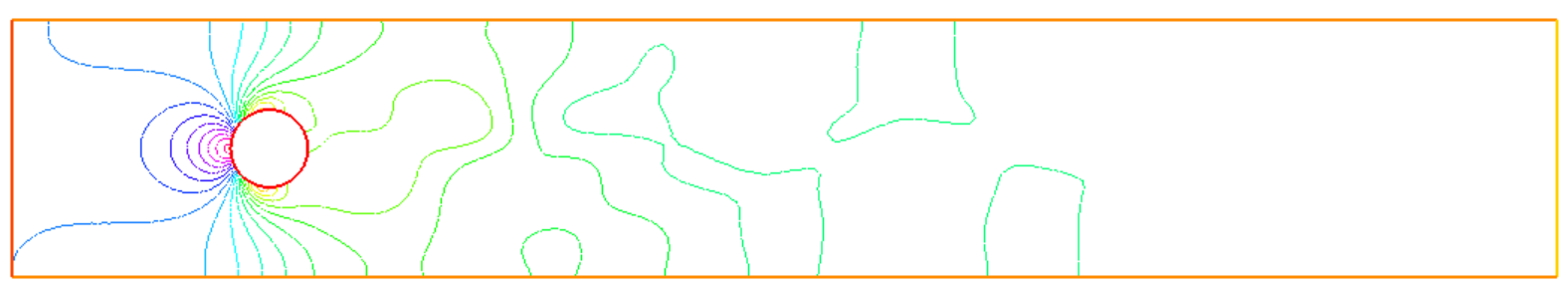}
	\end{minipage}
	\caption{Pressure $p_h^n$ of the cylinder flow {with a constant density} at $t=3$ (top),  $t=5$ (middle),  $t=7$ (bottom).}
\end{figure}

\begin{figure}
	\centering
	\begin{minipage}[c]{10cm}
		\vspace{-5cm}\includegraphics[width=12cm,height=14cm]{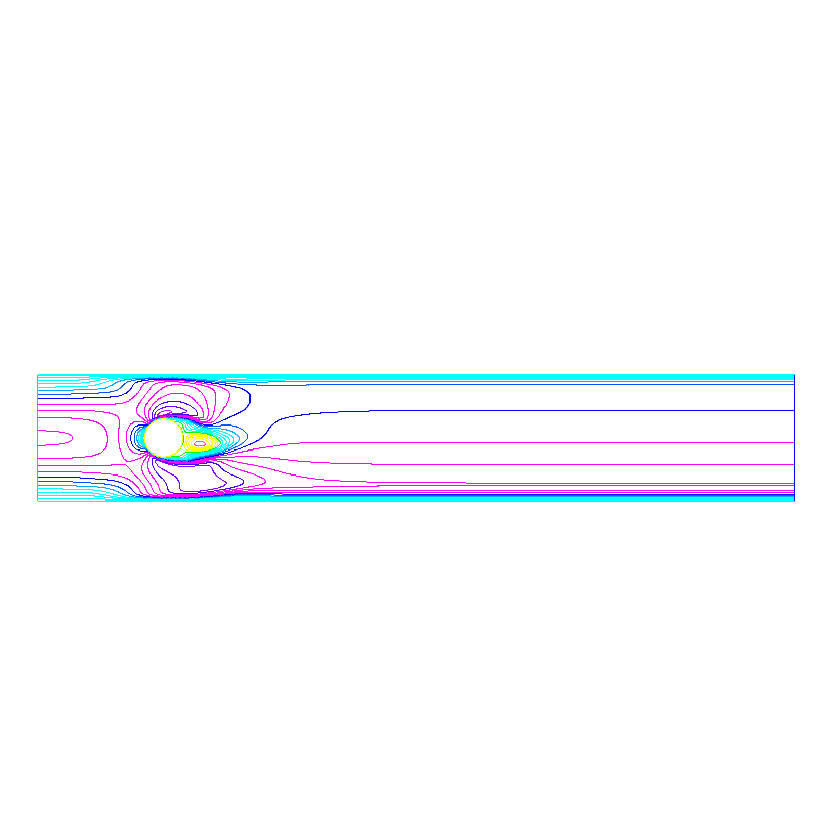}
	\end{minipage}
	\begin{minipage}[c]{10cm}
		\vspace{-11cm}\includegraphics[width=12cm,height=14cm]{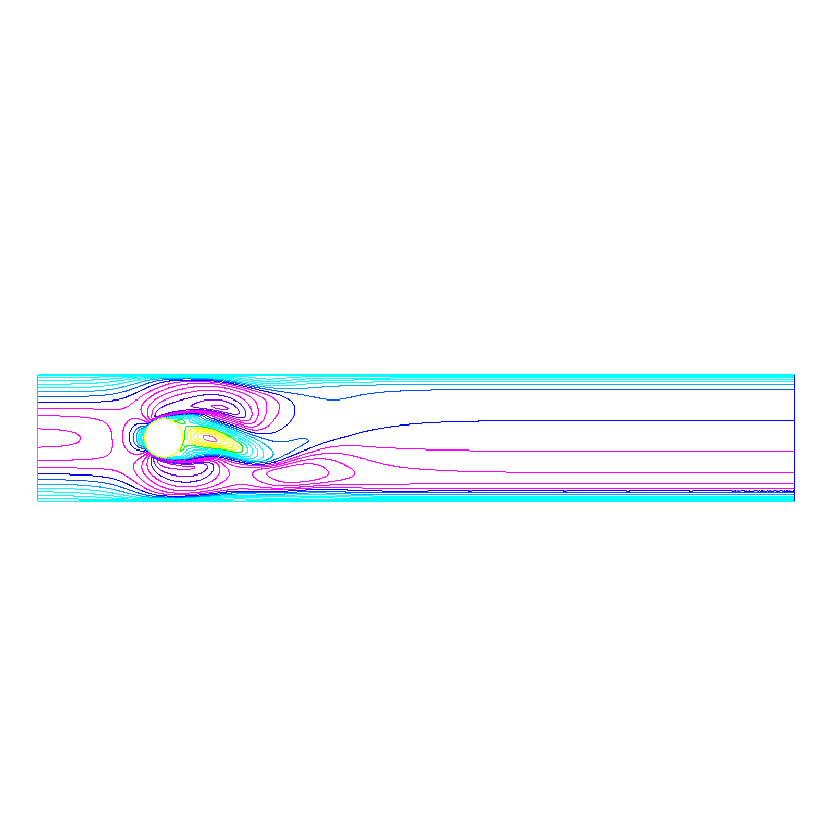}
	\end{minipage}
	\begin{minipage}[c]{10cm}
		\vspace{-11cm}\includegraphics[width=12cm,height=14cm]{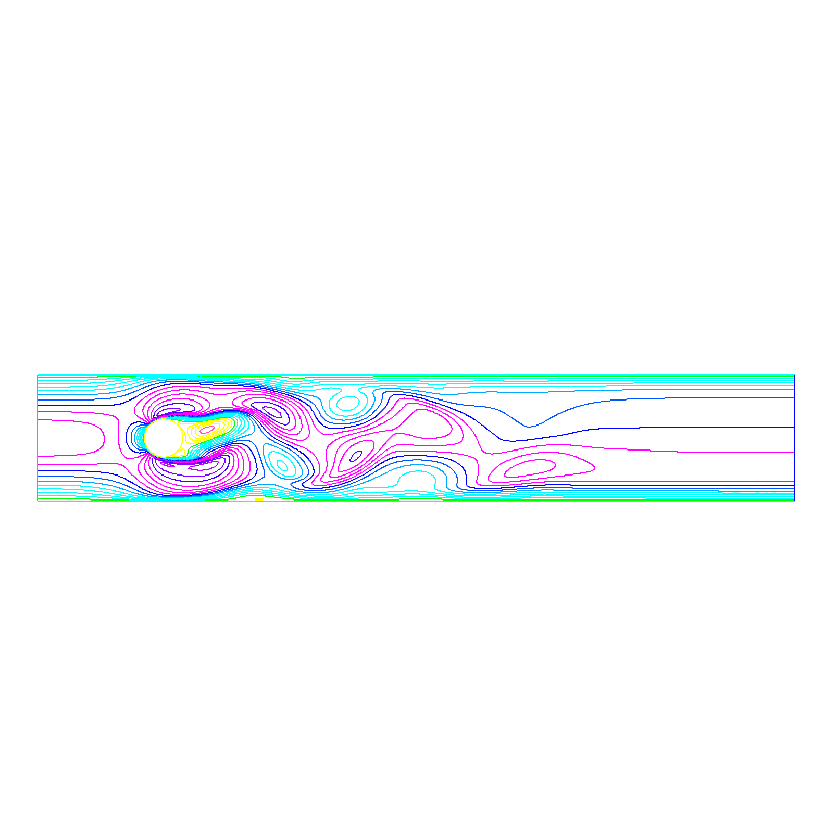}
	\end{minipage}
	\begin{minipage}[c]{10cm}
		\vspace{-11cm}\includegraphics[width=12cm,height=14cm]{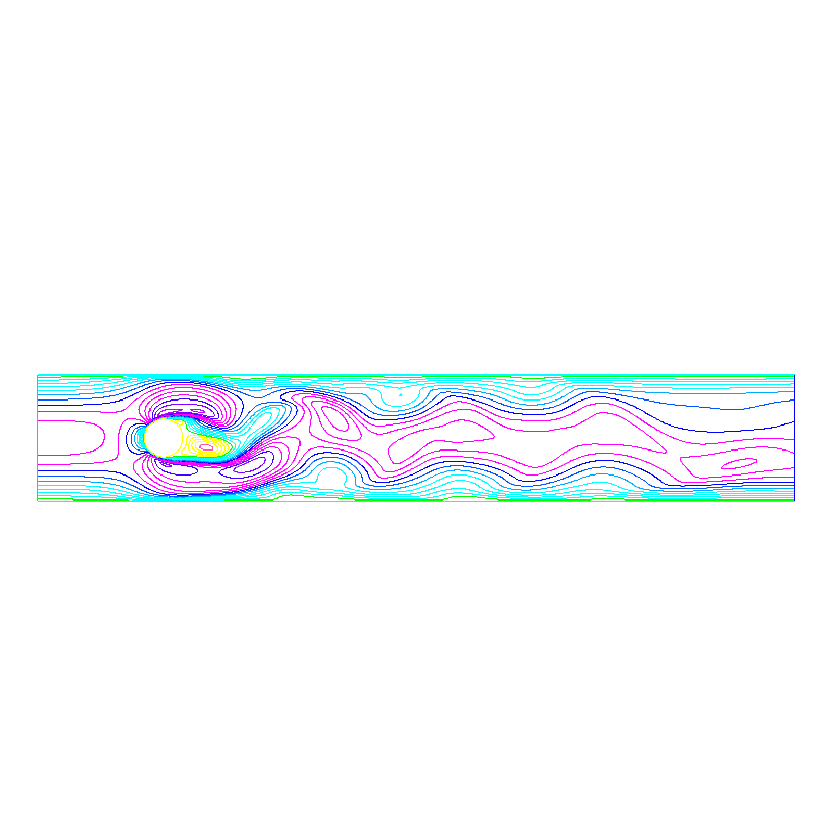}
	\end{minipage}
	\begin{minipage}[c]{10cm}
		\vspace{-11cm}\includegraphics[width=12cm,height=14cm]{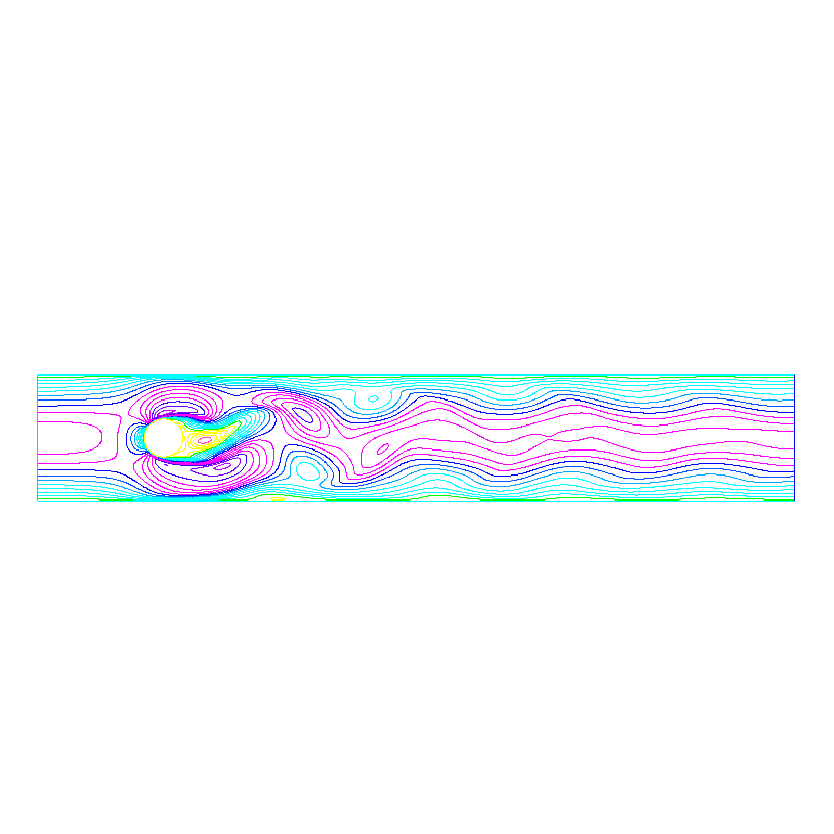}
	\end{minipage}
	\begin{minipage}[c]{10cm}
		\vspace{-11cm}\includegraphics[width=12cm,height=14cm]{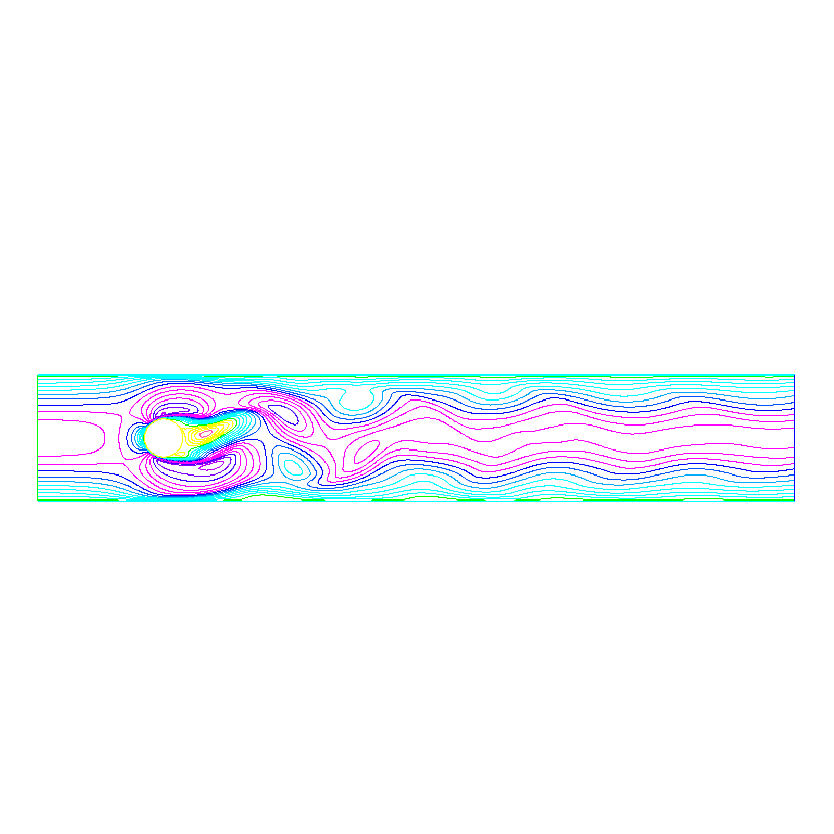}
	\end{minipage}
	\vspace{-5cm}\caption{{Velocity $u_{1h}^n$ of the cylinder flow with a variable density at $t=0.5,~1,~2,~3,~5$ and $7$ (from top to bottom).}}
\end{figure}

\begin{figure}
	\centering
	\begin{minipage}[c]{10cm}
		\vspace{-5cm}\includegraphics[width=12cm,height=14cm]{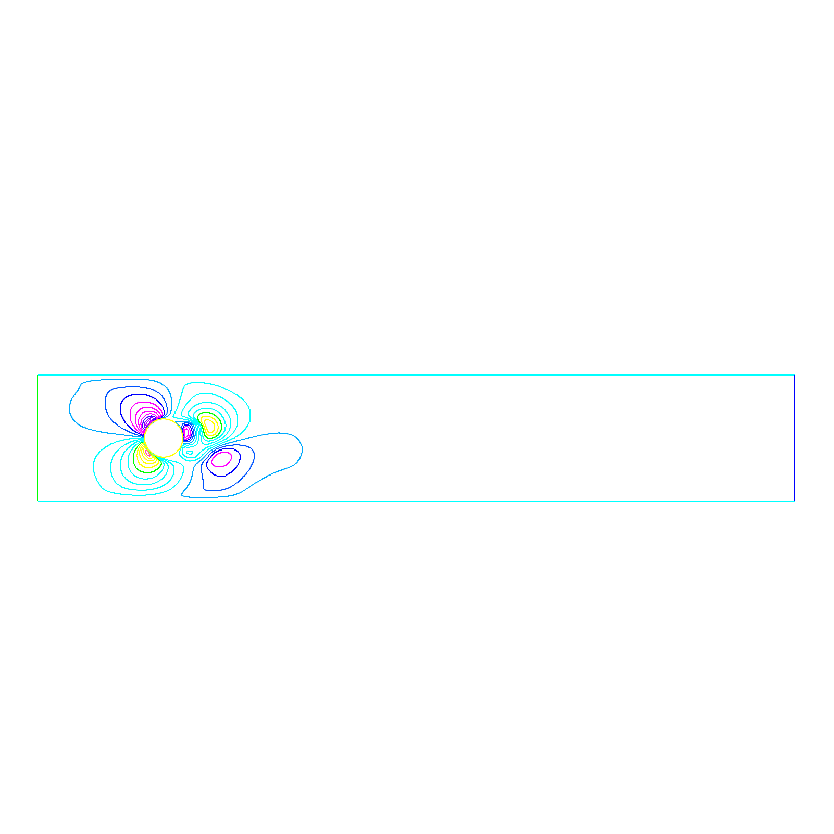}
	\end{minipage}
	\begin{minipage}[c]{10cm}
		\vspace{-11cm}\includegraphics[width=12cm,height=14cm]{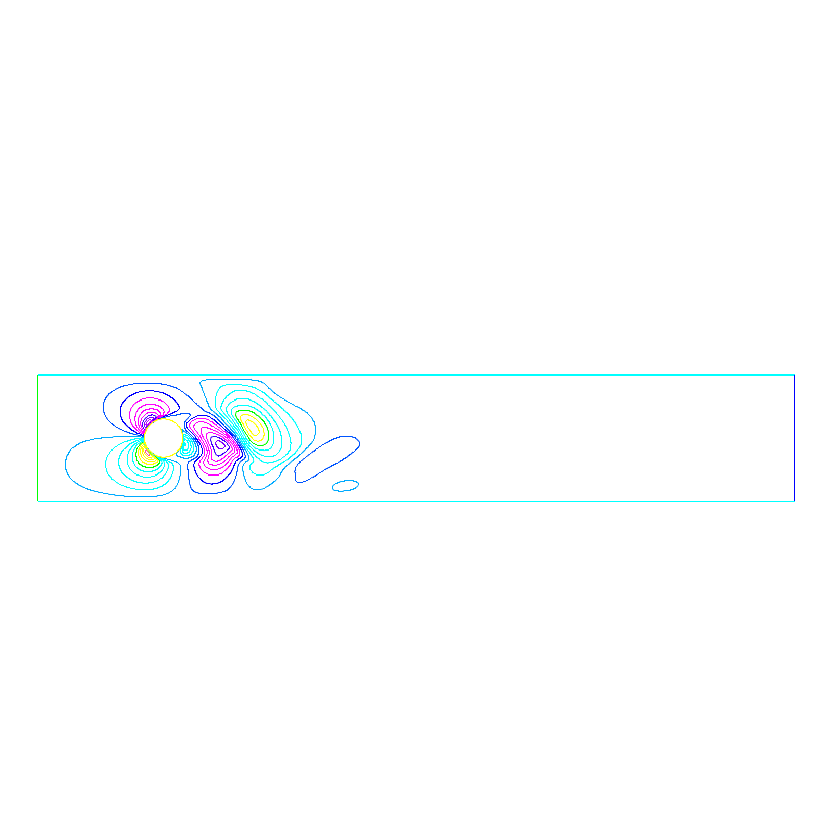}
	\end{minipage}
	\begin{minipage}[c]{10cm}
		\vspace{-11cm}\includegraphics[width=12cm,height=14cm]{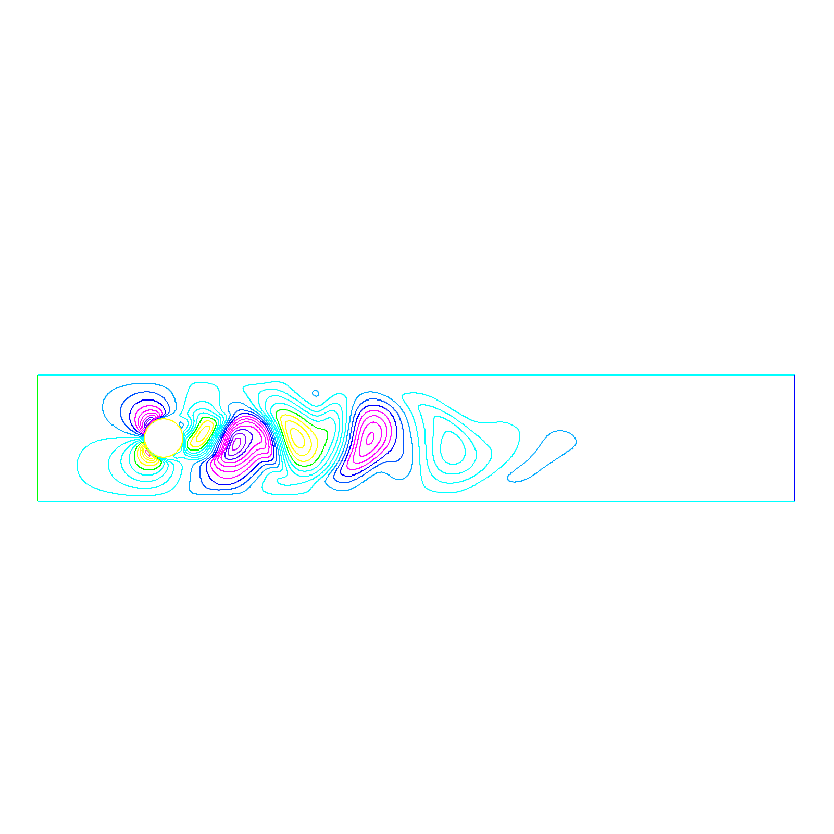}
	\end{minipage}
	\begin{minipage}[c]{10cm}
		\vspace{-11cm}\includegraphics[width=12cm,height=14cm]{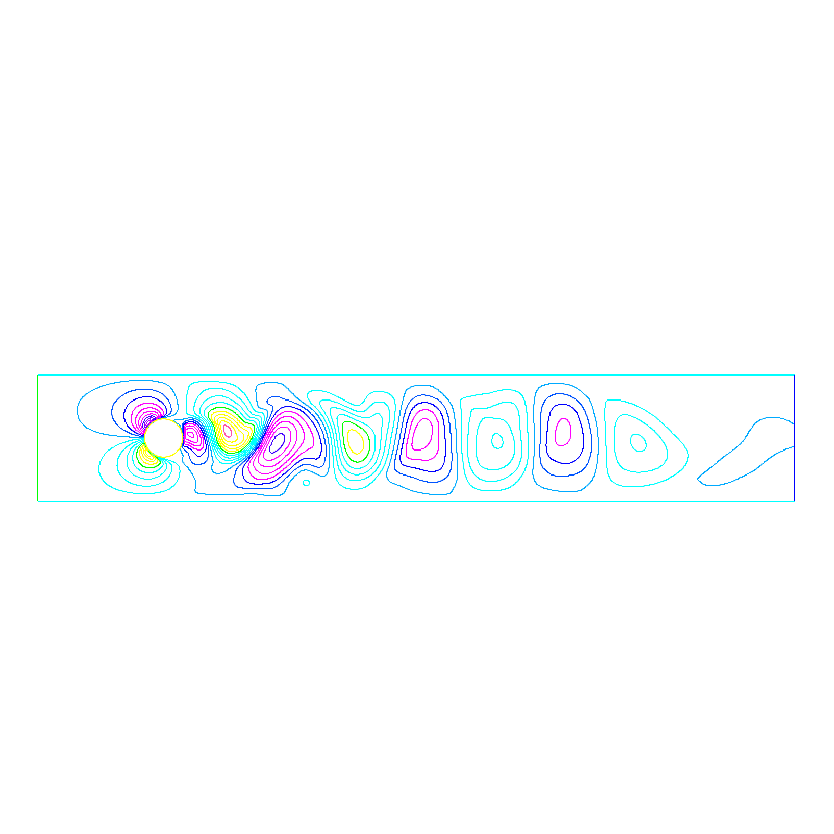}
	\end{minipage}
	\begin{minipage}[c]{10cm}
		\vspace{-11cm}\includegraphics[width=12cm,height=14cm]{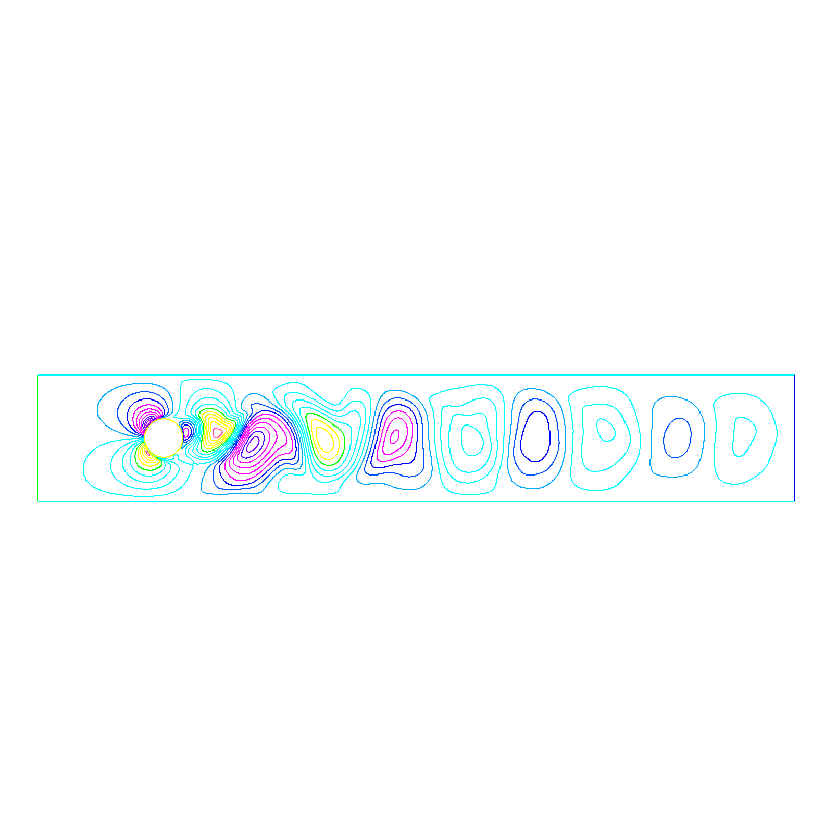}
	\end{minipage}
	\begin{minipage}[c]{10cm}
		\vspace{-11cm}\includegraphics[width=12cm,height=14cm]{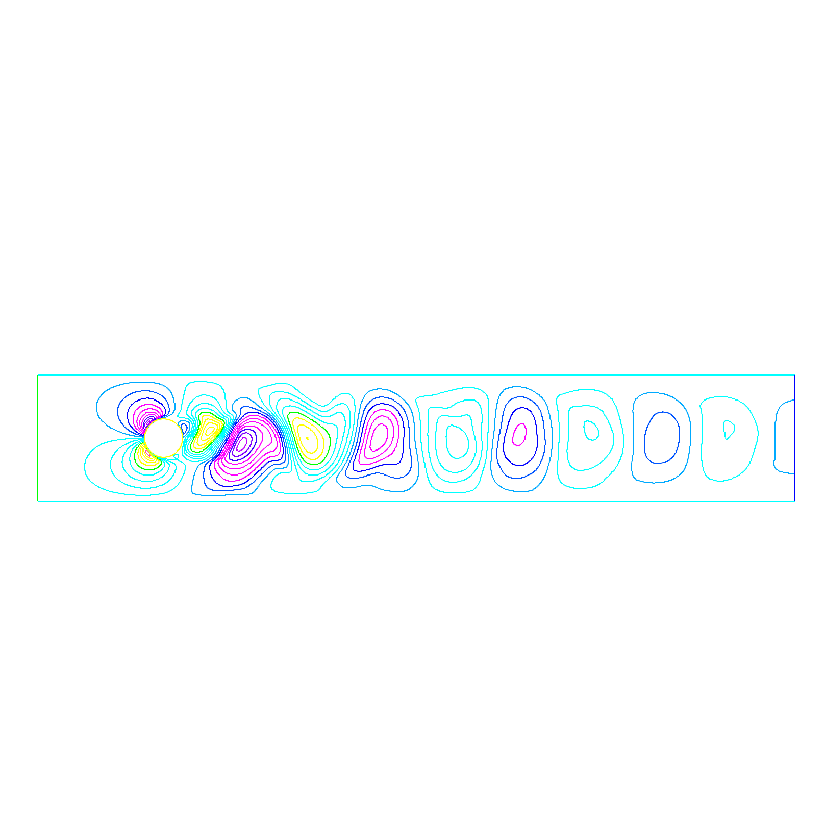}
	\end{minipage}
	\vspace{-5cm}\caption{{Velocity $u_{2h}^n$ of the cylinder flow with a variable density at $t=0.5,~1,~2,~3,~5$ and $7$ (from top to bottom).}}
\end{figure}

\begin{figure}
	\centering
	\begin{minipage}[c]{10cm}
		\vspace{-5cm}\includegraphics[width=12cm,height=14cm]{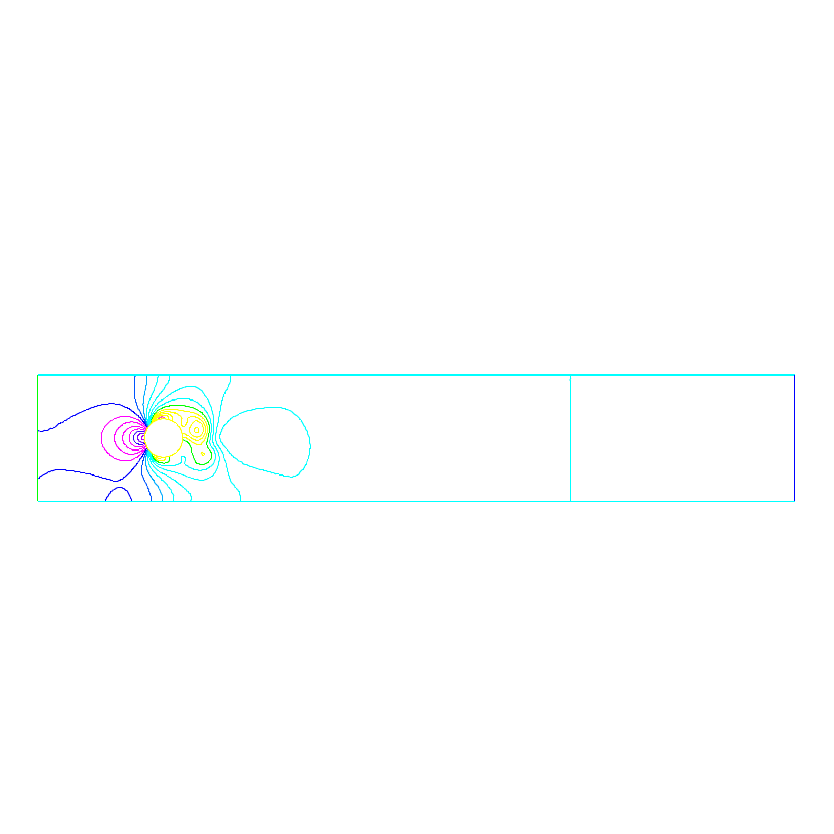}
	\end{minipage}
	\begin{minipage}[c]{10cm}
		\vspace{-11cm}\includegraphics[width=12cm,height=14cm]{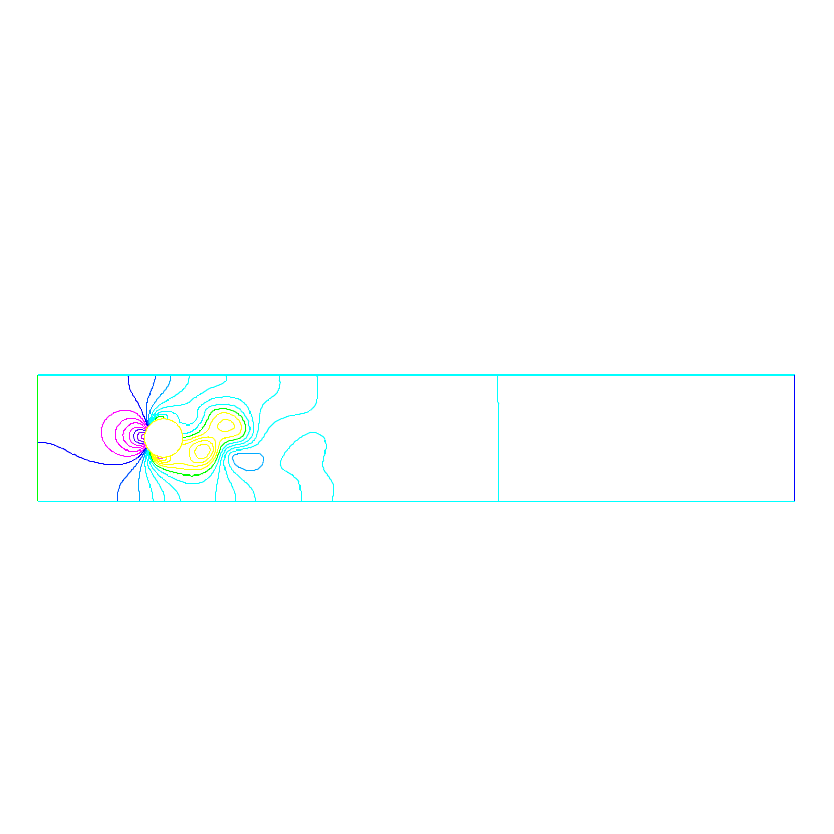}
	\end{minipage}
	\begin{minipage}[c]{10cm}
		\vspace{-11cm}\includegraphics[width=12cm,height=14cm]{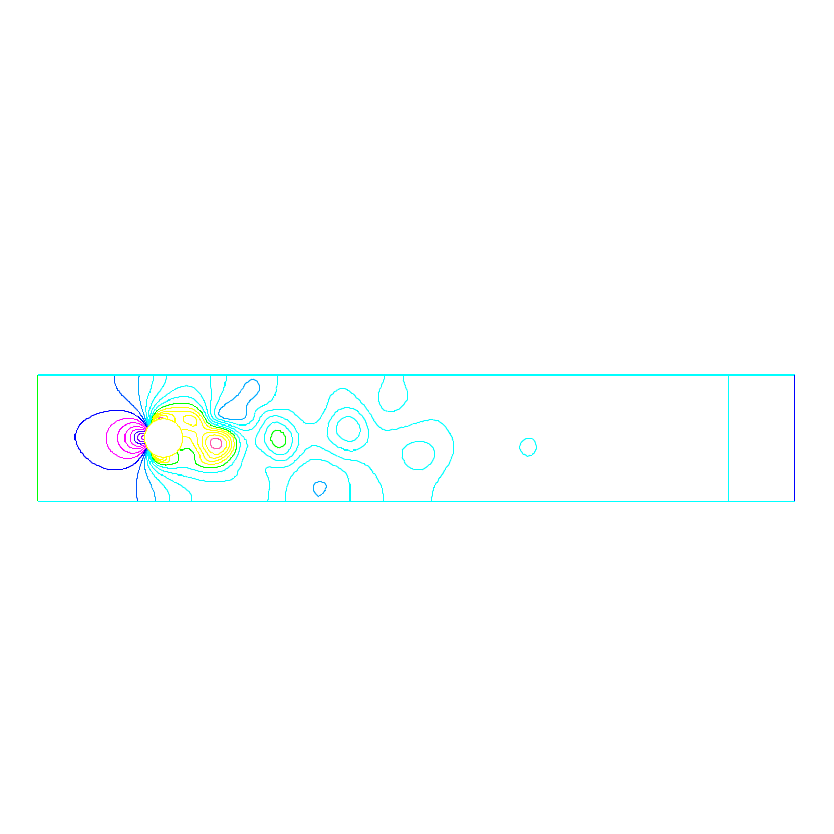}
	\end{minipage}
	\begin{minipage}[c]{10cm}
		\vspace{-11cm}\includegraphics[width=12cm,height=14cm]{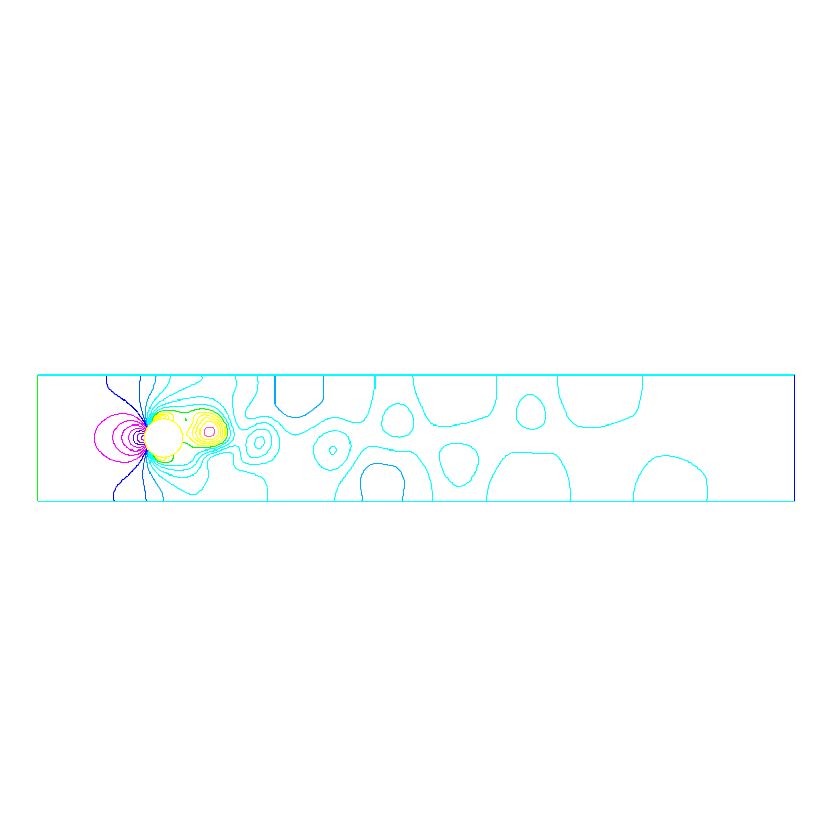}
	\end{minipage}
	\begin{minipage}[c]{10cm}
		\vspace{-11cm}\includegraphics[width=12cm,height=14cm]{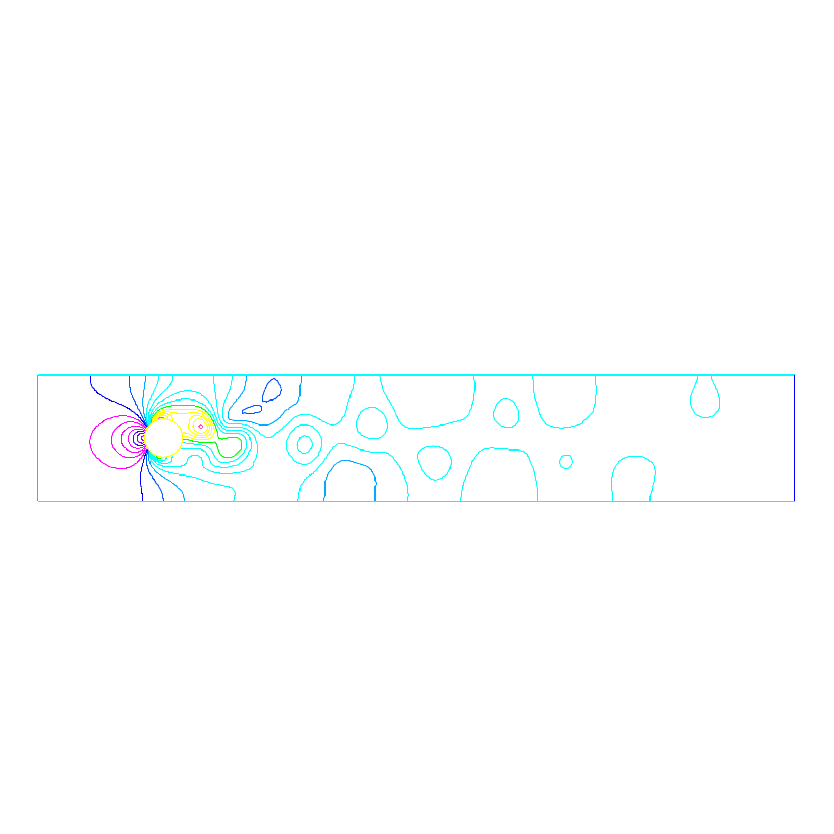}
	\end{minipage}
	\begin{minipage}[c]{10cm}
		\vspace{-11cm}\includegraphics[width=12cm,height=14cm]{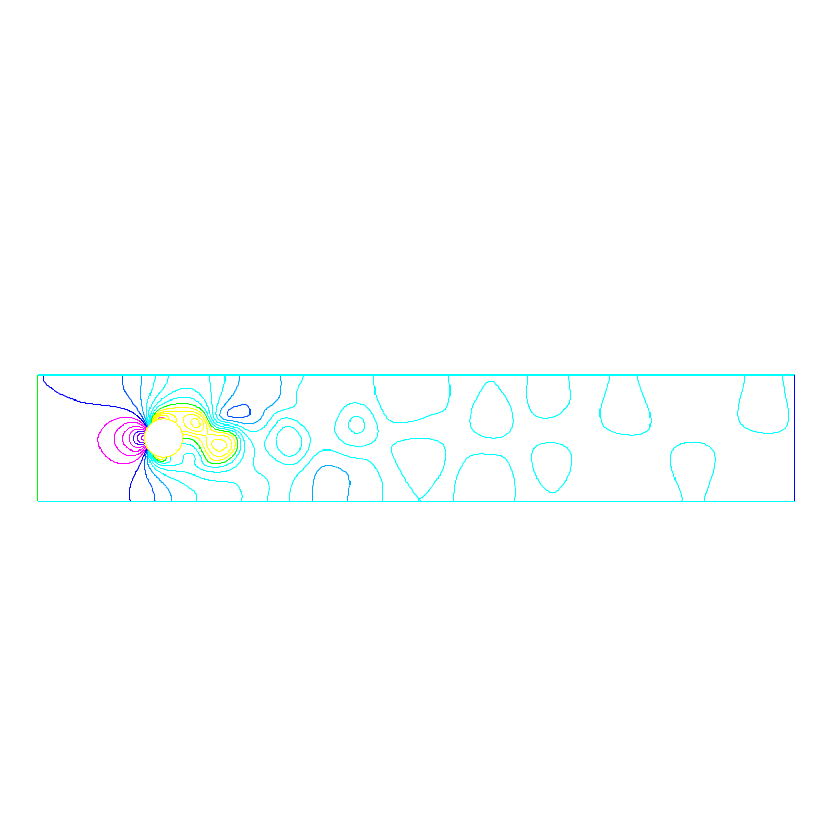}
	\end{minipage}
	\vspace{-5cm}\caption{{Pressure $p_{h}^n$ of the cylinder flow with a variable density at $t=0.5,~1,~2,~3,~5$ and $7$ (from top to bottom).}}
\end{figure}

\begin{figure}
	\centering
	\begin{minipage}[c]{10cm}
		\vspace{-5cm}\includegraphics[width=12cm,height=14cm]{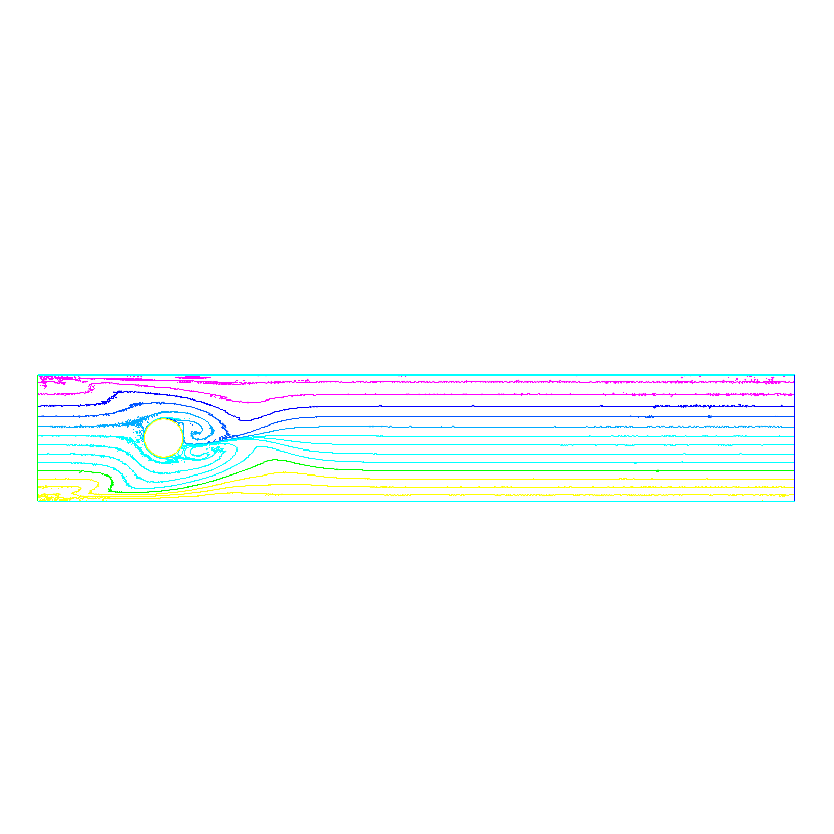}
	\end{minipage}
	\begin{minipage}[c]{10cm}
		\vspace{-11cm}\includegraphics[width=12cm,height=14cm]{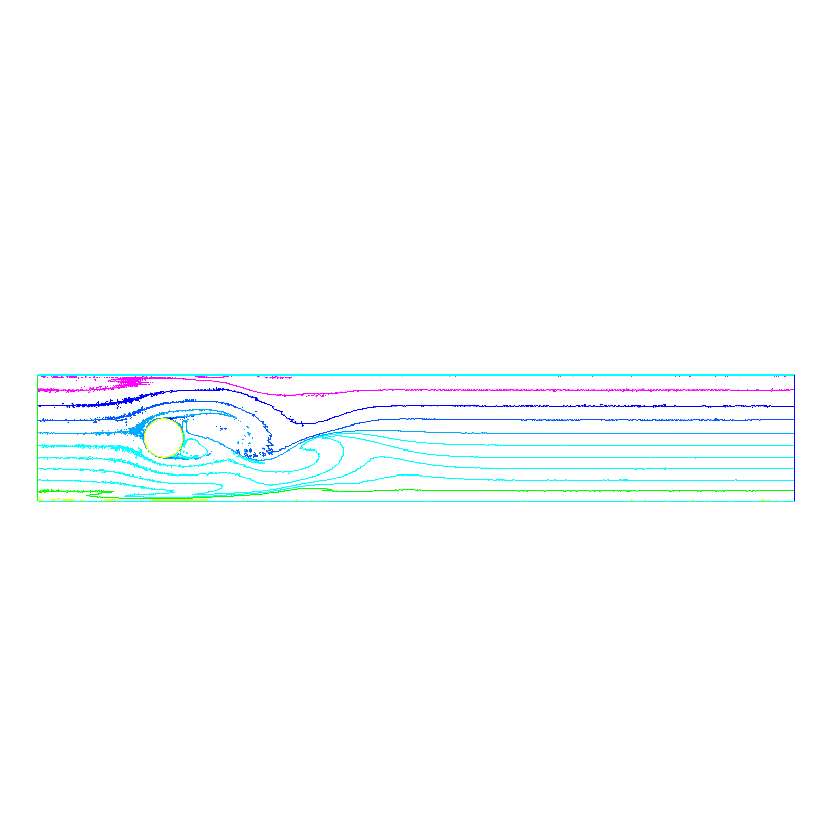}
	\end{minipage}
	\begin{minipage}[c]{10cm}
		\vspace{-11cm}\includegraphics[width=12cm,height=14cm]{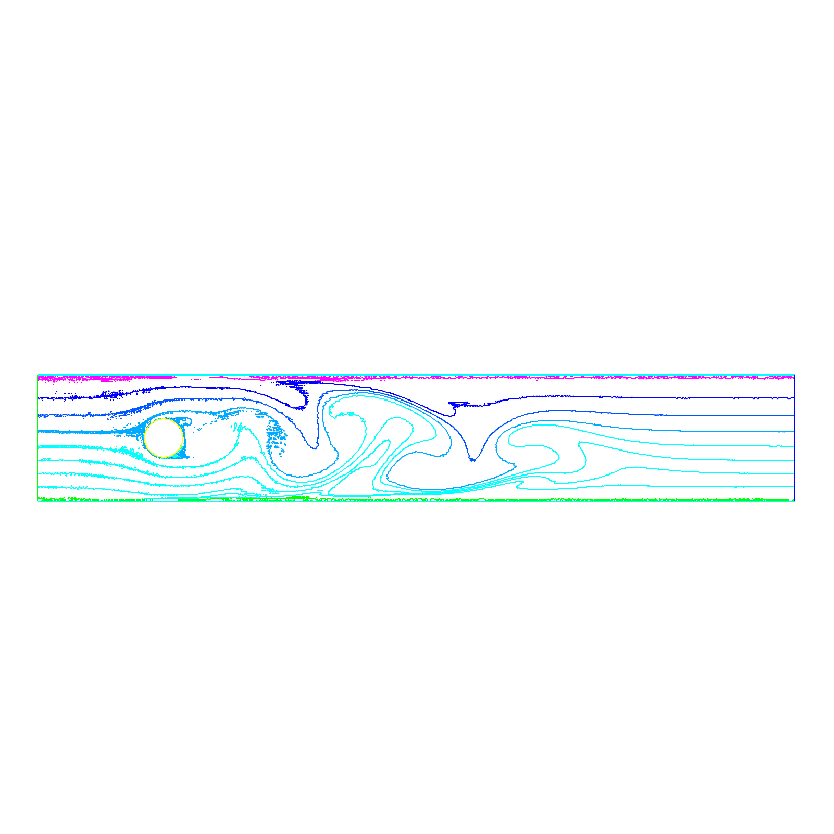}
	\end{minipage}
	\begin{minipage}[c]{10cm}
		\vspace{-11cm}\includegraphics[width=12cm,height=14cm]{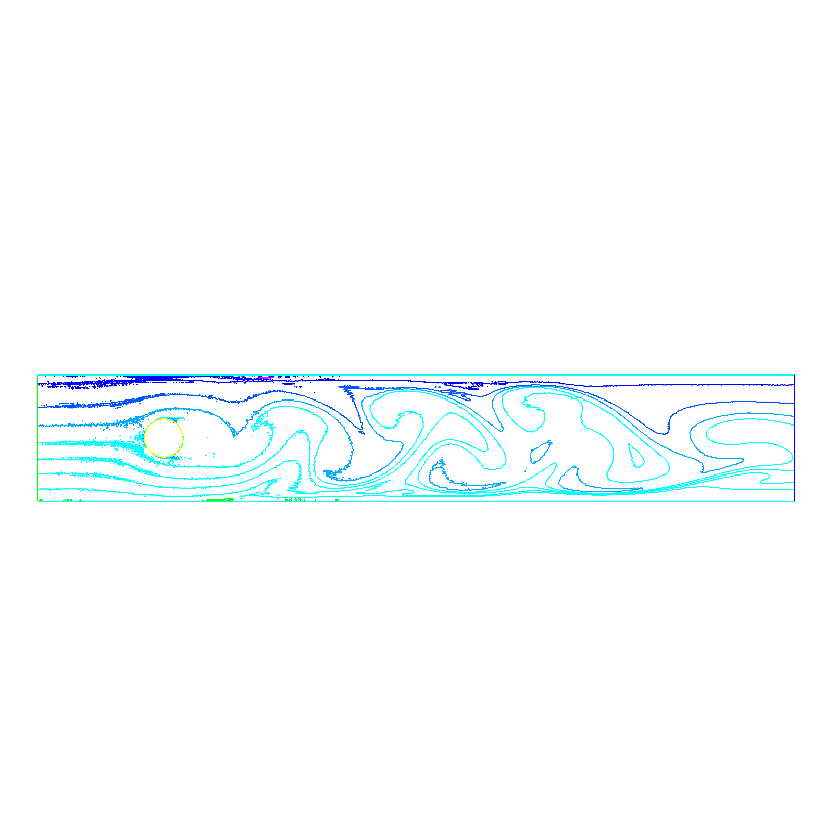}
	\end{minipage}
	\begin{minipage}[c]{10cm}
		\vspace{-11cm}\includegraphics[width=12cm,height=14cm]{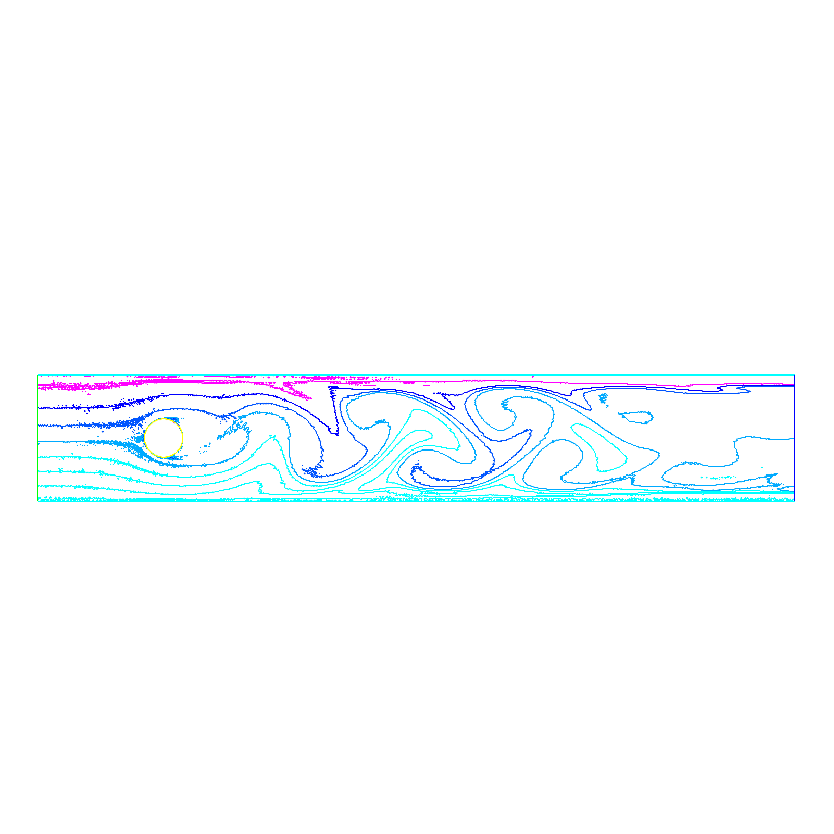}
	\end{minipage}
	\begin{minipage}[c]{10cm}
		\vspace{-11cm}\includegraphics[width=12cm,height=14cm]{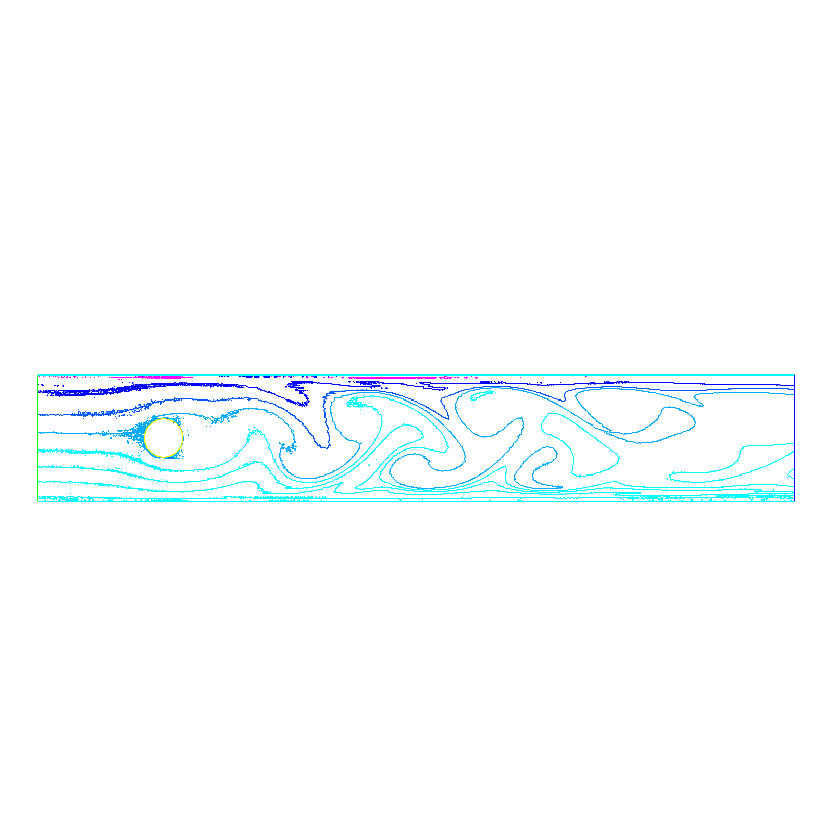}
	\end{minipage}
	\vspace{-5cm}\caption{{Density $\rho_{h}^n$ of the cylinder flow with a variable density at $t=0.5,~1,~2,~3,~5$ and $7$ (from top to bottom).}}
\end{figure}

\end{document}